\documentclass{memo-l}

\usepackage{xcolor}
\usepackage{amsxtra}
\usepackage{mathrsfs}
\usepackage{yfonts}
\usepackage{amsmath}
\usepackage{amsfonts}
\usepackage{amssymb}
\usepackage{amscd,gensymb}
\usepackage{mathtools}
\usepackage{graphicx}
\usepackage[percent]{overpic}
\usepackage{hyperref}
\hypersetup{
    colorlinks=true,
    citecolor=green!50!black,
    linkcolor=blue!50!black,
    filecolor=blue,
    urlcolor=blue!50!black,
}
\usepackage{cleveref}

\usepackage{tikz} 

\makeindex

\newtheorem{theorem}{Theorem}[chapter]
\newtheorem{lemma}[theorem]{Lemma}
\newtheorem{proposition}[theorem]{Proposition}
\newtheorem{corollary}[theorem]{Corollary}

\theoremstyle{definition}
\newtheorem{definition}[theorem]{Definition}
\newtheorem{example}[theorem]{Example}

\theoremstyle{remark}
\newtheorem{remark}[theorem]{Remark}

\numberwithin{section}{chapter}
\numberwithin{equation}{chapter}
\numberwithin{figure}{chapter}

\def\R{\mathbb{R}}  
\def\C{\mathbb{C}}  
\def\TT{\mathbb{T}} 
\def\GG{\mathfrak{G}} 
\def\llangle{\langle\!\langle}
\def\rrangle{\rangle\!\rangle} 
\def\coloneqq{:=}
\def\ii{\rm{i}} 
\providecommand{\pair}[1]{\langle #1 \rangle}

\newcommand{\onlyarxiv}[1]{#1}

\begin{document}

\frontmatter

\title{Information geometry of diffeomorphism groups}


\author{Boris Khesin}
\address{Department of Mathematics, University of Toronto, Toronto, ON M5S 2E4, Canada}
\email{khesin@math.toronto.edu}
\thanks{}

\author{Gerard Misio{\l}ek}
\address{Department of Mathematics, University of Notre Dame, Notre Dame, IN 46556, USA}
\email{gmisiole@nd.edu}
\thanks{}

\author{Klas Modin}
\address{Department of Mathematical Sciences, Chalmers University of Technology and University of Gothenburg, SE-412 96, Gothenburg, Sweden}
\email{klas.modin@chalmers.se}
\thanks{}

\date{\today}




\begin{abstract}
    The study of diffeomorphism groups and their applications to problems in analysis and geometry has a long history. 
    In geometric hydrodynamics, pioneered by V.~Arnold in the 1960s, one considers an ideal fluid flow as the geodesic motion on the infinite-dimensional group of volume-preserving diffeomorphisms of the fluid domain with respect to the metric defined by the kinetic energy. 
    Similar considerations on the space of densities lead to a geometric description of optimal mass transport and the Kantorovich-Wasserstein metric. 
    Likewise, information geometry associated with the Fisher-Rao metric and the Hellinger distance 
    has an equally beautiful infinite-dimensional geometric description and  
    can be regarded as a higher-order Sobolev analogue of optimal transportation.
    In this work we review various metrics on diffeomorphism groups relevant to this approach 
    and introduce appropriate topology, smooth structures and dynamics on the corresponding infinite-dimensional manifolds.
    Our main goal is to demonstrate how, alongside  topological hydrodynamics, Hamiltonian dynamics and optimal mass transport, information geometry with its elaborate toolbox has become yet another exciting field for applications of  geometric analysis on diffeomorphism groups.
\end{abstract}

\maketitle

\tableofcontents


\mainmatter

%
%
%

\chapter{Introduction} 
\label{sec:intro} 

One of the most prominent features of the geometric approach to hydrodynamics pioneered by V. Arnold \cite{Arnold66} 
is based on the observation that the group of volume-preserving diffeomorphisms of the fluid domain has a structure of an infinite dimensional Riemannian manifold 
and can serve as a natural configuration space for the motions of ideal fluids. The geodesic flow on this group, which is induced by the $L^2$ inner product corresponding to the fluid's kinetic energy, 
describes solutions of the Euler equations of an incompressible and inviscid fluid. The general framework of
Arnold turned out to include a variety of other nonlinear partial differential equations 
of mathematical physics and are now often referred to as the Euler-Arnold equations. 
The associated $L^2$ metric on the diffeomorphism group also happened 
to be relevant in another rapidly developing mathematical field, namely, the theory of optimal mass transport. 

Subsequently, a similar geometric approach to diffeomorphism groups, 
albeit related to the Sobolev $H^1$ inner product, 
shed new light on some of the most fundamental notions in geometric statistics and information theory. 
In this short book we concentrate on such $L^2$ and $H^1$ type metrics on diffeomorphism groups and their quotient spaces viewed as spaces of densities.

We have three principal goals.
First of all, we focus on developing a firm basis for the formalism of tame Fr\'echet spaces for diffeomorphism groups and demonstrate their utility in geometric hydrodynamics, information geometry, and the study of probability density spaces. In their classical paper \cite{EbinMarsden70} D.~Ebin and J.~Marsden worked out a rigorous analytical formalism for ideal hydrodynamics using Sobolev $H^s$ spaces with a finite smoothness index $s$. 
While this approach provides a convenient reformulation of the incompressible Euler equations of fluid dynamics as a system of ordinary differential equations on a Banach manifold, it is not free from various technical difficulties
making it impossible to take full advantage of the geometric formulation or extending it to other settings. 
One such important obstacle is the loss of some of the most useful Lie group properties of the underlying configuration space - inevitable in the case of function spaces with bounded smoothness properties. 
Indeed, while right compositions with Sobolev diffeomorphisms are always smooth in the $H^s$ topology, left compositions turn out to be merely continuous due to the obvious loss of derivatives. 
On the other hand, a carefully developed formalism of tame Fr\'echet spaces on diffeomorphism groups allows one to retain the desired regularity properties and regard both compositions on the left and on the right as smooth operations in this setting. 

The second goal  of this work is to survey some of the recent progress in information geometry. 
To this end we include a thorough review of the geometry of the Fisher-Rao metric, discuss its uniqueness, 
compare it with the Wasserstein metric  in optimal mass transport and  
recall the decompositions of matrices and diffeomorphisms related to the Fisher--Rao geometry. We describe the corresponding geodesics as well as relations between  entropy and the Fisher information functional and the Madelung transport and the  Schr\"odinger equation of quantum mechanics. We also describe the geometry underlying the Amari-Chentsov connections
in information geometry.

Finally, we present several new results related to the Euler-Arnold equations. These equations, describing 
the geodesic flows on Lie groups with respect to one-sided invariant metrics, are known to include
many interesting equations of mathematical physics. We give examples of derivations of these equations 
and, in particular, demonstrate how  the general Camassa-Holm equation can be obtained from 
a sub-Riemannian metric on the group of circle diffeomorphisms and as a geodesic equation 
on the extension of this group by a trivial two-cocycle. The corresponding construction turns out to be 
simpler than the Virasoro group and results in geometrically identical but analytically different equations.

\smallskip

In Section \ref{sec:inf-dim}
we start with a detailed background introducing appropriate tame Fr\'echet topology and smooth structures on infinite dimensional manifolds with an eye to a future application to diffeomorphism groups and density spaces. This will provide a firm basis to all the relevant differential-geometric and  dynamical considerations in Section \ref{sec:Diffeos}.
\smallskip

We continue in
Section \ref{sect:Riemannian} by equipping these spaces with Riemannian metrics, as an outlook of the Euler-Arnold equations and the geodesic flows of one-sided invariant metrics on Lie groups in finite and infinite dimensions. We present both the Lagrangian and the Hamiltonian side of the story and discuss how these structures induce interesting dynamical systems on certain quotient spaces. In particular, we show how the $L^2$ metric on the group of diffeomorphisms  naturally descends to the celebrated Wasserstein metric  
on the space of densities  whose geodesics describe the optimal mass transfer.
Several equations including Burgers' equation and the Camassa-Holm equation are discussed in more detail. 
The general Camassa--Holm equation is derived as a  geodesic equation of a right-invariant 
metric on a trivial extension of the group of circle diffeomorphisms. 
\smallskip

On the other hand, the Fisher-Rao metric of geometric statistics turns out to be
closely related to homogeneous Sobolev $\dot{H}^1$ right-invariant Riemannian metrics 
on the full diffeomorphism group. This is the main subject of Section \ref{sec:FR-met}.
In the setting of diffeomorphism groups, information geometry associated
with the Fisher-Rao metric and its spherical Hellinger distance can be viewed
as an $\dot{H}^1$-analogue of the standard optimal mass transport. 
\smallskip

We  describe this geometry in detail and discuss
properties of solutions of the associated geodesic equations in Section \ref{ch:FRgeodesics}. 
It turns out that the geometry of the space of densities is spherical
irrespective of the underlying compact manifold $M$, while the corresponding  Euler-Arnold equation is a natural generalization of 
the completely integrable one-dimensional Hunter-Saxton equation. 
\smallskip

Lastly, in Section \ref{sec:AC} we present geometric constructions of the so-called $\alpha$-connections
 introduced in geometric statistics by Chentsov \cite{Chentsov82} and Amari \cite{Amari82}. We also describe their generalizations to diffeomorphism groups of higher-dimensional manifolds.
\smallskip

The book concludes with an Appendix \ref{Banach} with some useful background information on Banach completions of manifolds of maps.
\medskip

We should point out that in the literature there exist other sources which present information geometry from a different perspective and develop alternative settings for infinite dimensional manifolds; 
e.g., see the books and surveys \cite{Amari16, ABNKLR, AmariNagaoka00, AyJostLeSchwachhofer17, Chentsov82, ChentsovMorozova91, KrieglMichor97, NomizuSasaki, Omori97, Schmeding}. 

\medskip 

We hope that the approach developed in this book places information geometry squarely within the general 
differential-geometric framework of diffeomorphism groups as it was envisioned at various times and contexts by H.~Cartan, A.N.~Kolmogorov, L.~Kantorovich and V.~Arnold 
and which includes hydrodynamics, symplectic geometry, optimal transport etc.. 
Our approach here lays down foundations for an infinite dimensional generalization of this fast-growing subject which will hopefully lead to further fruitful developments. 

\bigskip
~
\bigskip

\onlyarxiv{
{\bf Acknowledgements.} 
We are indebted to Martin Bauer, Peter Michor, Stephen Preston and other participants of ``Math en plein air'' series of workshops.
We also thank Alexander Schmeding for useful suggestions.
BK is grateful to the IHES (Bures-sur-Yvette, France) for its kind hospitality and support. He was also partially supported by an NSERC Discovery Grant.
KM was supported by the Swedish Research Council (grant 2022-03453) and the Knut and Alice Wallenberg Foundation (grant WAF2019.0201).
}



%
%
%

\chapter{Infinite dimensional manifolds} 
\label{sec:inf-dim} 

It has long been recognized that many of the function spaces that arise in analysis and geometry 
possess a natural structure of infinite dimensional differentiable manifolds. 
This includes various groups of diffeomorphisms of compact manifolds, 
which are of interest in this chapter. 
As mathematical objects these spaces are both very interesting and very complicated, 
and any researcher planning to take full advantage of their properties as manifolds and groups 
faces a problem at the outset of choosing a suitable topology. 

For our purposes a convenient and natural functional-analytic framework of tame Fr\'echet spaces, 
as introduced by Sergeraert and further developed by Hamilton, 
provides the most convenient setting. 

In this section we recall the basic notions of differential calculus in Fr\'echet spaces 
and then introduce the group of diffeomorphisms of a compact Riemannian manifold, 
its subgroup of diffeomorphisms preserving the volume form, 
and the quotient space of smooth probability densities, 
as tame Fr\'echet manifolds. 
%

\section{Differential calculus in Fr\'echet spaces} 
\label{subsec:DCFS} 
We begin with a brief review of the fundamentals of the calculus in Fr\'echet spaces. 
Most of the basic definitions and properties of Fr\'echet spaces can be found in the monographs 
of Dunford and Schwartz \cite{DunfordSchwartz58-63} and Rudin \cite{Rudin73}. 
The excellent expository article by Hamilton~\cite{Hamilton82} can be consulted for details 
regarding the constructions needed in the sequel. 

\subsection{Fr\'echet spaces} 
\begin{definition} \label{def:Fspace} 
A \emph{Fr\'echet space}\index{Fr\'echet space} $\mathfrak{X}$ is a complete Hausdorff topological vector space 
whose topology is defined by a countable collection of seminorms $\| {\cdot} \|_k$ 
where $k = 0, 1, 2... $. 
A sequence $u_n$ converges to $u$ in $\mathfrak{X}$ 
if and only if for all $k$ one has
$\| u_n - u \|_k \to 0$  as $n \to \infty$. 
\end{definition} 

The seminorms are separating in the sense that to each $u \neq 0$ 
there corresponds at least one $k$ for which $\| u \|_k \neq 0$. 
Furthermore, the topology on $\mathfrak{X}$ is locally convex and metrizable 
--- it has a countable local base (at the origin $0$) consisting of convex sets 
and there is a compatible translation-invariant distance function obtained 
directly from the seminorms by setting 
$$ 
\mathrm{dist}(u,v) := \sum_{k=0}^\infty 2^{-k} \frac{\| u - v \|_k}{1 + \|u - v \|_k} 
\qquad 
\text{for any $u, v \in \mathfrak{X}$}. 
$$ 
Bounded subsets of $\mathfrak{X}$ are precisely those which are bounded with respect to 
all the seminorms defining the topology. 
Thus, continuous linear transformations between Fr\'echet spaces 
can be characterized as mapping bounded sets to bounded sets. 

\begin{example} 
A canonical example of a Fr\'echet space is the space $\mathcal{C}^\infty(M)$ 
of smooth functions on a compact Riemannian manifold $M$ 
with a family of seminorms given by the uniform $\mathcal{C}^k$ norms. 
\end{example} 
\begin{example} \label{ex:F-sections} 
More generally, let $E$ be a vector bundle over $M$ equipped with a metric 
and a compatible connection $\nabla$. 
If $u$ is a smooth section of $E$ then 
$\nabla u \in \mathcal{C}^\infty(T^\ast M \otimes E)$ 
and, using the induced connection on the tensor product $T^\ast M \otimes E$, 
we also have $\nabla^2 u \in \mathcal{C}^\infty(T^\ast M \otimes T^\ast M \otimes E)$. 
Continuing this process we obtain a countable collection of the uniform $\mathcal{C}^k$ norms 
\begin{equation} \label{eq:Ck-norm} 
\| u \|_{\mathcal{C}^k} = \sum_{j=0}^k \sup_{x\in M} | \nabla^j u(x) | \,,
\qquad 
\text{$k=0, 1, 2 \dots$} \,,
\end{equation} 
which turn the space $\mathcal{C}^\infty(M, E)$ of smooth sections of $E$ into a Fr\'echet space. 
Other norms, which are often used in this setting, are 
the Sobolev $H^k$ norms 
\begin{equation} \label{eq:Sob-norm} 
\| u \|_{H^k}^2 = \sum_{j=0}^k \int_M | \nabla^j u |^2 d\mu \,,
\qquad 
\text{$k=0, 1, 2 \dots$} \,,
\end{equation} 
where $\mu$ is the Riemannian volume form on $M$, 
and 
the H{\"o}lder $\mathcal{C}^{k,\alpha}$ norms 
\begin{equation} \label{eq:Holder-norm} 
\| u \|_{\mathcal{C}^{k,\alpha}} 
= 
\| u \|_{\mathcal{C}^k} + \sup_{x \in M, 0<r<i_M}{r^{-\alpha} \omega_r( x,\nabla^k u )} \,,
\qquad 
k=0, 1, 2 \dots 
\end{equation} 
with 
$0<\alpha<1$ 
and the modulus of smoothness given by 
$$ 
\omega_r(x, u) = \sup_{y_1 \neq y_2 \in B_r(x)}{\big| \nabla^k u(y_1) - \Pi_{y_1}^{y_2} \nabla^k u(y_2) \big|} \,.
$$ 
Here 
$B_r(x)$ is a geodesic ball at $x$ of radius less than the injectivity radius $i_M>0$ of $M$ 
and 
$\Pi_{y_1}^{y_2}$ is the parallel translation operator in the tensor bundle $T^\ast M^{\otimes k} \otimes E$ 
along the unique minimal geodesic between $y_2$ and $y_1$. 
\end{example} 
\begin{example} \label{ex:Schwartz} 
Another example of a Fr\'echet space which is of particular importance in Fourier analysis 
is the Schwartz space $\mathcal{S}(\mathbb{R}^n)$. Its elements are the rapidly decreasing smooth functions 
on $\mathbb{R}^n$ and its topology is defined this time by seminorms 
$$ 
\| u \|_k = \sup_{x \in \mathbb{R}^n, |\alpha| \leq k} | x^\alpha D^\alpha u(x) | 
\qquad 
\text{for} 
\; 
k=0, 1, 2 \dots \,,
$$ 
that are not norms, where $\alpha = (\alpha_1, \dots , \alpha_n)$ is an $n$-tuple of nonnegative integers 
of length $|\alpha| = \sum_{j=1}^n \alpha_j$, 
$x^\alpha = x^{\alpha_1} \cdots x^{\alpha_n}$ 
and 
$D^\alpha = i^{-|\alpha|} \partial^\alpha/\partial x^\alpha$. 
\end{example} 

Many of the usual operations on normed spaces can be applied to construct further examples 
of Fr\'echet spaces. A closed linear subspace of a Fr\'echet space, the direct sum of Fr\'echet spaces and 
the quotient of a Fr\'echet space by a closed subspace are all Fr\'echet spaces. 
Fundamental results of abstract functional analysis such as 
the Hahn-Banach theorem, the closed graph theorem, and the open mapping theorem 
continue to hold in the Fr\'echet setting with proofs requiring only minor adjustments 
as compared with the standard Banach case. 
Thus, for example, an immediate consequence of the open mapping theorem is that 
any continuous linear bijection between Fr\'echet spaces has a continuous inverse, 
i.e., it is a topological isomorphism (homeomorphism) of Fr\'echet spaces. 
For proofs of such facts see e.g., \cite{DunfordSchwartz58-63}, \cite{Kothe69} or \cite{Rudin73}. 

An important exception arises when constructing the dual of a Fr\'echet space $\mathfrak{X}$, 
that is, the space $\mathfrak{X}^\ast$ of continuous linear functionals on $\mathfrak{X}$. 
Although the assumption of local convexity ensures, via the Hahn-Banach theorem, 
a good supply of such functionals, 
this is insufficient to guarantee that 
$\mathfrak{X}^\ast$ is a Fr\'echet space. 
In general, the dual space of a locally convex space does not carry any distinguished topology 
and it will not be Fr\'echet unless $\mathfrak{X}$ itself is normable, 
cf. \cite{Kothe69}. 
For example, the dual of $\mathcal{C}^\infty(M)$ is the space $\mathcal{D}'(M)$ 
of distributions on a compact manifold $M$. 
The same problem arises for more general continuous linear transformations between Fr\'echet spaces 
and hence some care must be taken when working with notions involving families of such maps. 
In particular, this entails the use of a notion of differentiability based on the Gateaux (directional) derivative.

\subsection{The Gateaux derivative and its properties} 
\label{subsubsec:DC} 
\begin{definition} \label{def:Gder} 
Let $\mathfrak{X}$ and $\mathfrak{Y}$ be Fr\'echet spaces. 
A continuous map $f: \mathfrak{X}  \supset  \mathscr{U} \to \mathfrak{Y}$ 
is of class $\mathcal{C}^1$ (continuously differentiable) on an open subset 
$\mathscr{U} \subset \mathfrak{X}$ 
if the limit 
\begin{equation} \label{eq:Gder} 
df(u)h 
= 
f'(u) h 
:= 
\lim_{t \to 0} \frac{f(u+th) - f(u)}{t} 
\end{equation} 
exists for all $u \in \mathscr{U}$ and $h \in \mathfrak{X}$ 
and if the map 
$df : \mathscr{U} \times \mathfrak{X} \to \mathfrak{Y}$ 
is continuous as a function of both variables. 
Partial derivatives of functions depending on two or more variables are defined in the usual manner. 
\end{definition} 

It should be noted that if $\mathfrak{X}$ and $\mathfrak{Y}$ are Banach spaces 
then this notion of differentiability is weaker than the standard one based on the Fr\'echet derivative 
which requires that the map $u \to f'(u)$ from $\mathscr{U}$ to the space of bounded linear maps 
$L(\mathfrak{X}, \mathfrak{Y})$ be continuous in the operator norm topology. 
In fact, the two derivatives do not coincide even in finite dimensions. 
\begin{example} \label{ex:GvsF} 
The real-valued function 
$$ 
f(x, y) 
= 
\left\{ 
\begin{array}{l@{\qquad}l}  
0 & \text{if} \quad (x, y) = (0,0) 
\\ 
 x+ y + \dfrac{x^3y}{x^4+y^2} & \text{if} \quad (x,y) \neq (0,0) 
 \end{array} 
 \right. 
 $$ 
is Gateaux but not Fr\'echet differentiable at the origin in $\mathbb{R}^2$. 
Note, however, that continuous Gateaux and Fr\'echet differentibility is equivalent. 
\end{example} 

The second derivative $d^2 f$ is defined as the derivative of the first order derivative, 
and we say that $f$ is of class $\mathcal{C}^2$ if $d^2 f$ exists and is continuous jointly as a function 
on the product space. 
Now $\mathcal{C}^r$ functions are defined  by induction and 
we say that $f$ is smooth of class $\mathcal{C}^\infty$ if it is $\mathcal{C}^r$ for all $r$. 
A function of two (or more) variables that is jointly continuous, $\mathcal{C}^\infty$ smooth with respect to one of the variables 
and linear with respect to the other is easily seen to be jointly smooth in both variables. 

Notwithstanding differences between the two notions,
the Gateaux derivative of any continuously differentiable function 
enjoys most of the expected properties and obeys the usual rules of calculus\footnote{It is not hard to show that 
any function of class $\mathcal{C}^2$ in the above sense is continuously differentiable 
in the sense of the standard Fr\'echet differentiability.}. 

\begin{proposition} \label{prop:Gder} 
Let $f: \mathfrak{X}\, {\supset}\,  \mathscr{U} \to \mathfrak{Y}$ be a function of class $\mathcal{C}^r$ 
between Fr\'echet spaces and let $\mathscr{U} \subset \mathfrak{X}$ be an open set. 
\begin{enumerate} 
\item{(Linearity)} 
For any $u \in \mathscr{U}$ and any $h, k \in \mathfrak{X}$ and for any scalar $c$ we have 
$$ 
df(u)(ch + k) = c df(u)h + df(u)k. 
$$ 
\item{(Fundamental theorem of calculus)} 
For any $h \in \mathscr{X}$ we have 
$$ 
f(u+th) - f(u) = \int_0^1 df(u+th)h \, dt 
$$ 
provided that $\mathscr{U}$ is convex so that the entire segment 
$u+th$ $(0 \leq t \leq 1)$ lies in $\mathscr{U}$; 
in particular, $f$ is locally constant if and only if $df=0$. 
\item 
For any $u \in \mathscr{U}$ the map 
$(h_1, \dots , h_r) \to (d^r f) (u)(h_1, \dots , h_r)$ 
is symmetric and $r$-linear. 
\item{(Chain rule)} 
If $g$ is another function of class $\mathcal{C}^r$ then so is the composition $g \circ f$, 
and we have 
$$ 
d(g \circ f)(u) = dg(f(u)) {\cdot} df(u)\,, 
$$ 
as well as analogous formulas for the iterated derivatives. 
\item{(Taylor's formula)} 
For any $h \in \mathfrak{X}$ we have 
$$ 
f(u+h) = f(u) + df(u)h + \cdots + \frac{1}{(r-1)!} \int_0^1 (1-t)^{r-1} (d^r f)(u+th)(h, \dots , h) \, dt 
$$ 
provided that $r \geq 2$ and the segment $u+th$ lies in $\mathscr{U}$. 
\end{enumerate} 
\end{proposition} 
\begin{proof} 
The proofs of all these facts are carried out in a routine manner with the help of the Hahn-Banach theorem 
by reducing to the real-valued case and applying the classical finite dimensional calculus, 
see e.g., \cite{Hamilton82}.  
\end{proof} 
%

\subsection{Manifolds modelled on Fr\'echet spaces} 
\label{subsubsec:MmFs} 
We now have enough tools of the Fr\'echet calculus to generalize 
a number of standard constructions of differential topology such as manifolds, 
vector bundles, principal bundles etc. to the Fr\'echet setting. 

\begin{definition} \label{def:Fmanifold} 
A \emph{Fr\'echet manifold}\index{Fr\'echet manifold} $\mathfrak{M}$ 
modelled on a Fr\'echet space $\mathfrak{X}$ 
is a Hausdorff topological space equipped with a maximal atlas of (pairwise compatible) coordinate charts 
$\{ \mathscr{U}_\alpha, \varphi_\alpha \}$ 
where $\{ \mathscr{U}_\alpha \}$ form an open cover of $\mathfrak{M}$ and 
$\varphi_\alpha: \mathscr{U}_\alpha \to \mathscr{V}_\alpha$ 
are homeomorphisms onto open subsets $\mathscr{V}_\alpha$ of $\mathfrak{X}$ 
such that for all $\alpha, \beta$ the coordinate transition maps 
$\varphi_\alpha \circ \varphi_\beta^{-1}$ 
defined on $\varphi_\beta(\mathscr{U}_\alpha \cap \mathscr{U}_\beta)$ are of class $\mathcal{C}^\infty$. 
\end{definition} 
\begin{remark} 
$\mathfrak{M}$ is locally compact if and only if the model space $\mathfrak{X}$ is finite dimensional. 
In this case $\mathfrak{M}$ is a smooth manifold in the usual sense. 
\end{remark} 

Many standard constructions can be now carried over from finite dimensions 
to the Fr\'echet setting without much difficulty. 
\begin{enumerate} 
\item[(i).] 
A subset $\mathfrak{S} \subset \mathfrak{M}$ is a \emph{Fr\'echet submanifold} of $\mathfrak{M}$ 
if $\mathfrak{X} = \mathfrak{X}_1 \times \mathfrak{X}_2$ is a product of Fr\'echet spaces 
and 
if around each point in $\mathfrak{S}$ there is a coordinate chart $(\mathscr{U}, \varphi)$ 
with 
$\varphi : \mathscr{U} \subset \mathfrak{M} \to \mathfrak{X}$ 
such that 
$\varphi( \mathscr{U}\cap \mathfrak{S} ) = \varphi(\mathscr{U}) \cap ( \mathfrak{X}_1 \times \{0\} )$. 
An atlas for $\mathfrak{S}$ is now obtained from that of $\mathfrak{M}$ 
by restriction, i.e., 
$\{ (\mathscr{U}_\alpha \cap \mathfrak{S}, \varphi_\alpha \vert_{\mathscr{U}_\alpha \cap \mathfrak{S}} )\}$. 
\item[(ii).] 
A continuous map $f: \mathfrak{M} \to \mathfrak{N}$ between two Fr\'echet manifolds 
modelled on $\mathfrak{X}$ and $\mathfrak{Y}$ is of class $\mathcal{C}^r$ (resp. $\mathcal{C}^\infty$) 
if for every $p \in \mathfrak{M}$ there exist charts $(\mathscr{U}, \varphi)$ at $p \in \mathscr{U}$ 
and $(\mathscr{V}, \psi)$ at $f(p) \in \mathfrak{N}$ 
such that the map 
$\psi\circ f \circ\varphi^{-1}$ from the open set $\varphi( f^{-1}(\mathscr{V}) \cap \mathscr{U})$ into $\mathfrak{Y}$ 
is of class $\mathcal{C}^r$ (resp. $\mathcal{C}^\infty$). 
If the map $f$ is bijective and if both $f$ and $f^{-1}$ are of class $\mathcal{C}^r$ (resp. $\mathcal{C}^\infty$) 
then $f$ is a $\mathcal{C}^r$ (resp. $\mathcal{C}^\infty$) \emph{diffeomorphism}\index{diffeomorphism}. 
Furthermore, any such map induces a well-defined tangent map 
$df : T\mathfrak{M} \to T\mathfrak{N}$ 
which for each $p \in \mathfrak{M}$ carries the fiber $T_p\mathfrak{M}$ linearly to $T_{f(p)}\mathfrak{N}$. 
If, in addition, we can choose the local charts so that the representatives of $f$ is the projection onto a factor in a direct sum then $f$ is a \emph{submersion}.
In that case, $df$ is surjective. 
\item[(iii).] 
In particular, let $t \to c(t)$ be a curve through a point $p$ in a Fr\'echet manifold $\mathfrak{M}$, 
that is, a differentiable map $c: \mathbb{R} \supset I \to \mathfrak{M}$ 
from an open interval $I$ containing zero with $c(0)=p$. 
Two curves $c_1$ and $c_2$ are tangent at $p$ 
if $c_1(0)=p=c_2(0)$ 
and 
$(\varphi \circ c_1)' (0) = (\varphi \circ c_2)' (0)$ 
in some chart $(\mathscr{U}, \varphi)$ (hence, by the chain rule, in every chart) around $p$. 
An equivalence class of such curves defines a \emph{tangent vector} to $\mathfrak{M}$ at $p$. 
As in finite dimensions, this establishes a bijection between 
the model space $\mathfrak{X}$ 
and 
the set $T_p\mathfrak{M}$ of all such equivalence classes (the tangent space at $p$) 
by means of which the latter acquires the structure of an isomorphic Fr\'echet space. 
The disjoint union $T\mathfrak{M} = \bigcup_{p \in \mathfrak{M}} T_p\mathfrak{M}$ 
with the natural smooth projection map $\pi: T\mathfrak{M} \to \mathfrak{M}$ 
given by $\pi(v) = p$ if $v \in T_p\mathfrak{M}$ 
becomes another Fr\'echet manifold, 
the \emph{tangent bundle} of $\mathfrak{M}$ 
modelled on the product space $\mathfrak{X}\times\mathfrak{X}$. 
\item[(iv).] 
More generally, 
one defines \emph{vector bundles} and \emph{fiber bundles} 
over $\mathfrak{M}$ in the usual way 
as another Fr\'echet manifold $\mathfrak{F}$ together with a smooth projection map 
$\pi: \mathfrak{F} \to \mathfrak{M}$ 
whose derivative is surjective (i.e., a submersion). 
Each point $p$ of the base manifold lies in some coordinate chart $(\mathscr{U}, \varphi)$ on $\mathfrak{M}$ 
with 
$\pi^{-1}(\mathscr{U}) \simeq \varphi(\mathscr{U}) \times \mathfrak{Y} \subset \mathfrak{X} \times \mathfrak{Y}$ 
and the fiber $\pi^{-1}(p)$ has the structure of a linear Fr\'echet space in the former 
and a Fr\'echet manifold in the latter case. 
A (cross-) section of either bundle is a smooth map $s: \mathfrak{M} \to \mathfrak{F}$ 
satisfying $\pi \circ s = \mathrm{id}_{\mathfrak{M}}$. 
When $\mathfrak{F}$ is the tangent bundle $T\mathfrak{M}$ then its sections are just 
smooth vector fields on $\mathfrak{M}$. 
\end{enumerate} 
%

\subsection{Geometric tools: connections, curvature, geodesics and metrics} 
\label{subsubsec:cgwrm} 
A \emph{connection} $\nabla$ on a Fr\'echet vector bundle $\mathfrak{F}$ with (standard) fiber 
$\pi^{-1}(p) \simeq \mathfrak{Y}$ 
over a manifold $\mathfrak{M}$ modelled on a Fr\'echet space $\mathfrak{X}$ 
is a smooth map that assigns to each point of $\mathfrak{F}$ 
a subspace of the tangent space which is complementary to the null space of $d\pi$ at that point. 
This amounts to assigning to each coordinate chart of the bundle $\mathfrak{F}$ 
a family of bilinear maps, called the \emph{Christoffel symbol} (or \emph{connection coefficient}) map 
\begin{equation} \label{eq:cov-der} 
\Gamma : (\mathscr{U} {\subset}\, \mathfrak{M}) \times \mathfrak{Y} \times \mathfrak{X} \to \mathfrak{Y} 
\qquad 
p, w, v \mapsto \Gamma_p (w,v) 
\end{equation} 
which is jointly continuous as a function on the product of Fr\'echet spaces, 
smooth in $p$ and linear in $v$ and $w$ 
(hence, jointly smooth in all three variables). 
The \emph{curvature} of a connection on a Fr\'echet vector bundle $\mathfrak{F}$ 
is the trilinear map 
$\mathcal{R}: T\mathfrak{M} \times T\mathfrak{M} \times \mathfrak{F} \to \mathfrak{F}$ 
whose local representation in a coordinate chart is 
$$ 
\mathcal{R}_p(u,v)w 
= 
d\Gamma(p)(w,u,v) - d\Gamma(p)(w,v,u) 
+ 
\Gamma_p(\Gamma_p(w,v),u) - \Gamma_p(\Gamma_p(w,u),v) 
$$ 
where $p \in \mathscr{U} \subset \mathfrak{M}$, $u, v \in T_x\mathfrak{M}$ and $w \in \mathfrak{F}_p$. 
The curvature $\mathcal{R}$ is independent of the choice of a chart. 

A connection\index{connection} on $\mathfrak{M}$ is by definition a connection on the tangent bundle $T\mathfrak{M}$ 
which gives rise to the notion of differentiation of vector fields on $\mathfrak{M}$. 
Namely, the \emph{covariant derivative}\index{covariant derivative} of a vector field $W$ in the direction of $V$ 
is the vector field $\nabla_V W$ whose expression in a coordinate chart $\mathscr{U} \subset \mathfrak{M}$ is 
\begin{equation} \label{eq:cov-der-TM} 
\nabla_V W(p) = dW(p) \cdot V(p) + \Gamma_p(V(p), W(p)) 
\qquad 
p \in \mathscr{U} \subset \mathfrak{M}. 
\end{equation} 
Furthermore, a connection
$\nabla$ is said to be \emph{symmetric} if its Christoffel symbols $\Gamma_p$ 
are symmetric with respect to its two entries at any point $p$. 

If $\gamma(t)$ is a smooth curve in $\mathfrak{M}$ 
and if $V$ is a vector field along\footnote{That is, the map 
$\mathbb{R} \ni t \mapsto V(t) \in T_{\sigma(t)}\mathfrak{M}$ depends differentiably on $t$.} 
$\gamma$ then a connection $\nabla$ induces covariant differentiation along $\gamma$ by the formula 
$\frac{D}{dt} V = \nabla_{\gamma'} V$. 
As in classical (finite dimensional) geometry, a curve $\gamma(t)$ in $\mathfrak{M}$ 
is a \emph{geodesic}\index{geodesic} of $\nabla$ if $\frac{D}{dt}\gamma' = 0$ 
(i.e., if the acceleration along $\gamma$ is zero) 
which in a local chart takes the form of a second order differential equation 
\begin{equation} \label{eq:GE} 
\gamma'' = \Gamma_\gamma(\gamma',\gamma'). 
\end{equation} 
\begin{remark} 
It should be pointed out that, in contrast to the case when $\mathfrak{M}$ is a Banach manifold, 
prescribing an initial position $\gamma(0)=p_0$ and velocity $\gamma'(0)=u_0$ 
does not imply that 
the corresponding Cauchy problem for the geodesic equations \eqref{eq:GE} 
admits a local (in time) unique solution. 
It is not hard to construct examples displaying either non-existence or non-uniqueness 
of solutions. 
\end{remark} 
%

%
In order to proceed it is also important to allow for generalizations of the classical Riemannian geometric notions. 
%
%
\begin{definition} \label{def:weak-rm} 
A \emph{weak-Riemannian}\index{Riemannian metric} (or \emph{pre-Riemannian}) metric on a Fr\'echet manifold $\mathfrak{M}$ 
modelled on $\mathfrak{X}$ is a smooth assignment to each point in $\mathfrak{M}$ of a positive definite bilinear form 
$g_p$ on the tangent space at $p$. 
In coordinate charts this yields a jointly continuous function $(\mathscr{U} {\subset} \,\mathfrak{M}) \times \mathfrak{X} \times \mathfrak{X} \to \mathbb{R} $
on the product:
$$ 
p, v, w \mapsto g_p(v,w) = \langle v, w \rangle_p 
$$ 
which is smooth in $p$ and linear in $v$ and $w$. 
\end{definition} 

The prefix \emph{weak-} (or \emph{pre-}) is meant to indicate that each tangent space $T_p\mathfrak{M}$ 
is an inner product (pre-Hilbert) space with the topology induced by $g_p = \langle \cdot, \cdot \rangle_p$ 
which is weaker than the Fr\'echet topology induced from the model space $\mathfrak{X}$. 
Likewise, the Riemannian distance function of a pre-Riemannian metric on $\mathfrak{M}$ 
defined in the usual manner as the infimum of lengths 
$$ 
\int_a^b \langle \eta'(t), \eta'(t) \rangle^{1/2} dt 
$$ 
of all piecewise smooth curves joining two points $p=\eta(a)$ and $q=\eta(b)$, 
induces a topology on $\mathfrak{M}$ that is also strictly weaker than its original Fr\'echet topology. 

\begin{example} \label{ex:Hr} 
For any integer $r \geq 0$ the Sobolev $H^r$ inner product 
\begin{equation} \label{eq:Hr} 
\langle u, v \rangle_{L^2} = \sum_{j=0}^r \int_M \langle \nabla^j u(x), \nabla^j v(x) \rangle \, d\mu 
\end{equation} 
on the space $\mathcal{C}^\infty(M, TM)$ of smooth vector fields on a compact Riemannian manifold $M$ 
(with volume form $\mu$) is a weak-Riemannian metric. 
\end{example} 

Let $\mathfrak{M}$ be a Fr\'echet manifold equipped with a weak-Riemannian metric 
$g = \langle \cdot, \cdot \rangle$. 
As in the finite dimensional case, we say that 
an affine connection\index{connection} $\nabla$ is \emph{Levi-Civita}\index{Levi-Civita connection} if it is symmetric and satisfies 
\begin{equation} \label{eq:met-con} 
X \langle V, W \rangle = \langle \nabla_X V, W \rangle + \langle V, \nabla_X W \rangle 
\end{equation} 
for any vector fields $X$, $V$ and $W$. 
Unlike in the finite dimensional case, the existence of a Levi-Civita connection on $\mathfrak{M}$ 
(compatible in the above sense with a weak-Riemannian metric) is not guaranteed.\footnote{Essentially, this is 
because not all continuous linear functionals on a pre-Hilbert space can be represented by the inner product.} 
However, if such a connection can be defined then it is necessarily unique. 

%
In their geometric treatments of statistics 
Amari \cite{ABNKLR} and Chentsov \cite{Chentsov82} 
found it useful to generalize the metric property \eqref{eq:met-con} 
and work systematically with a more adequate concept of (squared) distance from one point to another. 
These notions can be defined also in our present setting. 

More precisely, let $\mathfrak{M}$ be a Fr\'echet manifold with a pre-Riemannian metric $g$ as above. 
Two connections $\nabla$ and $\nabla^\ast$ on $\mathfrak{M}$ 
are said to be \emph{dual}\index{dual connection} (or \emph{conjugate}) with respect to $g$ if 
\begin{equation} \label{eq:dual-con} 
X \langle V, W \rangle = \langle \nabla_X V, W \rangle + \langle V, \nabla^\ast_X W \rangle 
\end{equation} 
for any vector fields $X$, $V$ and $W$ on $\mathfrak{M}$. 
It is clear that in this case $(\nabla + \nabla^\ast)/2$ is a connection satisfying \eqref{eq:met-con}. 
The triple $(g, \nabla, \nabla^\ast)$ is said to define a \emph{dualistic structure} on $\mathfrak{M}$. 

Similarly, we say that a smooth function $D: \mathfrak{M} \times \mathfrak{M} \to \mathbb{R}$ 
is a \emph{divergence} (also called a \emph{contrast function}\index{divergence}) if 
\begin{enumerate} 
\item[(i)] 
for any $p, q \in \mathfrak{M}$ we have 
$D(p, q) \geq 0$ with equality only if $p=q$ 
\item[(ii)] 
the matrix of second derivatives  $-({\partial^2 D(p,q)}/{\partial p\,\partial q}) |_{p=q}$
is strictly positive definite at every $p \in \mathfrak{M}$. 
\end{enumerate} 
Condition (ii) de facto defines a Riemannian metric on $\mathfrak{M}$ 
together with a covariant derivative whose Christoffel symbols can be obtained 
from the matrix of its third order derivatives. 
In the finite dimensional case these formulas take the form 
\begin{equation}\label{eq:metric_from_divergence}
g_{ij}(p) 
= 
-\frac{\partial^2 D(p,q)}{\partial p_i\partial q_j} \big|_{p=q}, 
\;\; 
\Gamma_{ij,k}(p) 
= 
-\frac{\partial^3 D(p,q)}{\partial p_i \partial p_j \partial q_k} \big|_{p=q} 
\quad 
1 \leq i, j, k \leq n. 
\end{equation}
 
A symmetric connection\index{connection} $\nabla$ is said to be \emph{flat} if its curvature tensor vanishes. 
A Fr\'echet manifold $\mathfrak{M}$ equipped with a dualistic structure is called \emph{dually flat} if 
both $\nabla$ and $\nabla^\ast$ are flat.

\section{The tame category of Sergeraert and Hamilton} 
\label{subsec:tameSH} 
One important item conspicuously missing from the list in Sect. \ref{subsubsec:DC} and \ref{subsubsec:MmFs} 
is the inverse function theorem --- perhaps the most fundamental result of differential calculus. 
It shows that the study of many nonlinear problems in analysis can be effectively 
accomplished by linearization. 
It is also a useful tool in geometric analysis when it comes to constructing 
nontrivial examples of manifolds. 
Such a tool will be needed to endow the group of volume preserving diffeomorphisms 
with the structure of a Fr\'echet manifold. 

\subsection{Tame Fr\'echet spaces} 
\label{subsubsec:TFs} 
There is a good reason for this omission.\index{tame Fr\'echet space}
While it is well known that this theorem holds in the category of Banach spaces 
(see e.g., \cite{Dieudonne69} or \cite{Lang99}) 
a straightforward generalization fails spectacularly for Fr\'echet spaces 
with any reasonable notion of differentiability. 
For a simple example consider the map $f \mapsto e^f$ 
from the space $\mathcal{C}(\mathbb{R})$ of continuous real-valued functions on the line into itself 
with the topology of uniform convergence on compact sets. 
Clearly, its differential at $0$ is the identity and the function is injective. 
However, it is not locally invertible because any neighbourhood of $1$ contains functions 
which can assume negative values. 
A more elaborate example of this phenomenon of greater geometric interest 
involves the (Lie group) exponential mapping of the group of diffeomorphisms, 
see Example \ref{ex:exp-Lie} below. 
Other interesting counterexamples can be found in \cite{LojasiewiczZehnder79} or \cite{Hamilton82}. 

A satisfactory replacement for that is the Nash-Moser inverse function theorem 
whose formulation requires introducing some extra structure. 
First, note that, without loss of generality, 
the seminorms in any Fr\'echet space $\mathfrak{X}$ can be assumed to be graded by strength 
$$ 
\| u \|_k \leq \| u \|_{k+1} 
\qquad \text{for} \; 
k = 0, 1, 2 \dots 
\; \text{and} \; 
u \in \mathfrak{X}. 
$$ 
This can be achieved simply by adding to each seminorm all the seminorms of lower index. 
There may be many different collections of gradings defining the same topology on $\mathfrak{X}$. 
Once a specific choice of a grading is made, the space $\mathfrak{X}$ is referred to as a graded Fr\'echet space. 
\begin{example} \label{ex:gF1} 
Any Banach space is a graded Fr\'echet space. 
\end{example} 
\begin{example} \label{ex:gF2} 
The sequential space 
$$ 
\Sigma\mathfrak{B} 
= 
\big\{ 
\{ x_n \}_{n = 1,2...} \subset \mathfrak{B}: \| \{ x_n \} \|_k = \sup_n{ e^{nk} \|x_n\| } < \infty 
\; \text{for all $k$} 
\big\}
$$ 
consisting of exponentially decreasing sequences 
in a fixed Banach space $\mathfrak{B}$ with norm $\| {\cdot} \|$ is a graded Fr\'echet space 
with seminorms $\| {\cdot} \|_k$. 
\end{example} 
\begin{example} \label{ex:gF3} 
The Fr\'echet space $\mathcal{C}^\infty(M, E)$ of smooth sections of 
a vector bundle $E$ over a compact manifold $M$ is graded by 
either the uniform $\mathcal{C}^k$ norms \eqref{eq:Ck-norm} 
or the Sobolev $H^s$ norms \eqref{eq:Sob-norm}. 
\end{example} 
\begin{definition} \label{def:tame-map} 
Let $\mathfrak{X}$ and $\mathfrak{Y}$ be graded Fr\'echet spaces 
and let $\mathscr{U} \subset\mathfrak{X}$ be an open set. 
A continuous map $f:  \mathfrak{X}  \supset \mathscr{U} \to \mathfrak{Y}$ is said to be a \emph{tame map}\index{tame map} 
if there are integers $r$ (the degree) and $b$ (the base) such that 
\begin{equation} \label{eq:tame-est} 
\| f(u) \|_k \leq C (1+ \| u \|_{k+r}) 
\qquad 
\text{for all $u \in \mathscr{U}$ and $k \geq b$} \,,
\end{equation} 
where $C > 0$ depends on $k$. (Note that the numbers $r$ and $b$ 
may be different for different open sets $\mathscr{U}$.)
The map is said to be a \emph{smooth} tame map if $f$ is of class $\mathcal{C}^\infty$ 
and all of its Gateaux derivatives are tame. 
\end{definition} 
\begin{remark} 
In the special case when $f$ is a linear map $L: \mathfrak{X} \to \mathfrak{Y}$ we have 
\begin{equation*} 
\| Lu \|_k \leq C \| u \|_{k+r} 
\qquad 
\text{for any $k \geq b$} 
\end{equation*} 
by applying \eqref{eq:tame-est}  
to $\varepsilon u/\|u\|_b$ 
with sufficiently small $\varepsilon >0$ and any $u \neq 0$ 
(increasing the base $b$ if necessary to ensure that $\|u\|_b \neq 0$) 
and using linearity of $L$. 
\end{remark} 

It is not hard to see that 
any tame linear map is continuous in the Fr\'echet topology. 
Furthermore, 
a linear isomorphism of Fr\'echet spaces which is tame and has a tame inverse 
establishes an equivalence of gradings. 
Consequently, in order to show that a map $f$ is tame it is enough to establish 
the estimates in \eqref{eq:tame-est} for any pair among the equivalent gradings 
on $\mathfrak{X}$ and $\mathfrak{Y}$. 
\begin{example} 
The gradings of the space of smooth sections $\mathcal{C}^\infty(M,E)$ of a bundle $E$ 
given by the uniform $\mathcal{C}^k$ norms, the Sobolev $H^s$ norms or 
the H\"older $\mathcal{C}^{k,\alpha}$ norms are equivalent.\footnote{These gradings are easily shown to be equivalent 
with the help of the Sobolev lemma.}
\end{example} 
\begin{example} 
A routine verification of the definitions shows that products and 
compositions of tame Fr\'echet maps are tame 
and, furthermore, 
any (linear or nonlinear) partial differential operator of order $r$ 
between smooth sections of vector bundles over a compact manifold $M$ 
is a tame map of degree $r$ 
(cf. also Section \ref{sec:Diffeos} below). 
\end{example} 
%

%
\begin{definition} \label{def:tame-man} 
A graded Fr\'echet space $\mathfrak{X}$ is \emph{tame} if 
there is a Banach space $\mathfrak{B}$ such that the identity operator on $\mathfrak{X}$ 
factors through $\Sigma\mathfrak{B}$, that is  
$$ 
\begin{tikzpicture}[scale=1.0] 
\path (-1.5, 0)  node (A) {$\mathfrak{X}$} 
         ( 1.5, 0)  node (B) {$\mathfrak{X}$} 
         ( 0.0, -2) node (C) {$\Sigma{\mathfrak{B}}$} ; 
\draw[->] (A)--(B)         node[above, midway]  {$\mathrm{Id}$}; 
\draw[->] (A)--(C.120)  node[left,midway]        {$L$}; 
\draw[->] (C.60)--(B)    node[right,midway]      {$M$}; 
\end{tikzpicture} 
$$ 
for some tame linear maps 
$L: \mathfrak{X} \to \Sigma\mathfrak{B}$ 
and 
$M: \Sigma\mathfrak{B} \to \mathfrak{X}$. 
A \emph{tame Fr\'echet manifold} is a Fr\'echet manifold $\mathfrak{M}$ 
modelled on a tame Fr\'echet space equipped with an atlas 
whose coordinate transition functions are smooth tame maps. 
\end{definition} 
\begin{example}
Any submanifold of a tame Fr\'echet manifold is tame. 
\end{example} 
\begin{example} \label{ex:tame-ex} 
Let $M$ be a compact manifold. 
The Fr\'echet space of smooth sections of any fiber (or vector) bundle over $M$ 
is a tame Fr\'echet manifold. 
\end{example} 
\begin{remark} 
Although Definition \ref{def:tame-man} looks somewhat technical, 
the property it captures can be described perhaps more intuitively as the space $\mathfrak{X}$ admitting 
a family of "smoothing operators" $S_\theta : \mathfrak{X} \to \mathfrak{X}$ with $\theta \geq 0$, 
that is, linear maps satisfying certain estimates that single out a preferred grading 
among those defining the same topology on $\mathfrak{X}$.\footnote{Such operators 
(defined by convolutions with smooth functions and Fourier truncation methods) 
were used in various contexts, e.g. by Nash \cite{Nash56}, Moser \cite{Moser61} and Hormander \cite{Hormander76} 
in the work on isometric embeddings of Riemannian manifolds 
and in designing iteration schemes for solving nonlinear partial differential equations.} 
They can be explicitly constructed on the sequential spaces $\mathfrak{X} = \Sigma\mathfrak{B}$ 
and 
shown to satisfy for any $m \leq n$ the estimates 
\begin{equation} \label{eq:smooth-est} 
\| S_\theta u \|_n \leq C e^{(n-m)}\theta \| u \|_m 
\quad \text{and} \quad 
\| (\mathrm{Id} - S_\theta)u \|_m \leq C e^{-(n-m)\theta} \| u \|_n \,,
\end{equation} 
where $C>0$ is a constant depending only on $m$ and $n$. 
In turn, these estimates yield useful interpolation inequalities\footnote{Such interpolation inequalities 
are well-known for $\mathcal{C}^\infty(M)$ equipped with either the Sobolev $H^k$ or 
the H\"older $\mathcal{C}^{k,\alpha}$ norms in \eqref{eq:Sob-norm} or \eqref{eq:Ck-norm}.} 
\begin{equation} \label{eq:interpol} 
\| u \|_m^{n-l} \leq C \| u \|_n^{m-l} \| u \|_l^{n-m} 
\end{equation} 
for any $l \leq m \leq n$ with $C>0$ depending on $l, m$ and $n$. 
Such estimates are then used to implement a rapidly converging iteration scheme 
(a modified Newton algorithm) needed in the proof. 
\end{remark} 

In the tame Fr\'echet setting we can state now the following version of the inverse function theorem. 
\begin{theorem} \label{prop:NMH} \emph{(The Nash-Moser-Hamilton theorem)} 
\index{Nash-Moser-Hamilton theorem}
Let $\mathfrak{X}$ and $\mathfrak{Y}$ be tame Fr\'echet spaces 
and let $\mathscr{U} \subset \mathfrak{X}$ be an open set. 
Let $f:  \mathfrak{X} \supset \mathscr{U}  \to \mathfrak{Y}$ be a smooth tame map. 
Suppose that there is an open subset $\mathscr{V} \subset \mathscr{U}$ such that 
$df(u):\mathfrak{X} \to \mathfrak{Y}$ is a linear isomorphism for all $u \in \mathscr{V}$ 
whose inverse 
$(df)^{-1} : \mathscr{V} \times \mathfrak{Y} \to \mathfrak{X}$ is a smooth tame map. 
Then $f$ is locally invertible on $\mathscr{V}$ and the inverse is also a smooth tame map. 
\end{theorem} 
\begin{proof} 
For a detailed proof we refer to \cite{Hamilton82}. 
Shorter expositions with a somewhat different emphasis 
can be found in \cite{LojasiewiczZehnder79} and \cite{KrieglMichor97}. 
\end{proof} 
\begin{remark} \label{rem:225}
The inverse function theorem may be used to solve differential equations and in particular 
to find integral curves of (possibly time-dependent) vector fields on Fr\'echet manifolds. 
Outside of Banach spaces the situation becomes much less clear as the main tool for such purposes, 
namely, Banach's contraction mapping principle, is no longer available 
and examples showing that solutions may neither exist nor be unique are not difficult to construct. 
However, in the tame Fr\'echet setting the presence of the smoothing operators $S_\theta$, 
such as those in \eqref{eq:smooth-est}, 
makes it possible to first mollify the differential equation and then produce a solution 
by passing to the limit with $\theta \to \infty$. 
\end{remark} 
%

\subsection{Manifolds of maps} \label{subsec:mm} 
The next example of a Fr\'echet manifold is of considerable geometric interest. 
%
Let $M$ be a compact manifold (without boundary) and let $N$ be a finite dimensional manifold. 
Without loss of generality we may assume $N$ to be Riemannian. 
\begin{proposition}[\cite{Hamilton82}] \label{prop:mm} 
The space $\mathfrak{M}(M,N)$ of all smooth maps of $M$ into $N$ 
is an infinite dimensional tame Fr\'echet manifold.\index{Fr\'echet manifold} 
\end{proposition} 
\begin{proof} 
There are different ways of exhibiting a differentiable manifold structure of this space. 
We sketch a proof based on an idea developed by Eells \cite{Eells58}. 
It consists of four steps. 

\begin{enumerate} 
\item[Step 1:] 
We fix a connected component of the space $\mathfrak{M}(M,N)$. 
Now, the first objective is to find a suitable candidate for the model space.
As indicated already, 
this space should be isomorphic to the tangent space 
$T_f \mathfrak{M}(M,N)$ 
at each point $f$ in $\mathfrak{M}(M,N)$ 
and it can be therefore identified with the Fr\'echet space 
$\mathcal{C}^\infty\big( f^{-1}( TN ) \big)$ 
of smooth sections of the pull-back of the bundle $TN$ to $M$ by the map $f$ 
and topologized by the uniform norms (cf. Example \ref{ex:F-sections}). 
\item[Step 2:] 
Next, in order to define local charts at $f \in \mathfrak{M}(M,N)$ 
we pick a Riemannian metric on $N$ (any Riemannian metric will do) 
and recall that 
for any $p \in N$ the associated \emph{Riemannian exponential}\index{Riemannian exponential} map $\exp_p : T_pN \to N$ 
is a local diffeomorphism near zero in $T_pN$ onto a neighbourhood of $p$. 
Since $f$ is continuous, $f(M)$ is a compact subset of $N$ and thus the injectivity radius 
of the target manifold has a global lower bound $\varepsilon > 0$ on $f(M)$. 
Setting 
\begin{align*} 
&\mathscr{U}_f(\varepsilon) 
= 
\Big\{  
M \ni x \to \exp_{f(x)} (w(x)): 
w \in \mathcal{C}^\infty \big( f^{-1}(TN) \big), 
\; 
\| w \|_k < \varepsilon_f  
\Big\} 
\end{align*} 
and 
\begin{equation} \label{eq:phi-chart} 
w \to \varphi_f (w) = \exp_f \circ w 
\end{equation} 
gives now a coordinate chart $(\mathscr{U}_f, \varphi_f)$ at $f$. 
\item[Step 3:]  
It remains to verify that given any points $f$ and $g$ in $\mathfrak{M}(M,N)$ 
the coordinate transition maps 
$$ 
\varphi_g \circ \varphi_f^{-1}: 
\varphi_f(\mathscr{U}_f \cap \mathscr{U}_g) \to \varphi_g(\mathscr{U}_f \cap \mathscr{U}_g) 
$$ 
are of class $\mathcal{C}^\infty$, 
which follows essentially from the chain rule etc. ---
after localizing the transition maps to functions defined on open sets in Fr\'echet spaces 
of smooth sections of appropriate vector bundles. 
This shows that $\mathfrak{M}(M,N)$ is a smooth Fr\'echet manifold. 
\item[Step 4:] 
The fact that it is tame is an immediate consequence of the fact that 
$\mathfrak{M}(M,N)$ can be viewed as the space of sections of the fiber bundle 
$F=M \times N$ over $M$ 
(see Example \ref{ex:tame-ex}). 
\end{enumerate} 
\end{proof} 
%
%

The following special cases of this example introduce objects that will play an important role in what follows. 
\begin{example} 
The set $\mathfrak{D}(M,N)$ of all smooth diffeomorphisms\index{diffeomorphism} between two compact manifolds $M$ and $N$ 
or, more generally, 
the set $\mathfrak{E}(M,N)$ of all smooth embeddings of a compact manifold $M$ into a manifold $N$, 
is a tame Fr\'echet manifold being an open subset of $\mathfrak{M}(M,N)$. 
\end{example} 
\begin{example} \label{ex:circle-dens} \emph{(The space of probability densities on $\mathbb{T}$)} 
If $M$ is the unit circle $\mathbb{T}$ then diffeomorphisms of $M$ it can be viewed as 
$2\pi$-pseudo-periodic functions, $\phi(x + 2\pi) = \phi(x) + 2\pi$ such that $\phi'(x) > 0$ for all $x \in \mathbb{R}$. 
It follows that any $\phi'$ is  $2\pi$-periodic and satisfies the integral (fixed volume) constraint 
$$ 
\int_0^{2\pi} \phi'(x) \, dx = \phi(2\pi) - \phi(0) = 2\pi 
$$ 
so that the set of normalized derivatives $\phi'/2\pi$ 
--- which we shall denote by $\mathfrak{Dens}(\mathbb{T})$ --- 
can be viewed as the Fr\'echet space of smooth (probability) densities on $\mathbb{T}$. 
Since $\phi'$ determines $\phi$ uniquely up to a constant in $2\pi \mathbb{Z}$ 
we find that 
the set of smooth orientation-preserving diffeomorphisms of $\mathbb{T}$, denoted by $\mathfrak{D}(\mathbb{T})$, 
is diffeomorphic (as a Fr\'echet manifold) to the product 
$\mathbb{T} \times \mathfrak{Dens}(\mathbb{T})$. 
Thus, the space of smooth probability densities on $\mathbb{T}$ is a quotient space of $\mathfrak{D}(\mathbb{T})$. 
Moreover, it is clearly contractible as a convex open subset of a closed affine subspace (of codimension $1$) 
of the Fr\'echet space of $2\pi$-periodic functions. 
\end{example} 
\begin{remark}[Banach completions of manifolds of maps] 
The construction in the proof of Proposition \ref{prop:mm} is quite general 
and can be readily adapted to other function spaces, 
including Banach spaces of functions with finite smoothness conditions such as 
$\mathcal{C}^k(M,N)$ with the uniform norm \eqref{eq:Ck-norm} 
or 
the space $H^k(M,N)$ of maps of Sobolev class with the norm \eqref{eq:Sob-norm}. 
In these cases the corresponding manifold of maps admits the structure of a smooth Banach manifold. 
However, to carry this out one requires that 
(i) the topology of the modelling space be stronger than the uniform topology\footnote{This means that 
in the two cases above we require $k \geq 0$ or $k > \dim{M}/2$, respectively.}  
and 
(ii) the functions are "well-behaved" under compositions and inverses. 
The latter, in particular, leads to serious analytical obstacles 
concerning derivative loss when performing various operations 
involving compositions of functions of finite smoothness, 
since e.g. any iteration scheme (such as Newton's algorithm, Picard's method of successive approximations etc.) 
would quickly degenerate in complete loss of differentiability. 
See also Appendix~\ref{Banach} below. 
\end{remark} 
%



%
%
%

\chapter{Diffeomorphism groups and their quotients} 
\label{sec:Diffeos} 
Groups of diffeomorphisms of compact manifolds arise naturally 
as symmetry groups of various geometric structures (such as volume forms or symplectic forms) 
carried by the underlying manifold 
or 
as configuration spaces of dynamical systems characterized by infinitely many degrees of freedom. 
As already mentioned, we will view them as infinite dimensional Fr\'echet manifolds 
(in the sense of the previous section). 
However, in the vast literature on the subject 
other function space topologies are also used and in many cases provide 
a more suitable setting 
depending on the analytical or geometric tasks at hand.\footnote{For example, when studying nonlinear PDE 
where Hilbert space techniques may provide more precise tools to derive the necessary {\it a priori} estimates.} 
%

\section{Fr\'echet Lie groups} 
%
\begin{definition} \label{def:tameLG} 
A \emph{tame Fr\'echet Lie group}\index{Fr\'echet Lie group} is a tame Fr\'echet manifold $\mathfrak{G}$ 
whose group operations of multiplication $g, h \mapsto g{\cdot}h$ and inversion $g \mapsto g^{-1}$ 
are $\mathcal{C}^\infty$ smooth tame maps 
of $\mathfrak{G}\times\mathfrak{G}$ and $\mathfrak{G}$ into $\mathfrak{G}$, respectively. 
The \emph{Lie algebra} $\mathfrak{g}$ of $\mathfrak{G}$ is the tangent space 
$T_e\mathfrak{G}$ at the identity element $e$. 
\end{definition} 

Despite the complications resulting from infinite dimensions Fr\'echet Lie groups 
can be very effectively studied using standard Lie-theoretic tools. 
The Lie algebra $\mathfrak{g}$ is naturally isomorphic (as a vector space) 
to the space of left- (resp. right-) invariant vector fields on $\mathfrak{G}$, 
since any element $v$ of the algebra 
generates a unique vector field $V$ on the group by the formula 
$V(g) = dL_g(e) v$ where $L_g(h) = g{\cdot}h$ (resp. $dR_g(e) v$, where $R_g(h) = h{\cdot}g$) 
is the left- (resp. right-) translation. 

The \emph{group adjoint} action\index{adjoint action} 
$\mathrm{Ad}: \mathfrak{G} \times \mathfrak{g} \to \mathfrak{g}$ 
of $\mathfrak{G}$ on its Lie algebra 
is defined in the standard way as the derivative at $h=e$ of the smooth map 
of $\mathfrak{G}$ to itself given by inner automorphisms 
$h \mapsto g{\cdot} h {\cdot} g^{-1}$ 
for any fixed group element $g$. 
Namely, we have 
$$ 
v \mapsto \mathrm{Ad}_g v = d(L_g \circ  R_{g^{-1}}) (e) v 
\qquad {\rm for~~ any }\quad
g \in \mathfrak{G}. 
$$ 

Since $\mathrm{Ad}_g v$ is smooth in $g$ and linear in $v$, we can define similarly 
the \emph{algebra adjoint} action 
$\mathrm{ad}: \mathfrak{g} \times \mathfrak{g} \to \mathfrak{g}$ 
as the derivative of the group adjoint at $g=e$ 
$$ 
v \mapsto \mathrm{ad}_u v = \frac{d}{dt} \big\vert_{t=0} \mathrm{Ad}_{g(t)} v 
\qquad {\rm for~~ any }\quad
v \in \mathfrak{g} \,,
$$ 
where $g(t)$ is a smooth curve in $\mathfrak{G}$ with $g(0)=e$ and $\dot{g}(0) = u$. 
As in finite dimensions it induces a commutation operation $ \mathrm{ad}_v w = [V, W] $ 
 on the Lie algebra, 
which coincides with the  Lie bracket of left-invariant vector fields 
on the group $\mathfrak{G}$ generated by $v, w \in \mathfrak{g}$. (It also coincides with the negative 
of the Lie bracket of  right-invariant vector fields.)
\begin{example} 
The Lie algebra of the general linear group $\mathfrak{G} = GL(n,\mathbb{R})$ of invertible $n \times n$ matrices  
is the space $\mathfrak{g} = \mathfrak{gl}(n,\mathbb{R})$ of all square $n \times n$ matrices. 
The corresponding adjoint actions are given by 
$\mathrm{Ad}_g B = g B g^{-1}$ and $\mathrm{ad}_A B = AB - BA$. 
\end{example} 

Let $\mathfrak{H}$ be a subgroup of a Fr\'echet Lie group $\mathfrak{G}$. 
Then $\mathfrak{H}$ acts on $\mathfrak{G}$ by left multiplications 
$$ 
\mathfrak{H} \times \mathfrak{G} \to \mathfrak{G} 
\qquad 
h, g \mapsto h{\cdot}g.
$$ 
The orbits in $\mathfrak{G}$ under this action are the \emph{right cosets}\footnote{Some authors refer to 
such orbits as left cosets, which may cause some confusion.} 
of $\mathfrak{H}$ in $\mathfrak{G}$, that is 
$$ 
\mathfrak{H} {\cdot} g = \{ h {\cdot} g\mid h \in \mathfrak{H} \} 
\qquad {\rm for }\quad
g \in \mathfrak{G} 
$$ 
and the quotient space consisting of all such cosets is denoted by $\mathfrak{H}\backslash\mathfrak{G}$. 
Similarly, using right multiplications one defines the quotient space of \emph{left cosets} 
$\mathfrak{G}/\mathfrak{H}$ whose elements are $g {\cdot} \mathfrak{H}$.

\section{Diffeomorphism groups as Fr\'echet manifolds and Lie groups} 
\label{subsec:Diff} 
Our principal examples are groups of diffeomorphisms and their quotient spaces. 
In the sequel we shall always assume that all diffeomorphisms are orientation preserving. 
Let $\mathfrak{G}=\mathfrak{D}(M)$ denote the set of all smooth diffeomorphisms 
of a compact $n$-dimensional manifold $M$. 
As an open subset of the tame Fr\'echet manifold $\mathfrak{M}(M,M)$ it is itself a smooth tame Fr\'echet manifold. 
Its model space 
(identified with the tangent space $\mathfrak{g}=T_e\mathfrak{D}(M)$ at the identity diffeomorphism $e$) 
is just the space $\mathcal{C}^\infty(M,TM)$ of all smooth vector fields on $M$. 

The group operations on $\mathfrak{D}(M)$ 
are given by compositions $\eta, \xi \mapsto \mathrm{c}(\eta, \xi) = \eta\circ\xi$ 
and inversions $\eta \mapsto \mathrm{i}(\eta) = \eta^{-1}$ of diffeomorphisms. 
The group adjoint action is given by the change of coordinates map 
$Ad_\varphi v = \varphi_\ast v \circ \varphi^{-1}$ 
and 
the algebra adjoint is the standard Lie derivative 
$ad_v w = \mathcal{L}_v w = [v, w]$ of vector fields on $M$.  
\begin{proposition}[\cite{Hamilton82}] \label{prop:Diff} 
Let $M$ be a compact manifold. The set $\mathfrak{D}(M)$ 
of all diffeomorphisms of $M$ is a smooth tame Fr\'echet Lie group.
\index{Fr\'echet Lie group}\index{diffeomorphism}
\end{proposition} 
\begin{proof} 
We already noted that $\mathfrak{D}(M)$ is a tame Fr\'echet manifold. 
To show that the composition map $\mathrm{c}(\eta, \xi)$ is smooth we need to compute its derivatives 
on the product manifold $\mathfrak{D}(M) \times \mathfrak{D}(M)$. 
If $t \to \eta(t)$ is a smooth curve through $\eta$ with the tangent vector $\partial_t\eta (0) = V \in T_\eta\mathfrak{D}$ 
and 
$s \to \xi(s)$ is a curve through $\xi$ with $\partial_s \xi(0) = W \in T_\xi\mathfrak{D}$,
then applying the rules of calculus from Section~\ref{sec:inf-dim} we find 
that the derivative at $t=s=0$ is the sum of its two partial derivatives 
$$ 
d\mathrm{c}(\eta, \xi) (V, W) 
= 
V \circ \xi + D\eta\circ\xi \cdot W \,,
$$ 
which shows that $\mathrm{c}(\eta,\xi)$ is differentiable. 
Existence of the higher order differentials follows similarly. 

To show that $\mathrm{c}(\eta, \xi)$ is tame it suffices to work in local charts on $\mathfrak{D}(M)$ 
and establish tame estimates \eqref{eq:tame-est} for the local representatives of the diffeomorphisms 
satisfying $\| \eta \|_1, \| \xi \|_1 \leq 1$ 
where $\| {\cdot} \|_k$ is any one of the equivalent gradings of the model space $\mathcal{C}^\infty(M,TM)$ 
given by the H\"older $\mathcal{C}^{k,\alpha}$, the Sobolev $H^k$, or the uniform $\mathcal{C}^k$ norms. 
For example, we have 
$$ 
\| \eta \circ \xi \|_0 
= 
\sup_{x} |\eta \circ \xi (x)| \leq \| \eta \|_{\mathcal{C}^1} 
= 
\| \eta \|_1 \leq 1. 
$$ 
Furthermore, for any $k \geq 1$ successive applications of the chain rule 
together with the interpolation inequalities \eqref{eq:interpol} for the uniform norms 
yield 
\begin{align*} 
\| \eta \circ \xi \|_k 
= 
\sum_{|\alpha| \leq k} 
\| D^\alpha (\eta \circ\xi) \|_0 
&\leq 
\sum_{l=0}^k \sum_{j=0}^l \sum_{i_1 + \cdots + i_j = l} C_{l, i_1, \dots, i_l} 
\| \eta \|_{\mathcal{C}^l} \| \xi \|_{\mathcal{C}^{i_1}} \cdots \| \xi \|_{\mathcal{C}^{i_l}} 
\\ 
&\leq 
C_k \sum_{l=0}^k 
\big( \| \eta \|_{\mathcal{C}^l} \| \xi \|_{\mathcal{C}^1} + \| \eta \|_{\mathcal{C}^1} \| \xi \|_{\mathcal{C}^l} \big) 
\| \xi \|_{\mathcal{C}^1}^{l-1} 
\\ 
&\leq 
C_k \big( 1+ \| \eta \|_{\mathcal{C}^k} + \| \xi \|_{\mathcal{C}^k} \big). 
\end{align*} 
Thus, the composition map satisfies a tame estimate of degree $0$ and base $1$. 
Since all of its Gateaux derivatives are linear combinations of compositions and products of derivatives 
of $\eta$ and $\xi$, it follows that $\mathrm{c}(\eta, \xi)$ is a smooth tame map. 

Turning to the inversion map we first note that differentiating in $t$ 
the identity $\eta^{-1}(t) \circ \eta(t) = e$ and evaluating at $t=0$ gives 
$$ 
d\mathrm{i}(\eta) V 
= 
- D\eta^{-1} \cdot V\circ\eta^{-1} 
= 
- (D\eta)^{-1} \circ \eta^{-1} \cdot V\circ\eta^{-1}. 
$$ 
This shows that $\mathrm{i}(\eta):= \eta^{-1}$ is differentiable and, in fact, smooth 
as a map from $\mathfrak{D}(M)$ to itself, 
since all the higher differentials are computed analogously. 
Showing that $\mathrm{i}(\eta)$ is a tame map involves once again local charts and 
successive applications of the chain rule and interpolation estimates 
in a manner similar to that for the composition map. 
For further details we refer to \cite{Hamilton82}. 
\end{proof} 

The group of diffeomorphisms $\mathfrak{D}(M)$ has several important subgroups. 
Of particular importance is the stabilizer subgroup of the Riemannian volume form $\mu \in \Omega^n(M)$, 
namely, the group of volume-preserving diffeomorphisms\index{volume-preserving diffeomorphisms} 
\begin{equation*} 
\mathfrak{D}_\mu(M) 
= 
\big\{ \eta \in \mathfrak{D}(M): \eta^\ast \mu = \mu \big\} 
\end{equation*} 
where 
$\eta^\ast \mu = \mathrm{Jac}_\mu \eta \, \mu$ 
and the Jacobian is computed with respect to the Riemannian reference volume $\mu$. 
The tangent space at the identity map $e$ consists of divergence-free vector fields on $M$, 
that is 
\begin{equation*} 
T_e \mathfrak{D}_\mu(M) 
= 
\big\{ u \in T_e\mathfrak{D}(M): \mathrm{div}_\mu u = 0 \big\} \,,
\end{equation*} 
which is evident by differentiating the identity $\eta_t^\ast \mu = \mu$ at $t=0$ 
to get 
$$ 
0 = \frac{d}{d t} \Big|_{t=0} \eta_t^\ast \mu 
= 
\eta_t^\ast \Big( \mathcal{L}_{\frac{d\eta_t}{dt} \circ\eta_t^{-1} } \mu \Big) 
= \mathrm{div}_\mu u \,,
$$ 
where $t \mapsto \eta_t$ is a curve of volume-preserving diffeomorphisms 
issuing from $e$ in the direction $u$. 
\begin{proposition}[\cite{Hamilton82}]  \label{prop:D-mu} 
The group $\mathfrak{D}_\mu(M)$ of smooth volume-preserving diffeomorphisms 
of a compact Riemannian manifold $M$ is a closed tame Fr\'echet Lie subgroup of $\mathfrak{D}(M)$. 
\end{proposition} 
\begin{proof} 
Since the defining pullback condition is nonlocal, constructing local charts for $\mathfrak{D}_\mu(M)$ 
is more complicated than for $\mathfrak{D}(M)$. 
It will  be sufficient to describe such a chart near the identity. 
Let $\varphi_e$ be the diffeomorphism 
from a neighbourhood $\mathcal{U}$ of the zero section in $\mathcal{C}^\infty(M,TM)$ 
to a neighbourhood of $e$ in $\mathfrak{D}(M)$ defined in \eqref{eq:phi-chart}.
Consider the pullback map 
$$ 
\phi_\mu: u \mapsto \varphi_e(u)^\ast \mu/\mu 
\qquad {\rm for}\quad
u \in \mathcal{U} \subset \mathcal{C}^\infty(TM). 
$$ 
Writing it out in local charts and using the estimates on compositions and products 
shows that $\phi_\mu$ is a tame nonlinear PDO of order $1$ in $u$ that maps smoothly to $\mathcal{C}^\infty(M)$ 
and whose differential at zero is 
\begin{equation} \label{eq: dphi} 
w \mapsto d\phi_\mu(0)(w) 
= 
\mathcal{L}_w \mu / \mu 
= 
\mathrm{div}_\mu\, w  \,,
\end{equation} 
where $\mathrm{div}_\mu$ is the divergence operator of the Riemannian metric on $M$. 
Using the Helmholtz-Hodge decomposition\index{Helmholtz-Hodge decomposition} of vector fields into $L^2$-orthogonal components 
$\mathcal{C}^\infty(TM) = \mathrm{div}_\mu^{-1}(0) \oplus \nabla \mathcal{C}^\infty(M)$ 
(see e.g., \cite{Ladyzhenskaya69} or \cite{Morrey66}) 
and the natural identification\footnote{For example, using a special case of de Rham's theorem which states that 
a closed form is exact if all of its periods vanish, see e.g., \cite{Warner71}.} 
of $\nabla\mathcal{C}^\infty(M)$ with the closed subspace $\mathcal{C}^\infty_0(M)$ of mean-zero functions on $M$ 
we now define a map 
\begin{align*} 
\Phi_\mu: 
u \mapsto \big( v, \phi_\mu(u) - 1 \big) 
\qquad {\rm for }\quad
u \in \mathcal{U} \subset \mathcal{C}^\infty(TM) 
\end{align*} 
on the product of Fr\'echet spaces $\mathrm{div}_\mu^{-1}(0) \oplus \mathcal{C}^\infty_0(M)$. 
The task of constructing a local chart at $e$ for $\mathfrak{D}_\mu(M)$ 
reduces now to showing that $\Phi_\mu$ is invertible near the zero section. 
Its derivative at $u = v + \nabla f$ is readily computed to be 
$$ 
d\Phi_\mu(u) (w, g) 
= 
\big( w, d\phi_\mu(v+\nabla f) (w+\nabla g) \big) 
\qquad 
w \in \mathrm{div}_\mu^{-1}(0), \; g \in \mathcal{C}^\infty_0(M). 
$$ 
In particular, if $v=0$ and $f=0$ then 
$d\Phi_\mu(0)(g,w) = (w, \mathrm{div}_\mu (w+\nabla g) ) = ( w, \Delta g )$ 
so that its second component is an elliptic operator, which remains elliptic 
under suitably small perturbations of $v$ and $f$. 
It follows that the derivative of $\Phi_\mu$ is a linear isomorphism with tame inverse at all points near $u=0$ 
and, consequently, it is locally invertible by the Nash-Moser-Hamilton theorem (Proposition~\ref{prop:NMH}). 
In order to obtain a chart near $e \in \mathfrak{D}_\mu(M)$ it suffices now 
to compose $\Phi_\mu^{-1}$ with the local chart $\varphi_e$ for $\mathfrak{D}(M)$ 
and set 
$\varphi_e\circ \Phi_\mu^{-1}\big\vert_{\mathcal{V}\cap (0\oplus\mathrm{div}^{-1}(0))}$ 
where $\mathcal{V}$ is some neighbourhood of $0$ on which $\Phi_\mu$ is a diffeomorphism. 

Finally, observe that $\mathfrak{D}_\mu(M)$ is a subgroup whose group operations 
enjoy the same regularity properties as those of the ambient group $\mathfrak{D}(M)$. 
This shows that $\mathfrak{D}_\mu(M)$ is a tame Fr\'echet Lie group. 
\end{proof} 

\begin{remark} 
A subtle but important point to keep in mind when constructing 
smooth tame families of elliptic inverses is 
that to make it work one needs 
to make a judicious choice of an appropriate grading 
(e.g., using the $\mathcal{C}^{k,\alpha}$ H\"older or the $H^s$ Sobolev norms) 
for the Fr\'echet spaces involved in the argument. 
The reason for this is that elliptic estimates do not hold in certain functional settings, 
such as e.g., the uniform $\mathcal{C}^k$ spaces. 
\end{remark} 
\begin{remark} 
As remarked in Section~\ref{subsec:mm}, slightly modifying the procedure described there 
it is possible to equip $\mathfrak{D}(M)$ with the structure of a Banach (or Hilbert) manifold,\index{Banach manifold} 
e.g., 
by enlarging it to Sobolev $H^s$ class diffeomorphisms $\mathfrak{D}^s(M)$ with $s>n/2+1$ 
or H\"older $\mathcal{C}^{1,\alpha}$ diffeomorphisms $\mathfrak{D}^{1,\alpha}(M)$ with $0<\alpha<1$. 
Both options are very convenient from the point of view of analysis\footnote{See Appendix \ref{Banach}.} 
- this is, once again, in contrast to the $\mathcal{C}^k$ case where this procedure seems to fail, 
as it depends crucially on the Hodge decomposition theorem. 
\end{remark} 
\begin{remark} 
The theory of (infinite dimensional) Banach Lie groups is well developed and goes back to Birkhoff \cite{Birkhoff38}, 
see also \cite{Bourbaki75}. 
However, repeating the construction in the proof of Proposition~\ref{prop:Diff} immediately runs into serious obstacles, 
since 
compositions with diffeomorphisms on the left (i.e., left translation maps in the group), 
as well as 
the Lie bracket of vector fields on $M$ (i.e., the commutator in the "Lie algebra"), 
lose derivatives with respect to both topologies. 
Consequently, neither $\mathfrak{D}^s(M)$ nor $\mathfrak{D}^{1,\alpha}(M)$ 
is a Banach Lie group in Birkhoff's sense. 
Of course, this problem disappears if the modelling space is a Fr\'echet space of smooth functions. 
\end{remark} 
\begin{example} \label{ex:exp-Lie} 
Consider the group $\mathfrak{D}(\mathbb{T})$ of smooth diffeomorphisms of the unit circle. 
Any vector field $u$ on the circle, viewed as an element of the Lie algebra 
$T_e\mathfrak{D}(\mathbb{T})$, 
gives rise to a one-parameter subgroup $\phi_t(x)$ obtained as a solution of 
the corresponding flow equation $d\phi_t/dt = u\circ\phi_t$ subject to the initial condition $\phi_0(x) = x$. 
The Lie group exponential map at the identity element $e$ in $\mathfrak{D}(\mathbb{T})$ 
is now defined in complete analogy to the finitie-dimensional case 
as a "time-one map" by the formula\index{group exponential}
$$ 
\exp_e: T_e\mathfrak{D}(\mathbb{T}) \to \mathfrak{D}(\mathbb{T})\,, 
\qquad 
\exp_e u = \phi_1.
$$ 
It is not hard to see that this map is continuous in the Fr\'echet (or any sufficiently strong Banach) topology 
on $\mathfrak{D}(\mathbb{T})$, 
but 
it is not of class $\mathcal{C}^1$. 
In fact, observe that even though its derivative at $t=0$ is $d\exp_e(0) = \mathrm{id}$, 
the map fails to be invertible in any neighbourhood of the identity $e$.\footnote{Note that in Proposition~\ref{prop:NMH} 
invertibility of the derivative is assumed to hold in a small neigbourhood rather than at a single point.} 
This ``bad" behaviour of the (Lie group) exponential map on $\mathfrak{D}(\mathbb{T})$ 
stands in sharp contrast to the case of the classical, as well as Banach, Lie groups such as loop groups. 
We refer to \cite{Omori97}, \cite{Hamilton82} or \cite{Grabowski88} for further details. 
\end{example} 
%

\section{The quotient space of probability densities} \label{sec:moser_fibration}
A fixed volume form on a compact $n$-manifold $M$ can be used to define a natural fibration of $\mathfrak{D}(M)$ 
over the Fr\'echet manifold of smooth positive normalized measures on $M$. 
To describe the latter, recall that any volume form $\nu$ defines a Borel measure, 
which in any coordinate chart $\big( U, x = (x_1, \dots, x_n) \big)$ on $M$ has the form $d\nu = \rho_U(x) dx$ 
where $\rho_U(x)$ is a positive $\mathcal{C}^\infty$ function and $dx = dx_1 \dots dx_n$ is the Lebesgue measure 
on $\mathbb{R}^n$. 
Its expressions in different charts are related by Jacobian functions 
$$ 
\rho_U(x) = |\det(\partial \eta^i/\partial x^j) | \rho_V\circ\eta(x) \,,
\qquad 
\eta \in \mathfrak{D}(U,V) \,,
$$ 
which, in the geometric language, represent the cocycle of transition functions of the corresponding line bundle 
of volume forms over $M$. 
Cross-sections of this vector bundle are called \emph{densities} 
and can be identified with smooth measures on $M$. 
If the manifold carries a Riemannian metric $g$ then there is a canonical choice of a positive density on $M$, 
which in a local chart is given by $\sqrt{\det{g}}(x)$. 
In this case the corresponding volume form $\mu$ can be used as a fixed reference measure on $M$. 

We set 
\begin{equation} \label{eq:Dens} 
\mathfrak{Dens}(M) 
= 
\big\{ \nu \in \Omega^n(M)\mid \nu > 0, \; \int_M d\nu = 1 \big\} \,,
\end{equation} 
\index{probability densities}
where $\Omega^n(M)$ denotes the tame Fr\'echet space of smooth $n$-forms on $M$ 
(i.e., the cross-sections of the one-dimensional vector bundle $\Lambda^n T^\ast M$). 
This set is clearly open and connected. 
In fact, if $\nu$ and $\lambda$ are in $\mathfrak{Dens}(M)$ then 
so is their convex combination $t \nu + (1-t)\lambda$ where $0 \leq t \leq 1$. 
Of course, we can also think of $\mathfrak{Dens}(M)$ as an open convex subset of positive density functions 
$\rho >0$ in the Fr\'echet manifold of $\mathcal{C}^\infty$ functions with average value $1$ on $M$, 
so that we can write $\nu = \rho \mu$. 
Note that, we denote volume forms by $\nu,\lambda,\mu,\ldots$ while when we integrate the corresponding measures they appear as $\int d\nu, \int d\lambda, \ldots$.

The tangent space at any $\nu \in \mathfrak{Dens}(M)$ is 
\begin{equation} \label{eq:TDens} 
T_\nu\mathfrak{Dens}(M) 
= 
\big\{ \beta \in \Omega^n M\mid  \int_M d\beta = 0 \big\} \,,
\end{equation} 
which follows at once by differentiating the condition in \eqref{eq:Dens}. 

\begin{figure} 
\begin{center}
 \includegraphics{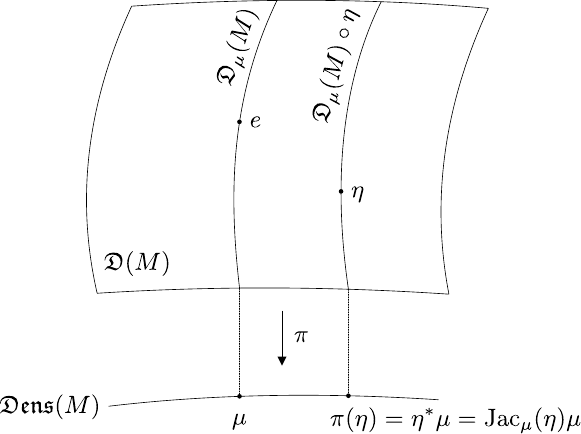}
\caption{Moser's fibration of diffeomorphisms $\mathfrak{D}(M)$ over smooth probability densities $\mathfrak{Dens}(M)$.
The identity fiber $\mathfrak{D}_\mu(M)$ is determined by a reference density $\mu \in \mathfrak{Dens}(M)$. 
The fiber structure provides an infinite dimensional principal bundle in the tame Fréchet category (cf.~Proposition~\ref{prop:PB}).
}\label{fig:Moser's fibration}
\end{center}
\end{figure}

Moser \cite{Moser65} proved that 
for any two volume forms $\nu$ and $\lambda$ on a compact manifold $M$ 
with $\int_M dv = \int_M d\lambda$ there is a smooth diffeomorphism $\xi$ 
such that $\xi^\ast \nu = \lambda$. 
He did this by constructing a map $\chi_\mu$ 
from the space of densities $\mathfrak{Dens}(M)$ to the diffeomorphism group $\mathfrak{D}(M)$ 
which is a smooth inverse of the pullback map $\pi: \xi \to \xi^\ast \mu$. 
Subsequently, Ebin and Marsden \cite{EbinMarsden70} observed that 
$\mathfrak{D}(M)$ is diffeomorphic\footnote{In particular, $\mathfrak{D}_\mu(M)$ is a deformation retract 
of $\mathfrak{D}(M)$, since $\mathfrak{Dens}(M)$ is convex.} 
to the product $\mathfrak{D}_\mu(M) \times \mathfrak{Dens}(M)$, 
as verified directly upon setting 
$(\eta, \nu) \mapsto \xi = \eta\circ\chi_\mu(\nu)$ 
with inverse 
$\xi \mapsto \big( \xi\circ\chi_\mu(\xi^\ast\mu)^{-1}, \xi^\ast\mu \big)$. 
Reformulating these statements somewhat gives 
\begin{proposition} \label{prop:PB} 
The group of diffeomorphisms $\mathfrak{D}(M)$ is a smooth tame Fr\'echet principal bundle 
over $\mathfrak{Dens}(M)$ with fiber $\mathfrak{D}_\mu(M)$ and projection given by 
the pullback map $\pi(\xi) = \xi^\ast\mu$. 
In particular, there is a short exact sequence of tame Fr\'echet spaces 
$$ 
0 \to T_e\mathfrak{D}_\mu(M) \to T_e\mathfrak{D}(M) \to T_\mu \mathfrak{Dens}(M) \to 0 \, .
$$ 
\index{Moser's principal bundle}
\end{proposition} 
\begin{proof} 
This is essentially an immediate consequence of the bundle structure theorem 
since Moser's map $\chi_e$ defines a (global) cross-section of $\mathfrak{D}_\mu(M)$ in $\mathfrak{D}(M)$, 
cf. Steenrod \cite{Steenrod51}. 
Just observe that the action of the diffeomorphism group $\mathfrak{D}(M)$ on the space of densities $\mathfrak{Dens}(M)$ is tame and (globally) transitive (this is essentially Moser's result) 
and $\mathfrak{D}_\mu(M)$ is a closed tame Frechet subgroup by Proposition~\ref{prop:D-mu}, which in turn is a consequence of the Nash-Moser theorem. 
For more details, see Hamilton's paper \cite{Hamilton82}. 
\end{proof} 

The fibers of the bundle $\pi: \mathfrak{D}(M) \to \mathfrak{Dens}(M)$ 
are precisely the \emph{right cosets}\index{coset, right} of the subgroup $\mathfrak{D}_\mu(M)$ in $\mathfrak{D}(M)$, 
that is
$$ 
\mathfrak{Dens}(M) 
\simeq 
\mathfrak{D}_\mu(M)\backslash\mathfrak{D}(M) 
= 
\big\{ 
[\xi] = \mathfrak{D}_\mu(M)\circ\xi\mid \xi \in \mathfrak{D}(M) 
\big\}\,,
$$ 
as illustrated in Figure~\ref{fig:Moser's fibration}. The action of diffeomorphisms on densities is by means of 
{\it pull-backs}. 
This fact will play an important role below.

\begin{remark}[Right cosets versus left cosets] 
In a similar manner one can consider the quotient space of densities $\Pi: \mathfrak{D}(M) \to \mathfrak{Dens}(M)$ as a space of \emph{left cosets}\index{coset, left} 
and identify 
$$ 
\mathfrak{Dens}(M) 
\simeq 
\mathfrak{D}(M)/\mathfrak{D}_\mu(M) 
= 
\big\{ 
[\zeta] = \zeta\circ\mathfrak{D}_\mu(M)\mid\zeta \in \mathfrak{D}(M) 
\big\}. 
$$ 
Here the action of diffeomorphisms on densities is by means of {\it push-forwards}. 
It is clear that in general a right cosets need not be a left coset. 
Nonetheless many constructions and arguments 
concerning the former apply, mutatis mutandis, to the latter as well. (By taking the inverse on the group 
$\mathfrak{D}(M)$ one takes  left cosets to right ones.)

It is interesting to note that 
identifying $\mathfrak{Dens}(M)$ with \emph{left} cosets 
provides the setting for the Wasserstein (or Kantorovich-Rubinstein) geometry of Optimal Transport. 
See Section~\ref{app2} below. 
\label{rem:right-left} 
\end{remark} 
%



%
%
%

\chapter[Riemannian structures]{Riemannian structures on infinite dimensional Lie groups} \label{sect:Riemannian}

In this chapter we are primarily interested in those pre-Riemannian metrics on Fr\'echet Lie groups $\mathfrak{G}$ 
and their quotient spaces that are \emph{invariant}, i.e., 
for which $\mathfrak{G}$ acts by isometries. 
Metrics invariant under actions given by right- (resp. left-) multiplications are called 
\emph{right-} (resp. \emph{left-}) \emph{invariant}; 
those invariant under both right and left actions are called \emph{bi-invariant}. \index{Riemannian metric}

Thus, given an inner product $\langle \cdot, \cdot \rangle_e$ on the Lie algebra $T_e\mathfrak{G}$, 
a right-invariant metric on the group $\mathfrak{G}$ can be defined simply by setting 
\begin{equation} \label{eq:ri-met} 
\langle V, W \rangle_g 
= 
\langle dR_{g^{-1}} V, dR_{g^{-1}} W \rangle_e 
\end{equation} 
for any vectors $V, W \in T_g\mathfrak{G}$ and any $g \in \mathfrak{G}$.

\section{The Euler-Arnold equations} 
\label{subsub:EAeq} 
In 1960's Arnold \cite{Arnold66} proposed a differential-geometric framework 
to study the Euler equations of ideal hydrodynamics. 
It is based on the observation that motions of an ideal (that is, incompressible and non-viscous) fluid 
in a bounded domain $M$ 
trace out curves in the group $\mathfrak{D}_\mu(M)$ of volume-preserving diffeomorphisms of $M$ 
which correspond to geodesics of the right-invariant pre-Riemannian metric 
defined by the kinetic energy of the fluid. 
This approach is very general and applies to numerous partial differential equations of interest 
in mathematical physics and geometry. 
Such equations arise within this framework through a general reduction procedure 
which starts with a given geodesic system on the group 
to produce a dynamical system on the tangent space at the identity. 
In this section, we describe Arnold's framework for general Lie groups and homogeneous spaces. 

Let $\mathfrak{G}$ be a (finite or infinite dimensional Banach or Fr\'echet) Lie group 
which carries a right-invariant pre-Riemannian metric $g = \langle \cdot, \cdot \rangle$ 
induced by an inner product on its Lie algebra $T_e\mathfrak{G}$ as in \eqref{eq:ri-met}. 
The \emph{Euler-Arnold equation}\index{Euler-Arnold equation} on the Lie algebra associated with the geodesic flow of 
$g$ has the form 
\begin{equation} \label{eq:Euler-Arnold} 
u_t = - \mathrm{ad}^\top_u u 
\end{equation} 
where $u(t)$ is a curve in $T_e\mathfrak{G}$ and the bilinear operator $\mathrm{ad}^\top$ 
on the right-hand side is an operator on $T_e\mathfrak{G}$ defined by 
\begin{equation} \label{eq:co-ad} 
\langle \mathrm{ad}_v^\top u , w \rangle_e = \langle u, \mathrm{ad}_v w \rangle_e 
\qquad 
\text{for any} 
\;\; 
u, v, w \in T_e\mathfrak{G} 
\end{equation} 
called the \emph{transposed operator}.\index{adjoint action} 

When equation \eqref{eq:Euler-Arnold} is augmented by an initial condition 
\begin{equation} \label{eq:EAic} 
u(0) = u_0 
\end{equation} 
then solutions of the resulting Cauchy problem \eqref{eq:Euler-Arnold}-\eqref{eq:EAic} 
describe evolution in the Lie algebra of the dynamical system 
$t \mapsto u(t) = \dot{g}(t) {\cdot} g^{-1}(t)$ 
obtained by right-translating the velocity field of the corresponding geodesic 
$g(t)$ in the group $\mathfrak{G}$ starting at the identity $e$ in the direction $\dot{g}(0)=u_0$. 

Observe that, conversely, if in turn $u(t)$ is known then the geodesic can be obtained 
 by solving the Cauchy problem for the \emph{flow equation}, namely 
$$ 
\frac{d g(t)}{dt} = d R_{g(t)} (e) u(t) \,,
\qquad 
g(0) = e. 
$$ 
\begin{remark} 
The Cauchy problem \eqref{eq:Euler-Arnold}-\eqref{eq:EAic} 
can be rewritten in the form 
$$ 
\frac{d}{dt} \big( \mathrm{Ad}_{g(t)}^\top u \big) = 0 \,,
\qquad 
u(0) = u_0 \,,
$$ 
where 
$$ 
\langle \mathrm{Ad}_g^\top u, v \rangle_e 
= 
\langle u, \mathrm{Ad}_g v \rangle_e 
\qquad 
\text{for any} 
\;\; 
v \in T_e\mathfrak{G} 
\;\, 
\text{and} 
\;\, 
g \in \mathfrak{G} \,,
$$ 
which immediately yields a \emph{conservation law} 
$$ 
\mathrm{Ad}_{g(t)}^\top u(t) = u_0. 
$$ 
This last equation expresses the fact that solutions $u(t)$ of the Euler-Arnold equation are confined to 
one and the same orbit during the evolution. 
\end{remark} 
\begin{remark} \label{rem:EA-LWP} 
In the infinite dimensional examples, whenever the tame Frechet framework described in the preceding chapters is applicable, local solutions of the Cauchy problem for the Euler-Arnold equations can be obtained from the Nash-Moser-Hamilton theorem.\index{Nash-Moser-Hamilton theorem} 
To that end one needs to first establish that the corresponding Riemannian exponential map is a local diffeomorphism of the Frechet spaces thus yielding geodesics (defined at least for short times) and next right-translate the velocities of these geodesics to the tangent space at the identity. 
The required regularity assumptions need to be carefully verified in each such case. 
See Remark \ref{rem:225} above. 
\end{remark} 
\begin{example}[The Rigid Body] 
In the important special case when $\mathfrak{G} = SO(3)$ this procedure (but for a left invariant metric) yields 
the classical Euler equations describing rotations of a rigid body in the internal coordinates of the body. 
In vector notation they have the form 
$$ \frac{d}{dt}P = P \times \Omega$$ 
where $P$ is the vector of angular momentum and $\Omega$ is the vector of angular velocity 
- the two are related by the so-called inertia operator of the system. 
\end{example} 
\begin{example}[The Euler equations of ideal hydrodynamics] \label{ex:ideal_euler}
\index{ideal hydrodynamics}
\index{volume-preserving diffeomorphisms}
Another special case involves the group of volume-preserving diffeomorphisms 
$\mathfrak{G} = \mathfrak{D}_\mu(M)$ of a compact Riemannian manifold $(M,g)$ 
--- see Section~\ref{subsec:Diff} below. 
Its Lie algebra $\mathfrak{g} = T_e\mathfrak{D}_\mu(M)$ is the space $\mathfrak{X}_\mu(M)$ of divergence-free vector fields on $M$. 
This group can be equipped with a right-invariant metric which is essentially the fluid's kinetic energy 
and which at the identity diffeomorphism is given by the $L^2$ inner product of vector fields $u,v\in \mathfrak{X}_\mu(M)$ on $M$:
\begin{equation}\label{eq:arnold_metric_L2}
    \langle u, v\rangle_e = \int_M g(u, v) \, d\mu .    
\end{equation} 
In this case the Euler-Arnold equation \eqref{eq:Euler-Arnold} becomes 
the {\it Euler equations of ideal hydrodynamics} \index{Euler hydrodynamics equations} 
$$ 
u_t + \nabla_u u = -\nabla p \,,
\qquad \mathrm{div}\, u = 0 \,,
$$ 
where $u$ is the vector field on $M$ representing the velocity field and $p$ is the function on $M$ 
representing the pressure in the fluid, see \cite{Arnold66}. 
\end{example} 
\begin{example}[Integrable systems and circle diffeomorphisms] 
If the group of circle diffeomorphisms $\mathfrak{G}=\mathfrak{D}(\mathbb{T})$ is equipped with the right-invariant metric 
generated by the $L^2$ inner product 
$$ 
\langle u, v \rangle_{L^2} = \int_{\mathbb{T}} u v \, dx 
\qquad {\rm for ~~any}\quad
u, v \in \mathfrak{X}(\mathbb T)=T_e\mathfrak{D}(\mathbb{T}) \,,
$$ 
then the Euler-Arnold equation is the (scaled) inviscid Burgers equation 
$$ 
u_t + 3 uu_x = 0. 
$$ 
If the metric is generated by the Sobolev $H^1$ inner product 
\begin{equation}\label{eq:H1metric}
\langle u, v \rangle_{H^1} 
= 
\int_{\mathbb{T}} (uv + u_x v_x) \, dx 
\end{equation}
then the Euler-Arnold equation yields the Camassa-Holm equation 
\begin{equation}\label{eq:CH0}
u_t - u_{txx} + 3uu_x - 2u_xu_{xx} - uu_{xxx} = 0, 
\end{equation}
see details in Section~\ref{sub:generalCH}.
Both of these equations are well-known examples of infinite dimensional completely integrable systems\index{integrable system} 
in that they are bi-hamiltonian, possess infinitely many conserved integrals etc., 
for more details see \cite{KhesinMisiolek03}; see also Section~\ref{sect:Hamiltonian}. 
\end{example} 

\begin{remark}
Many other conservative dynamical systems in mathematical physics also describe geodesic flows on appropriate Lie groups. 
In Figure \ref{fig:table} we list several examples of such systems to demonstrate the range of applications of this approach. 
The choice of a group $\mathfrak{G}$ (column~1) and an energy metric  $\langle \cdot, \cdot \rangle$ (column~2) defines the corresponding Euler equations (column~3). (Note that the $L^2$ and $H^s$ in the column~2 refer to various Sobolev inner products on vector fields in the corresponding Lie algebra.) This list is by no means complete, and we refer to  \cite{ArnoldKhesin98} for more details.

\begin{figure}
\small{
\begin{center}
\begin{tabular}{c|c|c}
Group ~~ $\mathfrak{G}$  & Metric ~~ $\langle \cdot, \cdot \rangle$ ~& Equation  \\
\hline
\hline
&&\\
$SO(3)$ & $\langle\omega,A\omega\rangle$ & {\rm Euler~top} \\ 
$E(3)=SO(3)\ltimes\mathbb R^3$ &   quadratic forms & Kirchhoff equation for a body in a fluid \\ 
 
$SO(n)$ & {\rm ~Manakov's~ metrics~} &   $ n$\text{-dimensional top} \\ 
$\mathfrak{D}(\mathbb{T}) $    &$ L^2$ & \text{Hopf~(or,~inviscid~Burgers) equation} \\ 
$\mathfrak{D}(\mathbb{T})  $    &$ \dot H^{1/2}$ & \text{Constantin-Lax-Majda-type equation} \\ 
$\mathfrak{D}(\mathbb{T})$    & $H^1$ & \text{Camassa--Holm equation}  \\ 
{\rm Virasoro}    & $L^2$ & \text{KdV~equation} \\ 
{\rm Virasoro}    & $H^1$ & \text{Camassa--Holm equation}  \\ 
{\rm Virasoro}    & ${\dot H}^1 $ & \text{Hunter--Saxton equation} \\ 
$\mathfrak{D}_\mu(M) $ & $L^2$ & \text{Euler~ideal~fluid} \\ 
$\mathfrak{D}_\mu(M)$ & $H^1 $ & \text{averaged~Euler~flow} \\ 
$\mathfrak{D}_\omega(M)$ & $L^2$ & \text{symplectic fluid} \\ 
$\mathfrak{D}(M)$ & $H^k$ & \text{EPDiff equation} \\ 
$\mathfrak{D}_\mu(M)\ltimes \mathfrak{X}^*_\mu(M)$ & $L^2\oplus L^2$ & {\rm Magnetohydrodynamics} \\ 
$\mathcal{C}^\infty(\mathbb T, SO(3)) $ & $ H^{-1}$ & \text{Heisenberg~chain~equation} 
\end{tabular}
\end{center}
\caption{Euler--Arnold equations related to various Lie groups and metrics.\index{Euler-Arnold equation}}\label{fig:table}
}
\end{figure}

\end{remark}

\begin{remark}[Euler-Arnold equation on quotient spaces] 
More generally, let $\mathfrak{G}$ be a Fr\'echet Lie group equipped with a right-invariant metric as above 
and let $\mathfrak{H}$ be a closed subgroup. 
The metric on $\mathfrak{G}$ descends to an invariant (under the right action of $\mathfrak{G}$) metric 
on the quotient $\mathfrak{H}\backslash\mathfrak{G}$ if and only if its projection onto the orthogonal complement 
$T_e^\perp \mathfrak{H} \subset T_e\mathfrak{G}$ is bi-invariant with respect to the action of $\mathfrak{H}$. 
In particular, if the metric on $\mathfrak{G}$ is degenerate along the subgroup $\mathfrak{H}$ then 
this condition reduces to the metric bi-invariance with respect to the $\mathfrak{H}$ action, 
see e.g., \cite{KhesinMisiolek03} and Section \ref{sect:quotient} below. 
In this case the corresponding Euler-Arnold equation is defined as before as long as the metric 
on the quotient $\mathfrak{H} \backslash\mathfrak{G}$ is nondegenerate. 
\label{rem:EA-quotient} 
\end{remark}

\section{Quotient spaces and Riemannian submersions} \label{sect:quotient}
Let $\mathfrak{F}$ be a smooth fiber bundle over a Fr\'echet manifold $\mathfrak{M}$ 
with projection $\pi$ and fibers which are modelled on a Fr\'echet space $\mathfrak{Y}$. 
Assume that both $\mathfrak{F}$ and $\mathfrak{M}$ carry (possibly weak) Riemannian metrics 
and let $\mathfrak{Y}^\perp$ denote the orthogonal complement. The projection
$\pi: \mathfrak{F} \to \mathfrak{M}$ defines an (infinite dimensional) \emph{Riemannian submersion}\index{Riemannian submersion} 
if $d\pi |_{\mathfrak{Y}^\perp}$ is an isometry at each point of $\mathfrak{F}$. 

In particular, if $\mathfrak{H}$ is a closed subgroup of a Fr\'echet Lie group $\mathfrak{G}$ 
equipped with a right-invariant (possibly weak) Riemannian metric 
then the following general result characterizes those metrics that descend to the base manifold 
of right cosets. 
\begin{proposition} \label{prop:descend} 
A right-invariant metric $\langle \cdot, \cdot \rangle$ on the group $\mathfrak{G}$ descends to a right-invariant metric on 
the quotient space $\mathfrak{H}\backslash\mathfrak{G}$ 
if and only if 
the inner product 
$\langle \cdot, \cdot \rangle_e$ restricted to the orthogonal complement $T_e^\perp\mathfrak{H}$ 
is bi-invariant with respect to the action of $\mathfrak{H}$, that is 
\begin{equation} \label{eq:dsc} 
\langle u, \mathrm{ad}_w v \rangle_e + \langle \mathrm{ad}_w u, v \rangle_e = 0 
\end{equation} 
for any $u, v \in T_e^\perp\mathfrak{H}$ and any $w \in T_e\mathfrak{H}$. 
\end{proposition} 
\begin{proof} 
The proof repeats, with obvious modifications, the arguments in the finite dimensional case, 
see \cite{KhesinLenellsMisiolekPreston13}. See also \cite{CheegerEbin75}. 
\end{proof} 
\begin{example}[Circle diffeomorphisms and the periodic Hunter-Saxton equation] 
\index{Hunter-Saxton equation}
Let $\mathfrak{G}$ be the group of circle diffeomorphisms $\mathfrak{D}(\mathbb{T})$ from Example \ref{ex:circle-dens} 
and let $\mathfrak{H} \simeq \mathbb{T}$ be the subgroup of rotations. 
The corresponding Lie algebra $\mathfrak{g}=T_e\mathfrak{G}$ is the Lie algebra 
$\mathfrak{X}(\mathbb{T})$ of vector fields on the circle with the standard commutator.
Equip $\mathfrak{G}$ with a right-invariant metric which at the identity is given by 
the (homogeneous) Sobolev $H^1$ inner product 
\begin{equation} \label{eq:H1met-S1} 
\langle u, v \rangle_{\dot{H}^1} = \int_{\mathbb{T}} u_x v_x  dx 
\qquad 
u, v \in \mathfrak{X}(\mathbb{T}) 
\end{equation} 
The tangent space to the quotient space $\mathfrak{Dens}(\mathbb{T})$ at the identity coset $[e]$ 
can be identified with the space of smooth periodic mean-zero functions 
and the Euler-Arnold equation for the geodesic flow of the metric  \eqref{eq:H1met-S1} (right-translated to the tangent space at the identity) 
is the completely integrable Hunter-Saxton equation 
\begin{equation} \label{eq:HS} 
u_{txx} + 2u_x u_{xx} + u u_{xxx} = 0 \,,
\end{equation} 
see \cite{KhesinMisiolek03}. 
We shall considerably generalize this example later on. 
\label{ex:HS} 
\end{example} 

\begin{remark}
    The $\dot H^1$ inner product \eqref{eq:H1met-S1} in the example above is degenerate in the direction of infinitesimal rigid rotations, i.e., along the Killing vector field of $\mathbb T$.
    In this 1-dimensional case the issue is minor, but in the generalization to higher dimensions the degenerate subspace is at least as large as the nondegenerate one.
    It is, however, possible to make the inner product fully nondegenerate while still satisfying the condition in Proposition~\ref{prop:descend}: this is the information metric on diffeomorphisms, discussed in Chapter~\ref{sec:FR-met} below.
\end{remark}

\section{Hamiltonian formulation of the Euler-Arnold equation} \label{sect:hamilt}
Before we describe one example in full detail, it is useful to give an alternative formula for the  Euler-Arnold equation \eqref{eq:Euler-Arnold}.
Namely, instead of defining  it in the Lie algebra $ \mathfrak{g}:=T_e\mathfrak{G}$ it is more convenient to define it on the corresponding dual space $\mathfrak{g}^*:=T_e^*\mathfrak{G}$. Denote the pairing of elements of  $\mathfrak{g}^*$ and $\mathfrak{g}$ by 
$\llangle m, u\rrangle$ for $m\in \mathfrak{g}^*$ and $u\in \mathfrak{g}$. Now for any $v\in \mathfrak{g}$ 
 the {\it  coadjoint operator} 
$\mathrm{ad}^*_v: \mathfrak{g}^*\to \mathfrak{g}^*$ is the operator on the dual of the Lie algebra, defined by the relation
\begin{equation} \label{eq:dualcoad} 
\llangle \mathrm{ad}^*_v m, u \rrangle = \llangle m, \mathrm{ad}_v u \rrangle 
\qquad 
\text{for any} 
\;\; 
u, v \in \mathfrak{g} \text{ and }\; m\in \mathfrak{g}^*\,.
\end{equation} 

Furthermore, the inner product $\langle~,~\rangle$ on $  \mathfrak{g}$ can be equivalently defined by a nondegenerate 
self-adjoint {\it inertia operator}\index{inertia operator} $A\colon  \mathfrak{g} \to \mathfrak{g}^*$ in the natural way: 
$$
\llangle Au, v\rrangle = \langle u,v\rangle
\quad\text{for any pair of vectors $u, v \in \mathfrak{g}$.}
$$ 
This allows one to lift the  equation \eqref{eq:Euler-Arnold}
from the Lie algebra $\mathfrak{g}= T_e\mathfrak{G} $ to its dual $ \mathfrak{g}^*=T_e^*\mathfrak{G} $ by means of the inertia operator $A$. 
Indeed, this lift gives rise to the {\it Euler-Arnold  equation}\index{Euler-Arnold equation} 
\begin{equation}\label{euler2} 
    m_t= -{\rm ad}^*_{A^{-1}m}m\,, 
 \end{equation} 
expressed in the \emph{momentum variable} $m=Au\in  \mathfrak{g}^*$. (The minus sign is related to the right-invariant metric on the group.)
Note that in the form \eqref{euler2} the only information needed is the operators ${\rm ad}^*$ and $A$ on the (dual) Lie algebra, rather than on the corresponding group.

\begin{figure}
\includegraphics{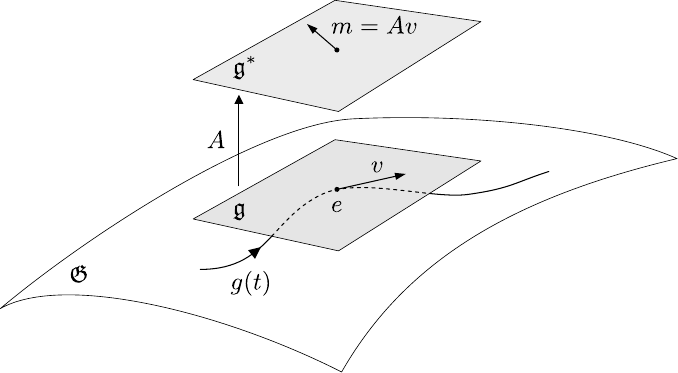} 
\caption{
The geometry of Euler-Arnold equations.
The vector $v$ in the Lie algebra $\mathfrak{g}$ traces the evolution of the velocity vector of a geodesic $g(t)$ on the group $\mathfrak{G}$. 
The inertia operator $A$ sends $v$ to an element $m$ in the dual space $\mathfrak{g}^*$.}\label{Fig:geodesic} 
\end{figure} 

An advantage of lifting is that the Euler-Arnold equation \eqref{euler2} is Hamiltonian on the dual $ \mathfrak{g}^*=T_e^*\mathfrak{G}$
with respect to the natural Poisson structure on $\mathcal{C}^\infty(\mathfrak{g}^*)$ (i.e., the \emph{Lie-Poisson structure}) given by
\begin{equation*}
    \{F,G \}(m) = \llangle m, \operatorname{ad}_{dF(m)}dG(m)\rrangle
\end{equation*}
and the Hamiltonian function given by the kinetic energy expressed in terms of $m$, namely 
\begin{equation*}
    H(m)=\frac 12 \llangle m, {A^{-1}m}\rrangle=\frac 12 \llangle Au, u\rrangle=E(u).    
\end{equation*}
Moreover, this approach allows one to use the same equation for different metrics on the group, which differ only by their inertial operators $A$. 

\subsection{Reminder on Hamiltonian formalism and integrability}\label{sect:Hamiltonian}
%
Let $M$ be a $2n$-dimensional manifold $M$ equipped with a nondegenerate closed $2$-form $\omega \in \Omega^2 (M)$ called a \emph{symplectic structure}.\index{symplectic manifold} Given a smooth function $H$ on $M$ one defines the \emph{Hamiltonian vector field} $v$ by means of the relation $\iota_{v_H} \omega=dH$, i.e. $\omega(v_H, u)=dH(u)$ for any test vector field $u$ on $M$. Here $H$ is called the  \emph{Hamiltonian function}\index{Hamiltonian function} of $v$ and a \emph{Hamiltonian system} is  a triple $(M, \omega, H)$. 

In this setting the Fr\'echet space $\mathcal{C}^\infty(M)$ of smooth functions on $M$ 
acquires a structure of an associative commutative algebra with a {\it Poisson bracket} given by 
$\{ f, g \} := \omega(v_f, v_g)$. 
The equations of motion of a Hamiltonian system assume a simple and elegant form 
\begin{equation} \label{eq:ham-eq} 
\dot f = \{ H, f \} 
\end{equation} 
for any function $f \in \mathcal{C}^\infty(M)$ and where $\dot f$ denotes derivative with respect to time.

Any smooth function $f$ such that $\{H,f \} =0$ necessarily assumes 
a constant value along any orbit of $v_H$ and it 
is called a \emph{first integral} (or a \emph{constant of motion} 
of the Hamiltonian system. 

\medskip

In fact, one can regard the Poisson manifolds as more basic objects, while symplectic manifolds are 
their submanifolds on which the Poisson structure is nondegenerate. Namely, a \textit{Poisson bracket} on $M$ is a bilinear operation on smooth functions $\{\cdot, \cdot\}\colon \mathcal{C}^\infty(M)\times \mathcal{C}^\infty(M)\to \mathcal{C}^\infty(M)$ satisfying  skew-symmetry, the Jacobi and Leibniz identities.
The skew-symmetry and the Jacobi identity imply that the functions form a Lie algebra  on $\mathcal{C}^\infty(M)$ with respect to $\{\cdot, \cdot\}$. The Leibniz identity means that  given a function $H\in \mathcal{C}^\infty(M)$ the operator 
$\{H, \cdot \}$ is a {\it derivation} of $\mathcal{C}^\infty(M)$, and hence $  \{H, f\}=L_{v_H}f$ for a certain vector field $v_H$ on $M$ and any test function $f$. As before, the field $v_H$ is called the \textit{Hamiltonian vector field} corresponding to a   Hamiltonian function $H$.

A manifold $M$ equipped with a Poisson bracket is called a {\it Poisson manifold}\index{Poisson manifold}. Each Poisson manifold $(M, \{,\})$ gets equipped with a distribution of planes (whose dimension depends on a point). Namely, at any point $p\in M$ consider vectors  $v_H(p)$ of the Hamiltonian fields for all functions $H\in \mathcal{C}^\infty(M)$. We obtain a subspace $V\subset T_pM$ formed by  all such Hamiltonian vectors $v_H(p)$ at $p$. 
By varying the point $p$ over $M$ one defines a {\it plane distribution} on $M$, associates with the Poisson bracket $\{,\}$. According to the Weinstein theorem \cite{Weinstein83}, for any Poisson manifold $(M, \{,\})$ the associated plane distribution is integrable. 
Integral  submanifolds of this distribution have a natural symplectic structure (and are called {\rm symplectic leaves} of $M$), i.e. each Poisson manifold is foliated into symplectic submanifolds.

\medskip

\begin{example}
Now, let $\mathfrak{G}$ be a Lie group with Lie algebra $\mathfrak{g}=T_e\mathfrak{G}$. Consider the dual space $\mathfrak{g}^*$ to this Lie algebra.  There exists a linear \textit{Lie-Poisson bracket}\index{Lie-Poisson bracket} on $\mathfrak{g}^*$, defined as an operation 
	$\{,\}_{LP}: \mathcal{C}^\infty(\mathfrak{g}^*)\times \mathcal{C}^\infty(\mathfrak{g}^*)\to \mathcal{C}^\infty(\mathfrak{g}^*)$ by
	\begin{equation*}
		\{f,g\}_{LP}(m):= \big( [df(m), dg(m)], m \big) \,,
	\end{equation*}
where $f, g \in \mathcal{C}^\infty(\mathfrak{g}^*)$ and the differentials 
$df(m)$ and $dg(m)$ at $m \in \mathfrak{g}^*$ 
are identified with the corresponding elements of the Lie algebra via the natural pairing 
of $\mathfrak{g}$ and $\mathfrak{g}^*$ given by $(\cdot, \cdot )$. 
If the Lie algebra $\mathfrak{g}$ is infinite dimensional, for the latter identification one needs to constrain to a smooth dual of $\mathfrak{g}$, which we discuss in the next section.

\begin{figure}[ht!] 
\begin{center} 
\includegraphics[width=0.8\textwidth]{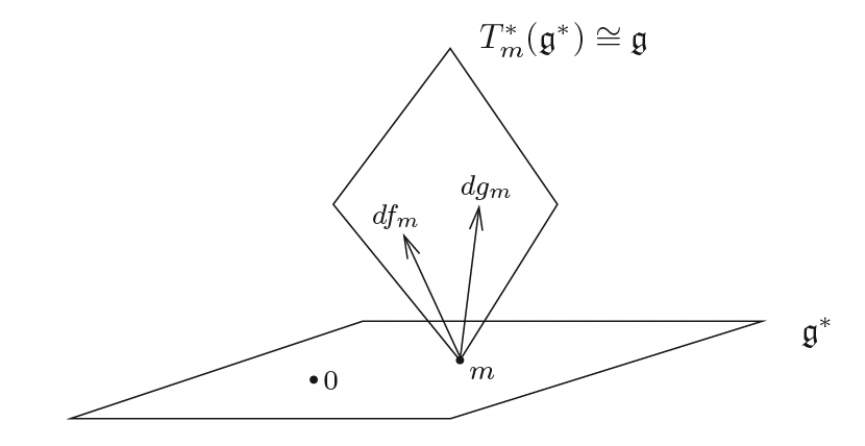} 
 \caption{Defining the Lie--Poisson structure: 
$df_m,~ dg_m\in \mathfrak{g}$, while $m\in \mathfrak{g}^*$.}\label{Fig:dfdg} 
\end{center} 
\end{figure} 


In this setting the Hamiltonian equation corresponding to the function $H$ 
and the Lie-Poisson bracket on $\mathfrak{g}^*$ takes the form 
$$ 
\frac{dm}{dt} = - \mathrm{ad}^\ast_{dH(m)} m. 
$$ 
\index{Euler-Arnold equation}
If the Hamiltonian system is driven by a kinetic energy $E$ (a quadratic function of its argument) 
then choosing $H=E$ yields the Euler-Arnold equation \eqref{euler2} of Section \ref{sect:hamilt} above. 
\end{example}%

\begin{remark}
    A finite-dimensional Hamiltonian system on a symplectic manifold $M$ of dimension $2n$ is said to be \emph{(completely) integrable}\index{integrable system} if it possesses 
$n$ (almost everywhere) functionally independent first integrals 
$f_1 =H$, $f_2, \dots, f_n$ 
which are pairwise in involution, i.e., $\{f_i, f_j \} =0$ for any $i$ and $j$. 
A large source of completely integrable systems are {\it bi-Hamiltonian systems}: they are defined by vector fields that are Hamiltonian with respect to two different but compatible Poisson brackets. (Two Poisson brackets are \textit{compatible} if all their linear combinations are themselves Poisson brackets, i.e. satisfy the skew-symmetry, Jacobi and Leibniz identities.)

In infinite dimensions the situation is more subtle. 
Possessing infinitely many first integrals may be insufficient to determine the motion of the system 
and the relations between various competing definitions of integrability are not yet fully understood. 
Nevertheless, there are infinite dimensional systems for which most of these notions agree. 
Perhaps, the most famous example is the Korteweg-de Vries equation of mathematical shallow water theory, see \cite{McKean82}. 
Other examples include the one-dimensional Hunter-Saxton equation\index{Hunter-Saxton equation}, 
the two-dimensional Kadomtsev-Petviashvili equation and the Camassa-Holm type equations, see e.g. \cite{Vaninsky}. All these equations 
can be shown to be bi-Hamiltonian with respect to two compatible Poisson brackets on appropriate infinite dimensional functional spaces \cite{KhesinMisiolek03}.
\end{remark}
%


\subsection{Smooth dual}\label{sec:smooth_dual}

If $\mathfrak{g}=T_e\mathfrak{G}$ is a Fréchet Lie algebra that is not a Banach space, then its dual space of continuous linear functionals is \emph{not} a Fréchet space.
The remedy is to instead fix a subspace of the full dual that is isomorphic to $T_e\mathfrak{G}$. 
Often this subspace is called the \emph{smooth dual}.
In the case of groups of diffeomorphisms on a manifold $M$, the standard setting is to select the smooth dual as the subspace of linear functionals identified by smooth one-forms, namely each smooth one-form $\alpha$ on $M$ gives rise to the functional $\bar\alpha\colon T_e\mathfrak{G}\to \R$ given by 
\begin{equation*}
    \bar\alpha(u) = \int_M \iota_u \alpha \, d\mu,
\end{equation*}
where $\mu$ is a reference volume form on $M$.
Throughout the remainder of the text, if $\mathfrak{g}=T_e\mathfrak{G}$ is a Fréchet Lie algebra of vector fields, we work with the smooth dual of one-forms, which we denote by $\mathfrak{g}^*=T_e^*\mathfrak{G}$.
 
In the case of $\mathfrak{G} = \mathfrak{D}(M)$, equation \eqref{euler2} can be written
\begin{equation}\label{eq:maineq1}
	m_t + L_u m + m \operatorname{div} u = 0, \quad m =A u ,
\end{equation}
\index{Euler-Arnold equation}
where $L_u$ denotes the Lie derivative along $u$, and, as just explained, the momentum variable $m$ is a smooth one-form.
From this formulation, it follows that the one-form density $\omega \coloneqq m\otimes\mu$ is  transported by the flow, i.e., $\frac{d}{dt}\omega + L_u\omega =0$.
This transport property generalizes how vorticity is transported by the flow of an ideal fluid.


\section{Example: the general Camassa--Holm equation} \label{sub:generalCH} 
 In this section we present in detail  an example of computations demonstrating that the general Camassa--Holm
 equation\index{Camassa-Holm equation}
 \begin{equation}\label{eq:genCH}
u_t + \kappa u_x - u_{txx} + 3uu_x - 2u_xu_{xx} - uu_{xxx} = 0\,
\end{equation}
for any real constant $\kappa$ 
is the Euler-Arnold equation on a certain extension of the Lie group $\mathfrak{G}=\mathfrak{D}(\mathbb{T})$ of circle diffeomorphisms. For $\kappa=0$ one obtains the ``classical" CH equation \eqref{eq:CH0}.

Rather then starting with a group, we first define a Lie algebra which, as we see below, corresponds to this equation.
This Lie algebra is the following extension $\widetilde{\mathfrak{X}}(\mathbb{T}) =\mathfrak{X}(\mathbb{T})\times \mathbb{R}$ given by the commutator 
 \begin{equation}\label{eq:VectExt}
[(u\partial, a), (v\partial, b)]:=([u\partial,v\partial],  \,\, \int_\mathbb{T} u_xv\, dx),
\end{equation}
which is the extension of the Lie algebra $\mathfrak{X}(\mathbb{T})$ of vector fields on the circle by means of the {\it (trivial) 2-cocycle} $c(u,v):=\int_\mathbb{T} u_xv\, dx$. Note that the commutator 
depends only on the vector fields $u\partial, v\partial$, but not on the constants $a,b$, which means that the latter belong  to the center of the new Lie algebra, it is a {\it central} extension of $\mathfrak{X}(\mathbb{T})$.
A Lie group corresponding to this extension is a central extension $\widetilde{ \mathfrak{G}}$
of the group of circle diffeomorphisms $\mathfrak{D}(\mathbb{T})$  by means of a line or a circle: topologically they are
direct products 
$\widetilde{ \mathfrak{G}} =\mathfrak{D}(\mathbb{T})\times \mathbb{R}$  
or $\widetilde{ \mathfrak{G}} =\mathfrak{D}(\mathbb{T})\times \mathbb{T}$. We  specify the corresponding group multiplications in the remark below, but as we have seen above, the Euler-Arnold equation relies on the Lie algebra only.

\begin{theorem}
The Euler-Arnold equation for the right-invariant $L^2$-metric on the group $\widetilde{ \mathfrak{G}}$  gives the general CH equation \eqref{eq:genCH} with an arbitrary constant $\kappa\in \mathbb R$.
\end{theorem}

\begin{proof}
For the Euler-Arnold equation \eqref{euler2}  on $\widetilde{\mathfrak{X}}^*(\mathbb{T})\equiv \{(m, k)\mid m\in \mathcal{C}^\infty(\TT), \, \kappa\in \R\}$ we compute here the operators ${\rm ad}^*$ and $A$.  First we note that  ${\rm ad}^*$ for the Lie algebra with commutator \eqref{eq:VectExt} is
$$
{\rm ad}^*_{(u\partial, a)} (m, \kappa)=(2u_x m + um_x+\kappa u_x, 0).
$$
Indeed, by definition
$$
\llangle {\rm ad}^*_{(u\partial, a)} (m, \kappa), (v\partial, b)\rrangle=
\llangle (m, \kappa), {\rm ad}_{(u\partial, a)} (v\partial, b)\rrangle=
$$
$$
\llangle (m, \kappa), ([u\partial,v\partial],  \,\, \int_\mathbb{T} u_xv\, dx)\rrangle=
\llangle (m, \kappa), ((u_x v -uv_x)\partial,  \,\, \int_\mathbb{T} u_xv\, dx)\rrangle=
$$
$$
\int_\mathbb{T} (m(u_x v -uv_x)+\kappa u_xv)\, dx = 
\int_\mathbb{T} (v(u_x m +(um)_x+\kappa u_x)\, dx = 
$$
$$
\llangle (2u_x m + um_x+\kappa u_x, 0),  (v\partial, b))\rrangle\,,
$$
from which the formula for ${\rm ad}^*$ follows.

Next, the $H^1$-inner product   gives rise to the inertia operator 
$A:\widetilde{\mathfrak{X}}(\mathbb{T}) \to \widetilde{\mathfrak{X}}^*(\mathbb{T}) $: 
$$ 
A=(1-\partial^2):\,  (u\partial,a)\mapsto(u-u_{xx},a)\,. 
$$ 
Indeed, 
$$
\langle  (u\partial,a), (v\partial,b)\rangle_{H^1}= 
 \int_{\mathbb{T}} (uv + u_x v_x) \, dx +ab=
\llangle ( u-u_{xx}, a), (v\partial, b)\rrangle\,. 
$$

Hence the  Euler-Arnold equation \eqref{euler2} assumes the form
$$
\frac{d}{dt} (m, \kappa)= -(2u_x m + um_x+\kappa u_x, 0)
$$
for $m=u-u_{xx}$. The latter is equivalent to the general CH equation:
$$
(u-u_{xx})_t=-2u_x(u-u_{xx})-u(u-u_{xx})_x-\kappa u_x\,,
$$
or, equivalently, 
$$
u_t-u_{txx}=-3uu_x +2u_x u_{xx} +u u_{xxx}-\kappa u_x\,,
$$
where $\kappa $ is a constant, $\kappa_t=0$.\qed
\end{proof}

\begin{remark}
There are different groups $\widetilde{ \mathfrak{G}}$ corresponding to the centrally extended Lie algebra $\widetilde{\mathfrak{X}}(\mathbb{T})  =\mathfrak{X}(\mathbb{T})\times \mathbb{R}$. They all are various central extensions of the group $\mathfrak{D}(\mathbb{T})$ of circle diffeomorphisms.  One option is to consider the set
 $\widetilde{ \mathfrak{G}} =\mathfrak{D}(\mathbb{T})\times \mathbb{T}$ with the following group multiplication between 
pairs  $(\varphi, a), (\psi, b)\in \mathfrak{D}(\mathbb{T})\times \mathbb{T}$:
$$
(\varphi, a)\circ(\psi, b):=(\varphi\circ\psi, B(\varphi,\psi))
$$
where the group 2-cocycle $B$ assumes values in the circle. Namely, let $C(\varphi):=\int_\TT (\varphi(x)-x)dx$ $\mod 2\pi$ 
is a $\mathbb{T}=\mathbb{R}/2\pi\mathbb{Z}$-valued function on the group $\GG=\mathfrak{D}(\mathbb{T})$, since the functions $\varphi(x),\psi(x)$ defining diffeomorphisms are multivalued $\mod 2\pi$. Now we define  the group 2-cocycle $B$
by  $B(\varphi,\psi):=C(\varphi)+C(\psi)-C(\varphi\circ\psi)$ $\mod 2\pi$.
(This formula means that $B$ is  a {\it trivial group 2-cocycle}, being the {\it co-differential} of the group 1-cochain $C$.) This product defines a (homologically trivial) extension $\widetilde{ \mathfrak{G}} =\mathfrak{D}(\mathbb{T})\times \mathbb{T}$ of the group $\mathfrak{G}=\mathfrak{D}(\mathbb{T})$. 

In order to construct an extension $\widetilde{ \mathfrak{G}} =\mathfrak{D}(\mathbb{T})\times \mathbb{R}$ 
one may adjust the cochain $C$ by setting $C'(\varphi):=\int_\TT h(\varphi(x)-x)dx$ where $h:\R\to \R$ is a $\mod 2\pi$-periodic function with the property $h(y)=y$ in a certain neighbourhood of the  origin $y=0$. Then the cochain 
$C':\mathfrak{D}(\mathbb{T})\to\R$ is an $\R$-valued (rather than $\TT$-valued) function on the group $\mathfrak{D}(\mathbb{T})$.

Note that another choice of the 2-cocycle $B(\varphi,\psi):=\int_\mathbb{T} \log (\varphi\circ\psi)'(x)\,d\log \psi'(x)$ defines the Virasoro 
group (and algebra), being a {\it nontrivial extension} of $\mathfrak{G}=\mathfrak{D}(\mathbb{T})$ and it is
responsible for the KdV equation, as well as other versions of the CH and HS equations related by a shift of variables, see e.g. \cite{Misiolek98}.
\end{remark}

\subsection{Sub-Riemanian metrics and the general Camassa--Holm equation} 

One more important source of examples of the Euler-Arnold equation is provided by sub-Riemannian metrics on Lie groups.
Such sub-Riemannian metrics are defined by fixing an inner product on a vector subspace $T_e\mathfrak{G}_a\subset T_e\mathfrak{G}$ of the corresponding Lie algebra and then extending it by, say, right-invariance on the whole Lie group. 
If such a subspace is not a Lie subalgebra, the corresponding invariant distribution of subspaces on the Lie group is not integrable, and this defines a one-sided invariant sub-Riemannian metric on the group $G$. 
(In a sense, such metrics are ``dual" to those  discussed in Section \ref{sect:quotient} for  homogeneous spaces:
rather then being vanishing in certain directions as in the homogeneous setting, the sub-Riemannian metrics are ``infinite" in the directions not lying in the subspace.)

The corresponding  Euler-Arnold equation has the form similar to \eqref{eq:Euler-Arnold} or \eqref{euler2}
with linear constraints in the right-hand side corresponding to the constraints on  $T_e\mathfrak{G}_a:=\{v\in T_e\mathfrak{G}~|~a(v)=0\}$ for some fixed element $ a\in  T_e^*\mathfrak{G}$. Namely, the sub-Riemannian Euler-Arnold equation in that case is
\begin{equation}\label{euler-sub} 
    m_t= -{\rm ad}^*_{A^{-1}m}m+\lambda a\,,
 \end{equation} 
where $\lambda$ is a Lagrange multiplier.


It turns out that the general Camassa-Holm equation can be described as a sub-Riemannian metric on the (non-extended!)
group $\mathfrak{D}(\mathbb{T})$ of circle diffeomorphisms. Namely, following  \cite{Grong14}
 in the Lie algebra $\mathfrak{X}(\mathbb{T})$ consider the hyperplane $\mathfrak{X}_0(\mathbb{T})$ of vector fields with zero mean. This is not a Lie subalgebra. Now look at the corresponding right-invariant distribution of hyperplanes 
in  $\mathfrak{D}(\mathbb{T})$ generated by $\mathfrak{X}_0(\mathbb{T})\subset \mathfrak{X}(\mathbb{T})$
at the identity $e\in \mathfrak{D}(\mathbb{T})$. 
It is nonintegrable (since $\mathfrak{X}_0(\mathbb{T})$ is not a Lie subalgebra), and it defines a contact structure on the group $\mathfrak{D}(\mathbb{T})$. 

Now fix the $H^1$-metric \eqref{eq:H1metric} on $\mathfrak{X}_0(\mathbb{T})$ and consider sub-Riemannian geodesics on the group  $\mathfrak{D}(\mathbb{T})$ with respect to the right-invariant metric on this distribution. The sub-Riemannian geodesics are defined by an initial vector and one more parameter (one may think of that as an analog of the acceleration), and their equation will be the general Camassa-Holm equation \eqref{eq:genCH}, where $\kappa$ is exactly this extra parameter (see \cite{Grong14}, Sections 5.3-5.4). 
This provides an alternative way to derive the general Camassa-Holm equation. It has its advantages and contraints: the corresponding vector fields are constrained by zero mean, i.e., not moving mass.

\section{Optimal transport and Wasserstein distance} 
\label{app2} 

Optimal mass transport goes back to Monge~\cite{Mo1781} and was later reformulated by Kantorovich~\cite{Ka1942,Ka1948}.\index{optimal mass transport}
(For a modern presentation, see any of the lecture notes by Evans~\cite{Ev2001_lecture_notes}, Ambrosio and Gigli~\cite{AmGi2009}, or McCann~\cite{Mc2010}, or any of the monographs by Villani~\cite{Villani09} or Peyré and Cuturi~\cite{PeCu2020}.)
Its connection to Riemannian geometry and hydrodynamics was pioneered by Brenier~\cite{Br1991}, who realized that the optimal mass transport problem for the ``distance square'' cost (also called the $L^2$ cost) gives rise to an ``optimal factorization'' of maps, which in turn gives the solution to the Monge problem.
The (weak) Riemannian aspects of optimal mass transport was then more explicitly pointed out by Otto~\cite{Otto01}, who showed that Arnold's right-invariant metric~\eqref{eq:arnold_metric_L2}, but extended to a semi-invariant metric on all vector fields, is descending with respect to the left coset projection $\Pi\colon \mathfrak{D}(M)\to \mathfrak{D}(M)/\mathfrak{D}_\mu(M)\simeq \mathfrak{Dens}(M)$. This projection $\Pi$ is the push-forward of a density by a diffeomorphism, and under this projection the corresponding Riemannian structure on $\mathfrak{Dens}(M)$ gives the Wasserstein-2 distance.
In this section we give a brief overview of this approach.
More details of local and global existence of geodesic curves for semi-invariant Riemannian metrics on the full group of diffeomorphisms are given by Bauer and Modin~\cite{BaMo2020}.
A finite dimensional version of the framework, corresponding to optimal transport between multivariate Gaussian distributions, is discussed by Modin~\cite{Modin17}.
Later, in Sections~\ref{sec:factorizations_abstract} and \ref{sec:factorizations_information}, we give an analogous construction for the right co-sets (as discussed in Section~\ref{sec:moser_fibration} above), but optimal with respect to the Fisher-Rao metric instead of the Wasserstein-Otto metric.\index{Wasserstein-Otto metric}\index{Riemannian metric}

Consider the space of densities $\mathfrak{Dens}(M)$ as the space  
$\mathfrak{D}(M)/\mathfrak{D}_\mu(M)$ of left cosets described in Remark \ref{rem:right-left} above. 
The projection of $\mathfrak{D}(M)$ onto this quotient given by the pushforward map 
$\Pi(\xi) = (\xi^{-1})^\ast \mu = (\mathrm{Jac}_\mu \xi^{-1}) \mu$ 
defines (as in the case of right cosets) a smooth principal bundle with fiber $\mathfrak{D}_\mu(M)$. 

The group $\mathfrak{D}(M)$ carries a natural $L^2$ metric 
\begin{equation} \label{eq:nonL2} 
\langle u\circ\xi, v\circ\xi \rangle_{L^2} 
= 
\int_M \langle u, v \rangle \, \mathrm{Jac}_\mu \xi \, d\mu 
\qquad 
u, v \in \mathfrak{X}(M)=T_e\mathfrak{D}(M), \; \xi \in \mathfrak{D}(M) \,,
\end{equation} 
whose geometry is relatively easy to visualize: 
a curve $t \mapsto \xi(t)$ is a geodesic in $\mathfrak{D}(M)$ 
if and only if $t \mapsto \xi(t)(x)$ is a geodesic in $M$ for each $x$. 
Observe that in general this metric is neither left- nor right-invariant. 
It is right-invariant when restricted to $\mathfrak{D}_\mu(M)$ 
and 
it becomes left-invariant only when restricted to the subgroup of isometries. 

In order to see that the projection $\Pi$ also defines a Riemannian submersion note that 
the horizontal vectors with respect to \eqref{eq:nonL2} have the form 
$\nabla f \circ \xi$ for some $f:M \to \mathbb{R}$ 
and that the metric descends to a (weak) Riemannian metric on the base 
\begin{equation} \label{eq:Otto} 
\langle \alpha, \beta \rangle_\rho 
= 
\int_M \langle \nabla f, \nabla g \rangle \, \rho \, d\mu \,,
\end{equation} 
where $f$ and $g$ solve the equations 
$\mathrm{div}\, (\rho\nabla f) = -\alpha$ 
and 
$\mathrm{div}\, (\rho\nabla g) = -\beta$ 
and the mean-zero functions $\alpha$ and $\beta$ are tangent vectors to $\mathfrak{Dens}(M)$ at $\rho$. 

The Riemannian distance of \eqref{eq:Otto} between two measures $\nu$ and $\lambda$ 
on the space of densities $\mathfrak{Dens}(M)$ (viewed as the space of left cosets) 
has a very appealing interpretation as the $L^2$ cost of transporting one density to the other 
\begin{equation} \label{eq:Kantorovich} 
\mathrm{dist}^2_W (\nu, \lambda) 
= 
\inf_\xi \int_M \mathrm{dist}^2_M (x, \xi(x)) \, d\mu(x) 
\end{equation} 
where the infimum is over all diffeomorphisms $\xi$ such that $\xi^\ast \lambda = \nu$ 
and $\mathrm{dist}_M$ is the Riemannian distance on $M$. 
In this setting $\mathrm{dist}_W$ is often referred to as 
the \emph{$L^2$-Wasserstein} or (perhaps more appropriately) the \emph{Kantorovich-Rubinstein distance} 
\index{Kantorovich-Rubinstein distance} between densities $\nu$ and $\lambda$, while its Riemannian version \eqref{eq:Otto}  is called  \emph{Wasserstein--Otto metric}  on $\mathfrak{Dens}(M)$.
It is of fundamental importance in optimal transport problems. 
The geometric picture behind these constructions are given in Figure~\ref{Fig:fibration-OT}.
We refer to the papers \cite{BenamouBrenier01}, \cite{Otto01} or the comprehensive monograph \cite{Villani09} for more details. 

\begin{figure}
    \includegraphics{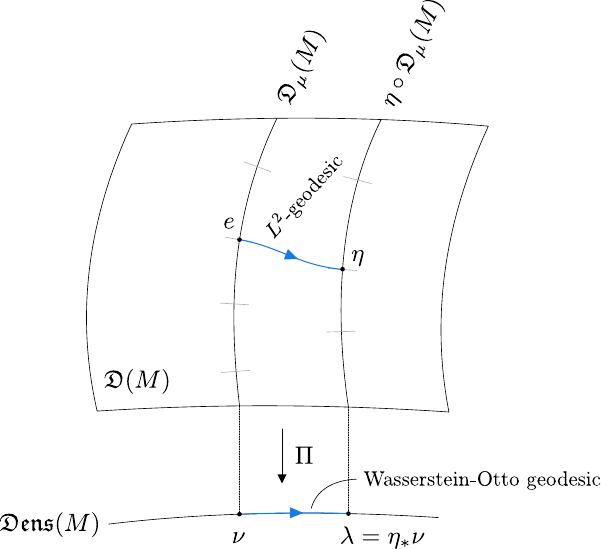} 
    \caption{
    The Riemannian geometry behind $L^2$ optimal mass transport.
    Shortest geodesics between fibers, with respect to the (non-invariant) $L^2$ metric on $\mathfrak{D}(M)$, project via push-forward to Wasserstein-Otto geodesics on $\mathfrak{Dens}(M)$ which give the $L^2$ Wasserstein distance between $\nu$ and $\lambda$. 
    }\label{Fig:fibration-OT} 
\end{figure}



%
%
%

\chapter{The infinite dimensional Fisher-Rao metric} 
\label{sec:FR-met} 
As configuration spaces of various physical systems, diffeomorphism groups 
often come equipped with natural (pre-) Riemannian structures. 
Historically, of most interest were those structures related to the $L^2$ inner product 
corresponding to the kinetic energy of the system 
-- the prime example being hydrodynamics of ideal fluids given in Example~\ref{ex:ideal_euler} above. 
However, in recent years there appeared many new examples from analysis, to geometry, 
to mathematical physics where the appropriate Riemannian structures 
came from higher order Sobolev inner products such as in Example \ref{ex:Hr}. 

In this chapter  we  show that the Fisher-Rao (information) metric, 
which plays a fundamental role in geometric statistics, is closely related to an $H^1$-type Sobolev inner product 
on the space of diffeomorphisms. 
The approach developed below places information geometry 
squarely within the general differential-geometric framework for diffeomorphism groups 
envisioned by Cartan, Kolmogorov and Arnold.

\section{A right-invariant homogeneous $H^1$ Sobolev metric on $\mathfrak{D}(M)$} 
\label{subsec:H1-met} 
We shall continue to assume that 
$M$ is an $n$-dimensional compact Riemannian manifold without boundary 
whose volume form is normalized so that $\mu(M)=1$. 
The triple $(M, \mathcal{B}, \mu)$, where $\mathcal{B}$ is the $\sigma$-algebra of Borel sets in $M$, 
will be the fixed background sample space, 
and the Fr\'echet manifold of right cosets 
$\mathfrak{Dens}(M) = \mathfrak{D}_\mu(M)\backslash\mathfrak{D}(M)$ 
will serve as an \emph{infinite dimensional statistical model}\index{statistical model} 
whose points are (smooth) probability measures\footnote{Alternatively, they are (smooth) density functions 
given by the corresponding Radon-Nikodym derivatives.} 
on $M$ that are absolutely continuous with respect to $\mu$. 
We shall equip the principal bundle $\mathfrak{D}(M)$ over $\mathfrak{Dens}(M)$ 
with the structure of a Riemannian submersion. 
To that end, observe that the condition stated in Proposition~\ref{prop:descend} is precisely what one needs.  

Consider the homogeneous Sobolev $H^1$ inner product on the Fr\'echet Lie algebra 
of divergence free vector fields on $M$. 
Define the corresponding (degenerate) right-invariant $\dot H^1$ inner product on the total space $\mathfrak{D}(M)$ 
using \eqref{eq:ri-met}, namely 
\begin{equation} \label{eq:homH1met} 
\langle V, W \rangle_{\dot{H}^1} 
= 
\frac{1}{4} \int_M \mathrm{div} \, v \cdot \mathrm{div}\, w \, d\mu \,,
\qquad 
V, W \in T_\eta\mathfrak{D}(M) \,,
\end{equation} 
such that $V = v \circ \eta$ and $W = w \circ \eta$ where $v, w \in T_e\mathfrak{D}(M)$ and $\eta \in \mathfrak{D}(M)$. \index{$\dot H^1$-metric}

Clearly, the metric \eqref{eq:homH1met} generalizes the one-dimensional case \eqref{eq:H1met-S1}. 
Moreover, it satisfies the condition \eqref{eq:dsc} and therefore descends to a (non-degenerate) metric 
on $\mathfrak{Dens}(M)$. 
The attendant geometry is particularly remarkable. 
As we will see below, the space of densities $\mathfrak{Dens}(M)$ viewed as the space of right cosets 
equipped with this metric is isometric to a subset of the unit sphere in the Hilbert space 
and its Riemannian distance coincides with the spherical Hellinger distance.\index{Hellinger distance} 
Furthermore, the homogeneous metric~\eqref{eq:homH1met} is related to the so-called Bhattacharyya coefficient (affinity) 
in probability and statistics. 

\section{The information metric on \texorpdfstring{$\mathfrak{D}(M)$}{D(M)}}

The bilinear form \eqref{eq:homH1met} is non-negative but degenerate: it vanishes on divergence free vector fields.
There is, however, a true (i.e., non-degenerate) right-invariant Riemannian metric on $\mathfrak{D}(M)$ that descends to the same metric on the quotient $\mathfrak{D}_\mu(M)\backslash \mathfrak{D}(M)$.
It is the \emph{information metric on diffeomorphisms}.
It is convenient to describe it in the language of differential one-forms rather than vector fields.
Recall that if $u$ is a vector field on $M$, we denote by $u^\flat$ the corresponding one-form obtained by contraction with the Riemannian tensor.

\subsection{The Hodge decomposition of $k$-forms}
Let $M$ be a compact $n$-dimensional Riemannian manifold without boundary and let $\Omega^k(M)$ denote the space of $k$-forms on $M$, which is a tame Fréchet space.
The Riemannian metric on $M$ gives rise to the $L^{2}$~inner product on $\Omega^{k}(M)$, given by
\begin{equation}\label{eq:L2_kforms}
	\langle a,b\rangle_{L^{2}} = \int_{M} a\wedge\star b ,
\end{equation}
where $\star\colon\Omega^{k}(M)\to\Omega^{n-k}(M)$ is the Hodge star map (which depends on the Riemannian metric).
This definition is compatible with the standard $L^2$ inner product on vector fields, i.e., $\langle u,v\rangle_{L^{2}} = \langle u^{\flat},v^{\flat}\rangle_{L^{2}}$.

The differential (or exterior derivative) takes $k$-forms to $(k+1)$-forms, 
$$d\colon \Omega^{k}(M)\to \Omega^{k+1}(M).$$
The adjoint of $d$ with respect to the $L^2$ inner product~\eqref{eq:L2_kforms} is the co-differential $$\delta\colon\Omega^{k+1}(M)\to \Omega^{k}(M),$$
defined by
\begin{equation*}
	\langle da, b \rangle_{L^2} = \langle a,\delta b \rangle_{L^2} \quad\text{for all $a\in\Omega^k(M)$ and $b\in\Omega^{k+1}(M)$.}
\end{equation*}
Just as the differential corresponds to the gradient, namely if $f\in \mathcal{C}^\infty(M)$ then $(\nabla f)^\flat = df$, the co-differential corresponds to the divergence, namely if $u\in T_e\mathfrak{D}(M)$ then $\operatorname{div}u = -\delta u^\flat$.
From this perspective, it is natural to define the \emph{Laplace-de Rham operator} as the second order differential operator on $k$-forms defined by
\begin{equation}\label{eq:hodge_laplacian}
	\Delta\colon\Omega^k(M)\to \Omega^k(M), \quad \Delta a = -(d\circ\delta + \delta\circ d)a .
\end{equation} 

\begin{remark}
	Often the Laplace-de Rham operator~\eqref{eq:hodge_laplacian} is defined without the minus sign, making it non-negative.
	Here, however, we include the minus sign so that it directly corresponds to the Laplace-Beltrami operator on $\mathcal{C}^\infty(M)=\Omega^0(M)$ (which is non-positive).
	Notice, however, that the tensor field version of the Laplace-Beltrami operator applied to $k$-form for $k\geq 1$ in general is different from the Laplace-de Rham operator~\eqref{eq:hodge_laplacian}.
\end{remark}

The kernel of the Laplace-de Rham operator is the space $H^k(M) = \{ a \in \Omega^k(M)\mid \Delta a=0 \}$ of harmonic $k$--forms.
It is isomorphic as a vector space to the $k$th cohomology class of $M$, which is the Hodge theorem. 
Furthermore, the Hodge decomposition states that
\begin{equation}\label{eq:hodge_decomposition}
	\Omega^k(M) = H^{k}(M)\oplus\delta\Omega^{k+1}(M)\oplus d \Omega^{k-1}(M), 	
\end{equation}
where the components are orthogonal with respect to the $L^2$ inner product~\eqref{eq:L2_kforms}.

Let $F^k(M) \coloneqq H^{k}(M)\oplus\delta\Omega^{k+1}(M)$.
Then $F^k(M) = \{ a \in \Omega^k(M)\mid \delta a = 0 \}$ is the space of co-closed $k$-forms, which in the case $k=1$ corresponds to the divergence free vector fields:
\begin{equation*}
	\operatorname{div}u = 0 \iff u^\flat \in F^1(M) .
\end{equation*}
Thus, we recover from the Hodge decomposition~\eqref{eq:hodge_decomposition} the Helmholtz decomposition\index{Helmholtz-Hodge decomposition} of vector fields, which states that each vector field $v\in T_e\mathfrak{D}(M)$ decomposes as $v = u + \nabla f$ where $\operatorname{div} u =0$.
 
We now have enough tools to define the information metric on $\mathfrak{D}(M)$.

\subsection{Definition of the information metric}

For simplicity, we assume for now that the first cohomology of $M$ is trivial, i.e., that $H^1(M) = \{0\}$.
The general situation, including when $M$ has a boundary, is briefly discussed in Remark~\ref{rmk:general_cohomology_information_metric} below.

\begin{definition}\label{def:information_metric}
	The \emph{information metric} on $\mathfrak{D}(M)$ is the right-invariant metric defined by the Sobolev $H^1$-type inner product on $T_e\mathfrak{D}(M)$ given by\index{information metric}\index{Riemannian metric}
	\begin{equation}\label{eq:info_metric}
		 \langle u,v\rangle_\Delta := \frac{1}{4}\langle u^\flat, -\Delta v^\flat \rangle_{L^2}.   
	\end{equation}
	The corresponding inertia operator $A\colon T_e\mathfrak{D}(M)\to T_e^*\mathfrak{D}(M)$ is $$Au = -\frac{1}{4}\Delta u^\flat.$$
\end{definition}

From the definition of the Laplace-de Rham operator and the Hodge decomposition it follows that the information metric~\eqref{eq:info_metric} can be written
\begin{equation}\label{eq:info_metric2}
	\langle u,v\rangle_\Delta  = \frac{1}{4}\langle \delta u^\flat,\delta v^\flat\rangle_{L^2} + \frac{1}{4}\langle du^\flat, d v^\flat \rangle_{L^2}.   
\end{equation}
Notice that the first term in the formula \eqref{eq:info_metric2} is exactly the degenerate $\dot H^1$ metric~\eqref{eq:homH1met} discussed above.
Thus, the information metric extends the $\dot H^1$ metric to a non-degenerate metric.
Of course, there are many extensions that make $\dot H^1$ non-degenerate.
The special property of the information metric in Definition~\ref{def:information_metric} is that it retains the descending property of the $\dot H^1$ metric. 
Since it is non-degenerate, it also defines a horizontal distribution.

\begin{theorem}\label{thm:metric_is_descending}
	The information metric~\eqref{eq:info_metric} on $\mathfrak{D}(M)$ descends to a metric on $\mathfrak{D}_\mu(M)\backslash \mathfrak{D}(M) \simeq \mathfrak{Dens}(M)$.
	Thus, the map $\pi\colon \mathfrak{D}(M)\to \mathfrak{Dens}(M)$ in Proposition~\ref{prop:PB} is a Riemannian submersion\index{Riemannian submersion}, for which the horizontal distribution is given by 
	\begin{equation*}
		\mathcal{H}_\eta = \big\{ U\in T_\eta\mathfrak{D}(M) \mid U\circ\eta^{-1} \in  \nabla( \mathcal{C}^\infty(M)) \big\} ,
	\end{equation*}
\end{theorem}

\begin{figure} 
\includegraphics{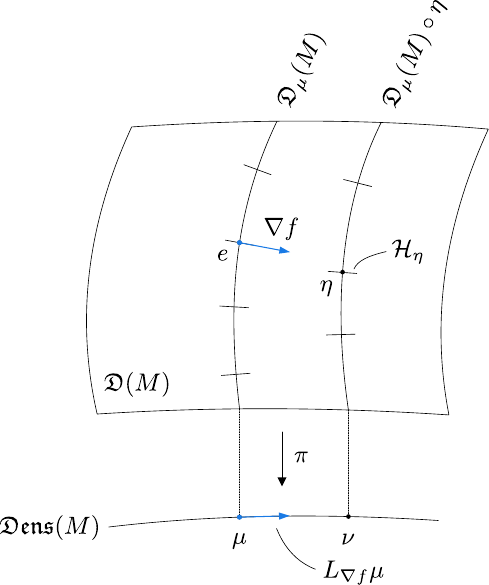}
\caption{
	By equipping $\mathfrak{D}(M)$ with the information metric~\eqref{eq:info_metric}, Moser's fibration of $\mathfrak{D}(M)$ discussed in section~\ref{sec:moser_fibration} above becomes a Riemannian submersion.
	The corresponding horizontal distribution $\mathcal H$ is generated from the Lie algebra $T_e\mathfrak{D}(M)$ by gradient vector fields that are right-translated to arbitrary points $\eta \in \mathfrak{D}(M)$.
	In Section~\ref{sec:factorizations_information} below we use this structure to construct a factorization of $\mathfrak{D}(M)$ which is optimal with respect to the information metric, analogous to how Brenier's factorization of maps is optimal with respect to the non-invariant $L^2$ metric.
}\label{fig:H1 fibration}
\end{figure}

\begin{proof}
	First, the inner product~\eqref{eq:info_metric} preserves orthogonality with respect to the Helmholtz decomposition, i.e.,
	\[
		AT_e\mathcal{D}_\mu(M) = F^1(M) ,
		\quad
		A \nabla(\mathcal{C}^\infty(M)) = d \Omega^0(M).
	\]
	It thereby follows that the horizontal distribution at the identity is given by
	\begin{equation*}
		T_e^\perp\mathfrak{D}_\mu(M) = \big\{ u\in T_e\mathfrak{D}(M) \mid u \in  \nabla( \mathcal{C}^\infty(M)) \big\} ,
	\end{equation*}
	i.e., vectors that are given by gradient vector fields.
	If $u\in T_e\mathfrak{G}_\mu(M)$ and $f,g\in \mathcal{C}^\infty(M)$, then
	\begin{equation*}
		\begin{split}
		\pair{ L_u  \nabla f, \nabla g}_{\Delta} 
		&= \int_M \delta ( L_u  \nabla f)^{\flat}\delta \nabla g^{\flat}\, \mu \\
		&= \int_M \,d  \iota_{ L_u  \nabla f}\mu \wedge\star\,d \iota_{ \nabla g}\mu \\
		&= -\int_M \,d  \iota_{ \nabla f}\mu \wedge\star\,d \iota_{ L_u \nabla g}\mu 
		= -\pair{ \nabla f, L_u \nabla g}_{\Delta},
		\end{split}
	\end{equation*}
	where we have used $ L_u\mu =0$ and $ L_u\star a = \star  L_u a$ for any $a\in\Omega^{n}(M)$.
	From Proposition~\ref{prop:descend} it follows that the information metric is descending.

	The tangent map $T\pi_\mu$ restricted to $T_e\mathfrak{D}(M)=\mathfrak{X}(M)$ is given by $u\mapsto  L_u\mu$.
	Now,
	\begin{equation*}
		\begin{split}
			\pair{ \nabla f, \nabla g}_{\Delta} &= 
			\int_M \,\iota_{ \nabla f}\mu\wedge\star\iota_{ \nabla g}\mu \\
			&= \int_M \, L_{ \nabla f}\mu\wedge\star L_{ \nabla f}\mu 
			= \pair{ L_{ \nabla f}\mu, L_{ \nabla g}\mu}_{L^2}.
		\end{split}
	\end{equation*}
	Therefore, the information metric for horizontal vectors at the identity tangent space is given by the Fisher metric multiplied by $\beta$ of the projection of the horizontal vectors to the tangent space $T_\mu\mathfrak{Dens}(M)$.
	Since this holds at one tangent space,  it holds at every tangent space, since both the information metric and the Fisher metric are right-invariant.
	This concludes the proof.
\end{proof}

\begin{remark}\label{rmk:general_cohomology_information_metric}
	In the case when the first co-homology of $M$ is non-trivial one can still define the information metric, but the harmonic parts have to be treated separately.
	Indeed, in this case the Hodge decomposition yields the finer decomposition
	\[
	\mathfrak{X}_\mu(M) = \mathfrak{X}_H(M)\oplus \mathfrak{X}_{\mu,ex}(M) ,
	\] 
	where $\mathfrak{X}_{\mu,ex}(M) = \delta\Omega^{2}(M)^\sharp$ is the space of \emph{exact} volume preserving vector fields, and $\mathfrak{X}_H(M) = H^{1}(M)^\sharp$ is the space of \emph{harmonic} vector fields.
	The trick is to define $A$ to be diagonal also with respect to this finer decomposition.
	Indeed, the $L^2$~orthogonal projection operator~$R\colon\Omega^{1}(M)\to H^{1}(M)$ onto the harmonic part is given by
	\[
		R = \mathrm{id} + d \circ\Delta^{-1}\circ\delta + \delta\circ\Delta^{-1}\circ d  .
	\]
	We can then choose $A$ to be
	$$A = (R + \delta\circ d  + d \circ\delta)\circ\flat = (R + \Delta)\circ\flat .$$
	If $h\in \mathfrak{X}_H(M)$ then
	\[
		A{h} = \underbrace{ R {h}^{\flat}}_{{h}^{\flat}}+ \delta\underbrace{d {h}^{\flat}}_{0}+d \underbrace{\delta{h}^{\flat}}_{0} = {h}^{\flat} \in   H^{1}(M).
	\]
	If $\xi\in \mathfrak{X}_{\mu,ex}(M)$ then
	\[
		A\xi = \underbrace{R\xi^{\flat}}_{0} + \delta d \xi^{\flat}+ d \underbrace{\delta\xi^{\flat}}_{0} =  \delta d \xi^{\flat} \in \delta\Omega^{2}(M).
	\]
	If $f\in \mathcal{C}^\infty(M)$ then
	\[
		A \nabla f  = \underbrace{ R d  f + \delta d d  f}_{0} -  d \Delta f = - d \Delta f \in d \Omega^0(M).
	\]
	Thus, if we represent $u = h + \xi +  \nabla f$ by its ``Helmholtz--Hodge components'' $(h,\xi,f)\in\mathfrak{X}_H(M)\times\mathfrak{X}_{\mu,ex}(M)\times \mathcal{C}^\infty_0(M)$, then 
	\[
		A(h,\xi,f) = (h^{\flat},-\Delta\xi^\flat,-\Delta f)\in  H^{1}(M)\times\delta\Omega^{2}(M)\times \mathcal{C}^\infty_0(M).
	\]
	Since both $\Delta\circ\flat\colon\mathfrak{X}_{\mu,ex}(M)\to\delta\Omega^{2}(M)$ and $\Delta\colon \mathcal{C}^\infty_0(M)\to \mathcal{C}^\infty_0(M)$ are invertible operators, $A$ is also invertible.
	This construction can also be extended to the case when $M$ has a boundary, by using the Hodge decomposition for manifolds with boundary (which can be found in the book by Schwarz~\cite{Schwarz1995}).
\end{remark}


\section{The square root map} 
\label{subsec:sphere} 
We begin by constructing a map between the quotient space of right cosets $\mathfrak{Dens}(M)$ 
and a subset of the unit sphere in the Lebesgue space of square integrable functions 
on $M$ with induced metric 
$$ 
S^\infty_{L^2} = \Big\{ u \in L^2(M,d\mu)\mid  \int_M |u|^2 d\mu = 1 \Big\}. 
$$  
As before we let $\mathrm{Jac}_\mu \eta$ denote the Jacobian of $\eta$ with respect to $\mu$. 
\begin{theorem} \label{thm:isometry} \emph{(Isometry theorem)} 
The map\index{square root map} 
$$ 
\Psi: \mathfrak{D}(M) \to L^2(M, d\mu) 
\qquad 
\xi \mapsto \Psi(\xi) = \sqrt{\mathrm{Jac}_\mu \xi} 
$$ 
defines an isometry from the space of densities 
$\mathfrak{Dens}(M) = \mathfrak{D}_\mu(M)\backslash\mathfrak{D}(M)$ 
with the $\dot{H}^1$ metric \eqref{eq:homH1met} 
to a subset of 
the unit sphere $S^\infty_{L^2}  \subset \mathcal{C}^\infty(M) \cap L^2(M, d\mu)$ 
with the standard $L^2$ metric. 
\end{theorem} 

\begin{proof} 
First, note that the Jacobian of any orientation preserving smooth diffeomorphism 
is a (strictly) positive smooth function. 
Since we are viewing $\mathfrak{Dens}(M)$ as the space of right cosets 
with the projection $\pi$ given by the pullback map (cf. Proposition \ref{prop:PB}), 
given any $\xi$ in $\mathfrak{D}(M)$ and $\eta$ in $\mathfrak{D}_\mu(M)$ 
we have 
$$ 
(\eta \circ \xi)^\ast \mu 
= 
\pi(\eta\circ\xi) 
= 
\pi(\xi) 
= 
\xi^\ast \mu \,,
$$ 
which implies that $\mathrm{Jac}_\mu(\eta\circ \xi) = \mathrm{Jac}_\mu \xi$. 
Furthermore, using the change of variable formula we find 
$$ 
\int_M \Psi^2(\xi) \, d\mu 
= 
\int_M \mathrm{Jac}_\mu \xi \, d\mu 
= 
\mu(M) = 1. 
$$ 
From the above it follows that $\Psi$ descends to a well-defined map 
from the quotient space $\mathfrak{Dens}(M)$ into the unit sphere in $L^2(M,d\mu)$. 
\begin{figure}	
	\includegraphics{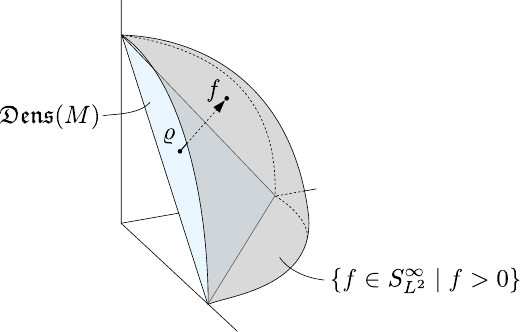}
	\caption{Illustration of the square root map in the positive quadrant: $\mathfrak{Dens}(M)\ni \varrho \mapsto f\in S_{L^2}^\infty$ where $\varrho=f^2 \, \mu$. 
	}\label{fig:SquareRoot}
\end{figure}

Next, if 
$\mathrm{Jac}_\mu \xi = \mathrm{Jac}_\mu \zeta$ 
for some $\xi$ and $\zeta$ in $\mathfrak{D}(M)$ 
then we have 
$(\xi \circ \zeta^{-1})^\ast \mu = \mu$ 
and, consequently, we find that $\Psi$ is injective. 

Finally, differentiating the identity 
$\mathrm{Jac}_\mu \xi \, \mu = \xi^\ast \mu$ 
with respect to $\xi$ in the direction $V \in T_\xi\mathfrak{D}(M)$ 
we obtain 
$$ 
d_\xi \mathrm{Jac}_\mu (V) 
= 
\mathrm{div}_\mu (V \circ \xi^{-1}) \circ \xi \mid \mathrm{Jac}_\mu \xi. 
$$ 
Recall that tangent vectors $V$, $W$ at any $\xi \in \mathfrak{D}(M)$ 
can be always represented in the form $V = v\circ\xi$ and $W = w\circ\xi$  
where $v, w \in T_e\mathfrak{D}(M)$. 
Thus, applying Fr\'echet calculus as before, changing variables and using \eqref{eq:homH1met}, 
we compute 
\begin{align*} 
\langle d_\xi (\Psi\circ\pi) (V), d_\xi (\Psi\circ) (W) \rangle_{L^2} 
&= 
\frac{1}{4} \int_M  
(\mathrm{div}_\mu v \circ \xi) \cdot (\mathrm{div}_\mu w \circ \xi) \, \mathrm{Jac}_\mu \xi 
\, d\mu 
\\ 
&= 
\frac{1}{4} \int_M \mathrm{div}_\mu v \cdot \mathrm{div}_\mu w \, d\mu 
\\ 
&= 
\langle v, w \rangle_{\dot{H}^1} 
\end{align*} 
This shows that $\Psi$ is an isometry. 
\end{proof} 
\begin{remark} 
If $s>n/2+1$ then the map $\Psi$ defines a diffeomorphism from the Sobolev completion 
$\mathfrak{D}_\mu^s(M)\backslash\mathfrak{D}^s(M)$ 
onto 
a subset of $S^\infty_{L^2}  \cap H^{s-1}(M)$ of positive functions on $M$ 
and the above proof extends with minor modifications. 
The fact that any positive function in $S^\infty_{L^2}  \cap H^{s-1}(M)$ 
lies in the image of $\Psi$ follows directly from Moser's lemma, whose generalization to 
the setting of Sobolev $H^s$ spaces can be found in \cite{EbinMarsden70}; 
see also Appendix \ref{Banach}. 
\end{remark} 
%

\section{The infinite dimensional Fisher-Rao metric on $\mathfrak{Dens}(M)$} 
\label{subsec:FR-met} 
The appearance of diffeomorphism groups in our formulation of 
the infinite dimensional generalization of information geometry should not be entirely surprising. 
In some sense it could be discerned from various results in the finite dimensional setting such as 
invariance properties of the Fisher information with respect to sufficient statistics\footnote{This suggests 
a possibility of further generalizations, which we will not pursue here.} 
--- in particular, with respect to invertible mappings of the sample space $M$. 

Attempts to find conceptually natural and useful approaches to mathematical statistics 
go back to the pioneering work of Fisher, Rao and Kolmogorov. 
Their ideas were subsequently further developed by 
Chentsov, Morozova, Efron, Amari, Barndorff-Nielsen, Lauritzen, Nagaoka among others; 
see e.g., 
\cite{Chentsov82}, \cite{ChentsovMorozova91}, \cite{Efron75}, \cite{Amari82}, \cite{ABNKLR}, 
\cite{AmariNagaoka00}. 
%
Along the way various infinite dimensional generalizations were also considered by 
David \cite{David77}, Friedrich \cite{Friedrich91}, Pistone and Sempi \cite{PistoneSempi95}, 
Gibilisco and Pistone \cite{GibiliscoPistone98}, and Ay, Jost, Le and Schwachhofer \cite{AyJostLeSchwachhofer17}. 
For an informative historical discussion, as well as for an excellent introduction to the field 
and its wide range of applications, we refer to the recent book of Amari \cite{Amari16}, 
as well as a survey paper of Morozova and Chentsov \cite{ChentsovMorozova91}. 

In the classical approach, one considers finite dimensional families of probability density functions on $M$ 
whose elements are (smoothly) parametrized by open subsets $\Sigma$ of the Euclidean space 
$$ 
\mathcal{S} 
= 
\big\{ 
\rho = \rho_{s_1, \dots, s_k} \in \mathfrak{Dens}(M) \mid  (s_1, \dots, s_k) \in \Sigma \subset \mathbb{R}^k 
\big\} \subset \mathfrak{Dens}(M). 
$$ 

Points of $M$ are viewed as random samples from some (typically unknown) distribution 
$\rho_{s_1, \dots, s_k}$ where $s_1, \dots, s_k$ are certain statistical parameters. 
When equipped with the structure of a smooth manifold, 
the set $\mathcal{S}$ is referred to as a \emph{$k$-dimensional statistical model}\index{statistical model} with $s_1, \dots, s_k$ playing the role of local coordinates. 
Rao \cite{Rao45} introduced further structure 
by defining at each point a $k \times k$ symmetric matrix 
\begin{equation} \label{eq:fdFR} 
g_{ij} 
= 
\int_M \frac{\partial\log{\rho}}{\partial s_i} \frac{\partial\log{\rho}}{\partial s_j} \, \rho \, d\mu 
\qquad 
1 \leq i, j \leq k \,,
\end{equation} 
which in the positive definite case defines a Riemannian metric on $\mathcal{S}$ called the \emph{Fisher-Rao (information) metric}.\index{Fisher-Rao metric} 
The significance of this metric for mathematical statistics was immediately recognized 
by mathematicians, see e.g., the work of Chentsov~\cite{Chentsov82}.  

In our setting, we shall regard a statistical model $\mathcal{S}$ 
as a $k$-dimensional Riemannian submanifold of the Fr\'echet manifold of smooth probability densities $\mathfrak{Dens}(M)$ on the underlying compact $n$-dimensional sample space manifold $M$. 
The next result shows that the metric $g_{ij}$ defined via \eqref{eq:fdFR} is in fact the same as the metric induced on $\mathcal{S}$ by the homogeneous Sobolev metric defined in the previous section on the full diffeomorphism group $\mathfrak{D}(M)$. 
\begin{theorem}\label{thm:idFR} 
The right-invariant Sobolev $\dot{H}^1$ metric \eqref{eq:homH1met} 
on the Fr\'echet Lie group $\mathfrak{D}(M)$ of diffeomorphisms of a compact Riemannian manifold $M$ 
descends to a (weak) Riemannian metric on the quotient space of right cosets $\mathfrak{D}_\mu(M)\backslash\mathfrak{D}(M)\simeq \mathfrak{Dens}(M)$. 
Furthermore, when restricted to a statistical model $\mathcal S$, this metric coincides (up to a constant multiple) with the standard Fisher-Rao metric~ \eqref{eq:fdFR}. \index{$\dot H^1$-metric}
\end{theorem} 
\begin{proof} 
First, we check that the homogeneous Sobolev $\dot{H}^1$ metric 
satisfies the required descent condition of Prop. \ref{prop:descend}. 
That is, we need to verify the formula \eqref{eq:dsc} for 
$\mathfrak{G} = \mathfrak{D}(M)$ and $\mathfrak{H} = \mathfrak{D}_\mu(M)$ 
where 
$\mathrm{ad}_w v = [v,w]$ is the Lie bracket of vector fields on $M$. 
Given any vector fields $u, v, w$ with $\mathrm{div}_\mu w =0$, 
we compute 
\begin{align*} 
\langle \mathrm{ad}_w v, u \rangle_{\dot{H}^1} 
+ 
\langle v, \mathrm{ad}_w u \rangle_{\dot{H}^1} 
&= 
- \frac{1}{4} \int_M 
\big( \mathrm{div}_\mu [w,v] \, \mathrm{div}_\mu u + \mathrm{div}_\mu [w,u] \, \mathrm{div}_\mu v \big)
d\mu 
\\ 
&= 
- \frac{1}{4} \int_M \Big\{ 
\big( 
\langle w, \nabla\mathrm{div}_\mu v \rangle - \langle v, \nabla\mathrm{div}_\mu w \rangle 
\big) \, \mathrm{div}_\mu u 
\\ 
& \hspace{2cm}+ 
\big( 
\langle w, \nabla\mathrm{div}_\mu u \rangle - \langle u, \nabla\mathrm{div}_\mu w \rangle 
\big) \, \mathrm{div}_\mu v 
\Big\} d\mu 
\\ 
&= 
\frac{1}{4} \int_M (\mathrm{div}_\mu w \cdot \mathrm{div}_\mu v \cdot \mathrm{div}_\mu u) \, d\mu 
= 0. 
\end{align*} 
This shows that the homogeneous Sobolev (degenerate) metric \eqref{eq:homH1met} on $\mathfrak{D}(M)$ 
descends to a non-degenerate metric on the quotient $\mathfrak{D}_\mu(M)\backslash\mathfrak{D}(M)$. 

For the second statement it will be convenient to carry out the calculations 
directly in the group $\mathfrak{D}(M)$. 
Given any vectors $v, w$ in $\mathfrak{X}=T_e\mathfrak{D}(M)$ consider a two-parameter family of diffeomorphisms 
$s_1, s_2 \mapsto \xi_{s_1,s_2}$ in $\mathfrak{D}(M)$ 
such that 
$\xi(0,0)=e$ with $\partial\xi/\partial s_1 (0,0) = v$ and $\partial\xi/\partial s_2 (0,0) = w$. 
Their right-translations 
$v\circ\xi_{s_1,s_2}$ and $w\circ\xi_{s_1,s_2}$ 
are the corresponding variation vector fields along the surface defined by the family. 

If $\rho$ is the Jacobian of $\xi_{s_1,s_2}$ computed with respect to 
the reference volume $\mu$ then \eqref{eq:fdFR} assumes the form 
$$ 
g(v, w) 
= 
\int_M 
\frac{\partial}{\partial s_1} \big( \log{\mathrm{Jac}_\mu\xi_{s_1,s_2}} \big) 
\frac{\partial}{\partial s_2} \big( \log{\mathrm{Jac}_\mu\xi_{s_1,s_2}} \big) 
\, 
\mathrm{Jac}_\mu \xi_{s_1,s_2} 
\, d\mu. 
$$ 
Since 
$$ 
\frac{\partial}{\partial s_1} \big( \mathrm{Jac}_\mu \xi_{s_1,s_2} \big) 
= 
\mathrm{div}_\mu v \circ \xi_{s_1,s_2} \cdot \mathrm{Jac}_\mu\xi_{s_1,s_2} 
$$ 
and similarly for the other partial derivative, using these formulas and changing variables in the integral 
we now obtain 
\begin{align*} 
g(v, w) 
&= 
\int_M 
\frac{ \frac{\partial}{\partial s_1}\mathrm{Jac}_\mu \xi_{s_1,s_2} \frac{\partial}{\partial s_2} \mathrm{Jac}_\mu \xi_{s_1,s_2}}
{\mathrm{Jac}_\mu\xi_{s_1,s_2}} 
\Big\vert_{s_1=s_2=0} d\mu 
\\ 
&= 
\int_M (\mathrm{div}_\mu v \circ \xi) {\cdot} (\mathrm{div}_\mu\circ\xi) \: \mathrm{Jac}_\mu\xi \, d\mu 
\\ 
&= 
\int_M 
\mathrm{div}_\mu v \cdot \mathrm{div}_\mu w \, d\mu 
\\ 
&= 
4 \langle v, w \rangle_{\dot{H}^1} ,
\end{align*} 
which proves the theorem.  
\end{proof} 

\begin{definition}
    We shall refer to the Riemannian metric in Theorem~\ref{thm:idFR} as the \emph{infinite dimensional Fisher-Rao metric} \index{Fisher-Rao metric} and denote it by the symbol $\mathcal{FR}$.
Explicitly, with $T\mathfrak{Dens}(M)$ as in~\eqref{eq:TDens}, it is given by
\begin{equation} \label{eq:FRmetric} 
	\mathcal{FR}_\varrho(\dot\varrho,\dot\varrho) = \int_M \left(\frac{\dot\varrho}{\varrho}\right)^2 \varrho .
\end{equation} 
\end{definition}
\index{Fisher-Rao metric}
That this metric is an infinite-dimensional version of the original Fisher-Rao metric~\eqref{eq:fdFR} was first demonstrated by Friedrich~\cite{Friedrich91}.
Indeed, if $\mathcal{S}$ is a statistical manifold as above, and $\varrho(t)=\rho(s_1(t),\ldots,s_k(t))\mu \in\mathcal{S}$, then
\begin{align*}
	& \mathcal{FR}_{\varrho(t)}(\dot\varrho,\dot\varrho) = \int_M \left(\frac{d}{dt}\log(\rho(s_1,\ldots,s_k)) \right)^2\varrho(t) \\
	&= \sum_{i=1}^k\sum_{j=1}^k \int_M  \left(\frac{\partial}{\partial s_i}\log(\rho(s_1,\ldots,s_k)) \frac{ds_i}{dt}\right)\left(\frac{\partial}{\partial s_j}\log(\rho(s_1,\ldots,s_k)) \frac{ds_j}{dt}\right) \varrho(t) \\ 
	&= \sum_{i=1}^k\sum_{j=1}^k g_{ij} \frac{ds_i}{dt}\frac{ds_j}{dt}.
\end{align*}
Thus, the restriction of $\mathcal{FR}$ to the statistical manifold $\mathcal{S}$ is precisely the Fisher-Rao metric~\eqref{eq:fdFR}.

As an immediate consequence of Theorem~\ref{thm:idFR} we have 
\begin{corollary} \label{cor:pos-cur} 
The space of smooth densities $\mathfrak{Dens}(M) = \mathfrak{D}_\mu(M)\backslash\mathfrak{D}(M)$ 
equipped with the infinite dimensional Fisher-Rao metric $\mathcal{FR}$ 
has positive constant curvature $4/\mu(M)$.\index{curvature}
\end{corollary} 
\begin{proof} 
Since $\Psi$ is an isometry by Theorem \ref{thm:isometry} 
the corollary follows directly from the previous theorem 
and the fact that the sectional curvature of a sphere in a Hilbert space of radius $r>0$ 
is precisely $1/r^2$. 
\end{proof} 

It is perhaps worth noting that if the volume of $M$ grows to infinity then the above corollary 
implies that the space of densities $\mathfrak{Dens}(M)$ becomes ``flatter" in the Fisher-Rao metric $\mathcal{FR}$. 

There is an analogue for $\mathcal{FR}$ of the well-known Chentsov uniqueness theorem 
for the Fisher-Rao metric $\mathcal{FR}$, 
according to which the former is essentially unique among those metrics on $\mathfrak{D}(M)$ that descend to 
the base $\mathfrak{Dens}(M)$ of right cosets. 
%
%
\begin{theorem}\cite{BauerBruverisMichor16}\label{thm:invarFR} 
Let $M$ be a compact Riemannian manifold of dimension $n \geq 2$ without boundary. 
Any (weak) Riemannian right-invariant metric on $\mathfrak{D}(M)$ which descends to 
the quotient space of right cosets 
$\mathfrak{Dens}(M) = \mathfrak{D}_\mu(M)\backslash\mathfrak{D}(M)$ 
(i.e., a metric on densities on $M$ invariant with respect to the natural action of diffeomorphisms)
is a multiple of the infinite dimensional Fisher-Rao metric $\mathcal{FR}$ given by \eqref{eq:FRmetric}. 
\end{theorem} 
\begin{proof}[Proof (sketch)] 
As in the finite dimensional case this is essentially a consequence of invariance properties 
under the action of diffeomorphisms. 
It is based on the following ideas. 
Since the metric on $\mathfrak{Dens}(M)$ should be right-invariant, it is enough to consider it on one tangent space $T_\mu\mathfrak{Dens}(M)$.
At this base-point, any metric has to have a form $(\dot\mu_1,\dot\mu_2)_\mu:=\langle G, \dot\mu_1\otimes \dot\mu_2\rangle $,
where $\dot\mu_1,\dot\mu_2$ are tangent vectors at $\mu\in\mathfrak{Dens}(M)$ (and hence have zero means), while the kernel $G\in \mathcal D'(M\times M)$. First the authors in \cite{BauerBruverisMichor16} prove by using the invariance that the support of the kernel is on the diagonal 
$\Delta\subset M\times M$. This is the most tedious part of the proof. Once this is established, 
one envokes H\"ormander's theorem \cite{Hormander83}, which 
implies that  the corresponding metric must be uniquely represented by a differential ``inertia operator" $L$ on densities:
$(\dot\mu_1,\dot\mu_2)_\mu=\int_M \frac{\dot\mu_1}{\mu} L(\frac{\dot\mu_2}{\mu})\,\mu$, where $L=\sum_{|\alpha|\le k}C_\alpha \partial^\alpha$.
Again, by using the invariance with respect to the $\mu$-preserving diffeomorphism action, one first shows that $L$ must have only constant coefficients, and then that it could be only a multiple of the identity operator, $L=C\cdot \mathrm{Id}$.
This implies that $(\dot\mu_1,\dot\mu_2)_{\mu}=C \int_M  \frac{\dot\mu_1}{\mu}\frac{\dot\mu_2}{\mu}\, \mu$, which boils down to the Fisher-Rao formula.
We refer to \cite{BauerBruverisMichor16} for details.
\end{proof} 
\begin{remark}
A similar question about a complete description of diffeomorhism-invariant metrics on densities on a manifold with boundary is still open.

Note that when restricted to statistical manifolds, which are finite-dimensional
submanifolds of $\mathfrak{Dens}(M)$, the infinite-dimensional Fisher-Rao metric reduces 
to the ``classical" finite-dimensional one.
Its uniqueness result on finite sample
spaces was established in \cite{Chentsov82}   and extended it to infinite sample spaces in \cite{AyJostLeSchwachhofer17}. The latter  result is proved under the invariance assumtion  not only under smooth diffeomorphisms, but under all sufficient statistics. This is a stronger invariance assumption, allowing the authors to consider step function-type probability densities and reduce the problem
to the finite-dimensional case of \cite{Chentsov82}. 
\end{remark}



%
%
%

\chapter{Fisher-Rao geodesics}\label{ch:FRgeodesics}

\section{Geodesic equations and complete integrability} 
\label{subsect:Euler-Arnold} 

For a deeper insight into the infinite dimensional analogue $\mathcal{FR}$ of the Fisher-Rao metric 
on $\mathfrak{Dens}(M)$ we can turn to the study of its geodesics. 
Since the metric is invariant the associated geodesic equation can be derived via a reduction procedure 
as an Euler-Arnold equation on the quotient space $\mathfrak{D}_\mu(M)\backslash\mathfrak{D}(M)$. 
In fact, it will be convenient to work with the right-invariant Sobolev $\dot{H}^1$ metric 
"upstairs" on the total space $\mathfrak{D}(M)$ of all diffeomorphisms. 
\begin{theorem} \label{thm:EA-Dens} 
The Euler-Arnold equation of the homogeneous metric \eqref{eq:homH1met} has the form 
\begin{equation} \label{eq:EA-Dens} 
\nabla \mathrm{div}_\mu u_t + \mathrm{div}_\mu u {\cdot} \nabla\mathrm{div}_\mu u 
+ 
\nabla \langle u, \nabla \mathrm{div}_\mu u \rangle 
= 0 
\end{equation} 
or, equivalently, 
\begin{equation} \label{eq:EA-Dens-int} 
h_t + \langle u, \nabla h \rangle + \frac{1}{2} h^2 
= 
-\frac{1}{2\mu(M)} \int_M h^2 \, d\mu 
\end{equation} 
where $h = \mathrm{div}_\mu u$. \index{$\dot H^1$-metric}
\end{theorem} 
\begin{proof} 
Using the general Euler-Arnold equation and Remark~\ref{rem:EA-quotient} of Section~\ref{subsub:EAeq} 
we only need to compute the coadjoint operator with respect \eqref{eq:homH1met}. 
On the one hand, from \eqref{eq:homH1met} for any $u, v$ and $w$ in $T_e\mathfrak{D}(M)$ 
we have 
$$ 
\langle \mathrm{ad}^\top_v u, w \rangle_{\dot{H}^1} 
= 
-\frac{1}{4} \int_M \langle \nabla \mathrm{div}_\mu \mathrm{ad}^\top_v u, w \rangle \, d\mu. 
$$ 
On the other hand, using \eqref{eq:co-ad} we compute 
\begin{align*} 
\langle \mathrm{ad}^\top_v u, w \rangle_{ \dot{H}^1 } 
&= 
\langle u, \mathrm{ad}_v w \rangle_{\dot{H}^1} 
= 
\frac{1}{4} \int_M \mathrm{div}_\mu u \, \mathrm{div}_\mu [v, w] \, d\mu 
\\ 
&= 
\frac{1}{4} \int_M \big( 
\mathrm{div}_\mu u \, \langle v, \nabla \mathrm{div}_\mu w \rangle 
- 
\mathrm{div}_\mu u \, \langle \nabla \mathrm{div}_\mu v, w \rangle 
\big) d\mu 
\\ 
&= 
- \frac{1}{4} \int_M 
\mathrm{div}_\mu \big( \mathrm{div}_\mu u \, v \big) \, \mathrm{div}_\mu w 
\, d\mu 
- 
\frac{1}{4} \int_M 
\langle \mathrm{div}_\mu u \, \nabla \mathrm{div}_\mu v, w \rangle \,
d\mu 
\\ 
&= 
\frac{1}{4} \int_M 
\big\langle 
\nabla \mathrm{div}_\mu (\mathrm{div}_\mu u \, v ) 
- 
\mathrm{div}_\mu u \, \nabla\mathrm{div}_\mu v, w 
\big\rangle \, d\mu. 
\end{align*} 
Since $w$ is an arbitrary vector field on $M$, comparing the two integral expressions above 
we obtain 
\begin{align*} 
\nabla \mathrm{div}_\mu (\mathrm{ad}^\top_v u) 
&= 
- \nabla \mathrm{div}_\mu (\mathrm{div}_\mu u \cdot v ) 
+ 
\mathrm{div}_\mu u \cdot \nabla\mathrm{div}_\mu v 
\\ 
&= 
- \nabla \langle \nabla \mathrm{div}_\mu u, v \rangle - \nabla ( \mathrm{div}_\mu u \cdot \mathrm{div}_\mu v ) 
+ 
\mathrm{div}_\mu u \cdot \nabla\mathrm{div}_\mu v. 
\end{align*} 
Substituting into \eqref{eq:Euler-Arnold} yields the desired Euler-Arnold equation \eqref{eq:EA-Dens}. 
\end{proof} 

Observe that in the one-dimensional case when $M=\mathbb{T}$ differentiating equation \eqref{eq:EA-Dens-int} 
with respect to $x$ gives the Hunter-Saxton equation \eqref{eq:HS} of Example \ref{ex:HS}.

\subsection{The Cauchy problem: explicit solutions} 
The question of wellposedness of the Cauchy problem for a nonlinear evolution equation 
subject to an appropriate initial data 
involves constructing a unique solution which belongs to a given function space, satisfies both the equation 
and the initial condition, and depends at least continuously on the data. 
In the case of the general Euler-Arnold equation \eqref{eq:Euler-Arnold} 
this question can be studied 
either by working directly with the \emph{partial differential equation} 
or indirectly 
by reformulating it in terms of the associated geodesic flow in the group (or the homogeneous space). 
One advantage of the latter approach is that in a suitable Banach space setting 
(such as Sobolev $H^s$ with $s>n/2+1$ or H\"older $\mathcal{C}^{1,\alpha}$ with $0<\alpha<1$) 
the geodesics can be often constructed using Banach-Picard iterations 
as solutions of an \emph{ordinary differential equation}. 
We point out however that the two formulations of the Cauchy problem 
(one in the Lie algebra and the other in the group) 
are in general not equivalent 
since in the latter case the data-to-solution map is typically smooth, 
while in the former it is at best continuous in any reasonable Banach space topology. 

As it turns out, in our case we can solve the Euler-Arnold equations of Theorem~\ref{thm:EA-Dens} 
by deriving explicit formulas for the corresponding solutions. 
\begin{theorem} \label{thm:EA-Dens-sol} 
Let $h = h(t,x)$ be the solution of \eqref{eq:EA-Dens-int} with the initial condition 
\begin{equation} \label{eq:EA-Dens-ic} 
h(0,x) = \mathrm{div}_\mu u_0(x). 
\end{equation} 
Let $t \mapsto \eta(t)$ be the flow of the corresponding velocity field $u = u(t,x)$, 
that is 
$$ 
\frac{d}{dt} \eta(t,x) = u(t, \eta(t,x)) \,,
\qquad 
\eta(0,x) = x. 
$$ 
Then, we have 
\begin{align} \label{eq:EA-sol} 
h(t, \eta(t,x)) 
&= 
2\kappa \tan{ \Big( \arctan{ \frac{ \mathrm{div}_\mu u_0(x)}{2\kappa} } - \kappa t \Big) } \,,
\\ 
\text{where} \quad 
\kappa^2 
&= 
\frac{1}{4\mu(M)} \int_M (\mathrm{div}_\mu u_0)^2 d\mu. 
\end{align} 
Furthermore, the Jacobian of the flow is
\begin{equation} \label{eq:EA-Jac} 
\mathrm{Jac}_\mu (\eta(t,x)) 
= 
\Big( \cos{\kappa t} + \frac{\mathrm{div}_\mu u_0(x)}{2\kappa} \sin{\kappa t} \Big)^2. 
\end{equation} 
\end{theorem} 
\begin{proof} 
If $\theta(t,x)$ is a smooth real-valued function then the chain rule gives 
$$ 
\frac{d}{dt} \theta(t,\xi(t,x)) 
= 
\frac{\partial\theta}{\partial t}(t, \xi(t,x)) 
+ 
\big\langle u(t, \xi(t,x)), \nabla\theta (t, \xi(t,x)) \big\rangle. 
$$ 
From this formula and from \eqref{eq:EA-Dens-int} we obtain an equation for $\theta = h\circ\xi$, 
namely 
\begin{equation} \label{eq:EA-ODE} 
\frac{d\theta}{dt} + \frac{1}{2} \theta^2 = - C(t) 
\end{equation} 
where 
$C(t) = (2\mu(M))^{-1} \int_M h^2 d\mu$. 
Observe that $C(t)$ is, in fact, independent of the time variable $t$ since 
\begin{align*} 
\mu(M) \frac{dC}{dt}(t) 
&= 
\int_M h h_t d\mu 
= 
\int_M \mathrm{div}_\mu u {\cdot} \mathrm{div}_\mu u_t d\mu 
\\ 
&= 
- \int_M \langle u \nabla\mathrm{div}_\mu u \rangle {\cdot} \mathrm{div}_\mu u \, d\mu 
- 
\frac{1}{2} \int_M (\mathrm{div}_\mu u)^3 d\mu 
= 0 \,,
\end{align*} 
where the last step follows at once by integrating by parts. 

Set $C = 2\kappa^2$. 
Then, for any fixed $x \in M$ the solution of the ODE in \eqref{eq:EA-ODE} 
has the form 
$$ 
\theta(t) = 2\kappa \tan{(\arctan{( f(0)/2\kappa)} - \kappa t)} \,,
$$ 
which is precisely \eqref{eq:EA-sol}. 

Finally, to find the formula for the Jacobian we first compute the time derivative of 
$(\mathrm{Jac}_\mu \eta) \, \mu$ to get 
\begin{equation}\label{eq:Jac_evol}
\frac{d}{dt} ( \mathrm{Jac}_\mu \eta) \, \mu 
= 
\frac{d}{dt} \eta^\ast \mu 
= 
\eta^\ast ( \mathcal{L}_u \mu ) 
= \theta \,( \mathrm{Jac}_\mu \eta) \, \mu \,,
\end{equation}
which gives a differential equation, whose solution can be now easily verified to be \eqref{eq:EA-Jac} by making use of the formulas \eqref{eq:EA-sol} . 
\end{proof} 
\begin{remark} 
The formula \eqref{eq:EA-Jac} for the Jacobian of the flow $\eta(t)$ can be viewed in light of the correspondence between the geodesics of $\mathcal{FR}$ in $\mathfrak{Dens}(M)$ and those on the round sphere in a Hilbert space established in Theorem~\ref{thm:isometry}. 
Indeed, the map 
$$ 
t \mapsto \sqrt{ \mathrm{Jac}_\mu(\eta(t,x)) } 
= 
\cos{\kappa t} + \frac{\mathrm{div}_\mu u_0(x)}{2\kappa} \sin{\kappa t} 
$$ 
describes the great circle on the unit sphere $S^\infty_{L^2}  \subset L^2(M, d\mu)$. 
\end{remark} 

\begin{remark}
    From Theorem~\ref{thm:EA-Dens-sol} one directly recovers the geodesic equation for the Fisher--Rao metric in Hamiltonian form.
    Indeed, the Lagrangian on $T\mathfrak{Dens}(M)$ is
    \[
        L(\varrho,\dot\varrho) = \frac{1}{2}\int_M \left(\frac{\dot\varrho}{\varrho}\right)^2 \varrho .
    \]
    The momentum variable $\theta \in T_\varrho^*\mathfrak{Dens}(M)$ is the co-set (defined up to a global constant) given by the Legendre transformation
    \[
        \theta \coloneqq \frac{\delta L}{\delta\dot\varrho} = \frac{\dot\varrho}{\varrho} \; ,
    \]
	which corresponds to selecting the constant in the co-set $[\theta]$ so that $\int_M \theta\varrho = 0$.
    The corresponding Hamiltonian is 
    \[
        H(\varrho,\theta) = \frac{1}{2}\int_M \theta^2 \varrho
    \]
    which gives
    \[
        \dot\theta = - \frac{\delta H}{\delta \varrho} = -\frac{1}{2}\theta^2 - C,
    \]
    with the constant $C$ as in equation \eqref{eq:EA-ODE}.
    Then the corresponding evolution of the configuration variable $\varrho$ is
    \[
        \dot\varrho = \frac{\delta H}{\delta\theta} = \theta \varrho.
    \]
    Notice that this is precisely the evolution for $\mathrm{Jac}_\mu = \varrho/\mu$ in equation \eqref{eq:Jac_evol}.
\end{remark}

\subsection{The Cauchy problem: breakdown of solutions} 
Using the explicit formulas of Theorem~\ref{thm:EA-Dens-sol} it possible to make conclusions 
regarding long time behaviour of solutions. 
For example, it turns out that all smooth (classical) solutions of the Euler-Arnold equation \eqref{eq:EA-Dens-int} 
must break down in finite time. 
\begin{theorem} \label{thm:EA-breakdown} 
The lifespan of any (smooth) solution of the Cauchy problem \eqref{eq:EA-Dens-int}-\eqref{eq:EA-Dens-ic} 
constructed in Thmeorem~\ref{thm:EA-Dens-sol} is 
\begin{equation} \label{eq:EA-lifespan} 
0 < T_{\mathrm{max}} 
= 
\frac{\pi}{2\kappa} 
+ 
\frac{1}{\kappa} \arctan{ \big( \frac{1}{2\kappa} \inf_{x\in M} \mathrm{div}_\mu u_0 (x) \big) }. 
\end{equation} 
Furthermore, $\| u(t) \|_{\mathcal{C}^1} \nearrow + \infty$ as $t \to T_{\mathrm{max}}$. 
\end{theorem} 
\begin{proof} 
The theorem follows at once from formula \eqref{eq:EA-sol} and $\mathrm{div}_\mu u = h$. 
Alternatively, observe that formula \eqref{eq:EA-Jac} implies that the flow of $u(t,x)$ ceases to be 
a diffeomorphism at the critical time $t=T_{\mathrm{max}}$. 
\end{proof} 

This result can be also interpreted geometrically as saying that the corresponding geodesics of 
the Fisher-Rao metric $\mathcal{FR}$ leave the set of positive densities $\mathfrak{Dens}(M)$ 
and can be no longer lifted to a smooth curve of diffeomorphisms in $\mathfrak{D}(M)$. 

\subsection{Complete integrability} 
\label{subsec:c-int} 

The equations induced by the $g_{\dot{H}^1}$ metric in Theorem~\ref{thm:EA-Dens} can be viewed as an infinite dimensional integrable system.\index{integrable system} 
%
\begin{theorem} \label{thm:integ}
The Euler-Arnold equation of the Fisher-Rao metric $\mathcal{FR}$ on the space of densities $\mathfrak{Dens}(M)$ 
is an infinite dimensional completely integrable dynamical system. 
\end{theorem} 
\begin{proof} 
This is essentially a consequence of the fact that \eqref{eq:EA-Dens-int} describes the geodesic flow 
on the sphere $S^\infty_{L^2}(r) \subset L^2(M, d\mu)$ of radius $r$ and it therefore admits infinitely many first integrals 
in direct analogy with the finite dimensional case $S^{n-1}(r) \subset \mathbb{R}^n$. 
We refer to \cite{KhesinLenellsMisiolekPreston13} for further details. 
\end{proof} 
\begin{remark} 
It is natural to expect that in any dimension $n$ the Euler-Arnold equation \eqref{eq:EA-Dens-int} 
is integrable in that it admits a bi-hamiltonian structure. 
In the one-dimensional case of the Hunter-Saxton equation \eqref{eq:HS} this fact is well known, 
see e.g., \cite{KhesinMisiolek03}.

Another direction, explored in \cite{SavarinoAlbers}, is related to Fisher--Rao metrics on graphs.
It developed the viewpoint of geometric mechanics to study flows of time-dependent 
assigning labels to datapoints based on the metric distance between labels and data, one of important problems in information geometry.
\end{remark} 

\section{The metric space structure of the space of densities} 
\label{subsec:met-structure} 
Consider two smooth measures $\lambda$ and $\nu$ on $M$ that are absolutely continuous 
with respect to the reference measure $\mu$ and have the same total volume $\mu(M)$. 
Let $d\lambda/d\mu$ and $d\nu/d\mu$ be the respective Radon-Nikodym derivatives. 
\begin{theorem} \label{thm:FRdist} 
The Riemannian distance between measures $\lambda$ and $\nu$ induced by 
the infinite dimensional Fisher-Rao metric $\mathcal{FR}$ on $\mathfrak{Dens}(M)$ is 
\begin{equation} \label{eq:FRdist} 
\mathrm{dist}_{\mathcal{FR}} (\lambda, \nu) 
= 
\sqrt{\mu(M)} \arccos{\bigg( 
\frac{1}{\mu(M)} \int_M \sqrt{ \frac{d\lambda}{d\mu} \frac{d\nu}{d\mu} } \, d\mu 
\bigg)}. 
\end{equation} 
\end{theorem} 

Equivalently, this distance can be computed as the Riemannian distance of \eqref{eq:homH1met} 
between two diffeomorphisms $\xi$ and $\zeta$ in $\mathfrak{D}(M)$ 
which map the volume form $\mu$ to $\lambda$ and $\nu$, respectively, from the formula 
$$ 
\mathrm{dist}_{\dot{\mathcal{H}}^1} (\xi, \zeta) 
= 
\sqrt{\mu(M)} \arccos{\bigg( 
\frac{1}{\mu(M)} \int_M \sqrt{ \mathrm{Jac}_\mu\xi \cdot \mathrm{Jac}_\mu\zeta } \, d\mu 
\bigg)}. 
$$ 
\begin{proof} 
Let $f^2 = d\lambda/d\mu$ and $g^2 = d\nu/d\mu$. 
If $\lambda = \xi^\ast\mu$ and $\nu = \zeta^\ast\mu$ then using the isometry given by 
the square root map $\Psi$ of Theorem \ref{thm:isometry} 
it suffices to compute the distance between the functions 
$\Psi(\xi)=f$ and $\Psi(\zeta)=g$ 
considered as points on the sphere $S^\infty_{L^2}(r)$ of radius  $r=\sqrt{\mu(M)}$ 
with the metric induced from $L^2(M,d\mu)$. 
However, since geodesics of this metric are precisely the great circles, it follows that the length of 
the corresponding arc joining $f$ and $g$ is 
$$ 
r \arccos{ \bigg( r^{-2} \int_M fg \, d\mu \bigg) } \,,
$$ 
which is formula \eqref{eq:FRdist}. 
\end{proof} 
Recall that the diameter of a Riemannian manifold is defined as the supremum of the Riemannian distances 
between its points. Thus, in particular 
$$ 
\mathrm{diam}_{\mathcal{FR}} \big( \mathfrak{Dens}(M) \big) 
= 
\sup{ \Big\{ \mathrm{dist}_{\mathcal{FR}}(\lambda, \nu)\mid \lambda,\nu \in \mathfrak{Dens}(M) \Big\} }. 
$$ 
\begin{corollary} \label{cor:diam} 
The diameter of $\mathfrak{Dens}(M)$ with the metric $\mathcal{FR}$ 
equals $\pi\sqrt{\mu(M)}/2$, i.e., 
it is a quarter  of the circumference of the sphere in $L^2(M,d\mu)$ of radius $\sqrt{\mu(M)}$. 
\end{corollary} 
\begin{proof} 
The upper bound follows easily from the formula \eqref{eq:FRdist} since the argument of 
the arccos function is always between $0$ and $1$. 
To see that it can be arbitrarily close to $0$ it suffices to choose the functions $f$ and $g$ 
as in the proof of Theorem \ref{thm:FRdist} with supports in disjoint subsets. 
\end{proof} 

The Riemannian distance of the metric $\mathcal{FR}$ on $\mathfrak{Dens}(M)$ is closely related to 
the \emph{Hellinger distance}. 
Recall that the latter is defined to be 
$$ 
\mathrm{dist}_H^2 (\lambda, \nu) 
= 
\int_M \Big( \sqrt{d\lambda/d\mu} - \sqrt{ d\nu/d\mu} \, \Big)^2 d\mu 
$$ 
for any probability measures $\lambda$ and $\nu$ on $M$ 
which are absolutely continuous with respect to $\mu$. 
It is readily checked that if $\lambda$ and $\nu$ coincide then 
$\mathrm{dist}_H(\lambda, \nu) = 0$ 
and 
if they are mutually singular then 
$\mathrm{dist}_H(\lambda, \nu) = \sqrt{2}$. 
Of course, when comparing it with \eqref{eq:FRdist} one needs to normalize all the measures involved. 
\begin{remark} \label{rem:Bh} 
The Hellinger distance\index{Hellinger distance} is also related to the so-called \emph{Bhattacharyya affinity} $BC$ 
by the formula 
$$ 
\mathrm{dist}_H^2 (\lambda, \nu) = 2 \big( 1 - BC(\lambda, \nu) \big) 
$$ 
see e.g., \cite{Chentsov82} for more information. 
\end{remark} 

The following two corollaries can be readily verified using the isometry property of Theorem \ref{thm:isometry}. 
\begin{corollary} 
The Hellinger distance between two normalized densities 
$d\lambda = f^2 d\mu$ and $d\nu = g^2 d\mu$ 
is equal to the distance in the Hilbert space $L^2(M,d\mu)$ 
between the functions $f$ and $g$ considered as points on the unit sphere $S^\infty_{L^2}$. 
\end{corollary} 
\begin{corollary} 
The Bhattacharyya coefficient $BC(\lambda, \nu)$ of two normalized densities 
$d\lambda = f^2 d\mu$ and $d\nu = g^2 d\mu$ 
is equal to the inner product of the corresponding (positive) functions $f$ and $g$ in $L^2(M, d\mu)$, 
that is $BC(\lambda, \nu) = \int_M fg \, d\mu$. 
\end{corollary} 

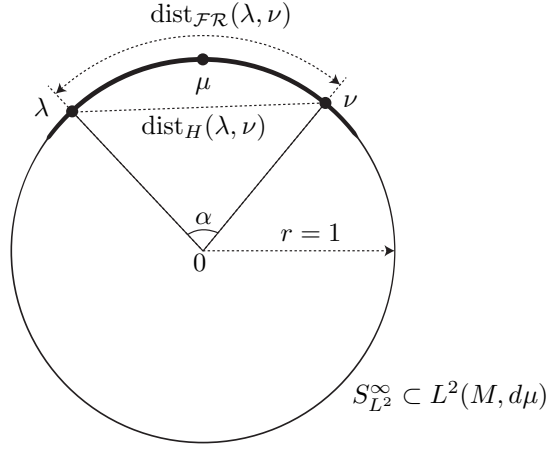
\begin{figure}
\begin{center}
\bigskip
\bigskip
 \begin{overpic}[width=.45\textwidth]{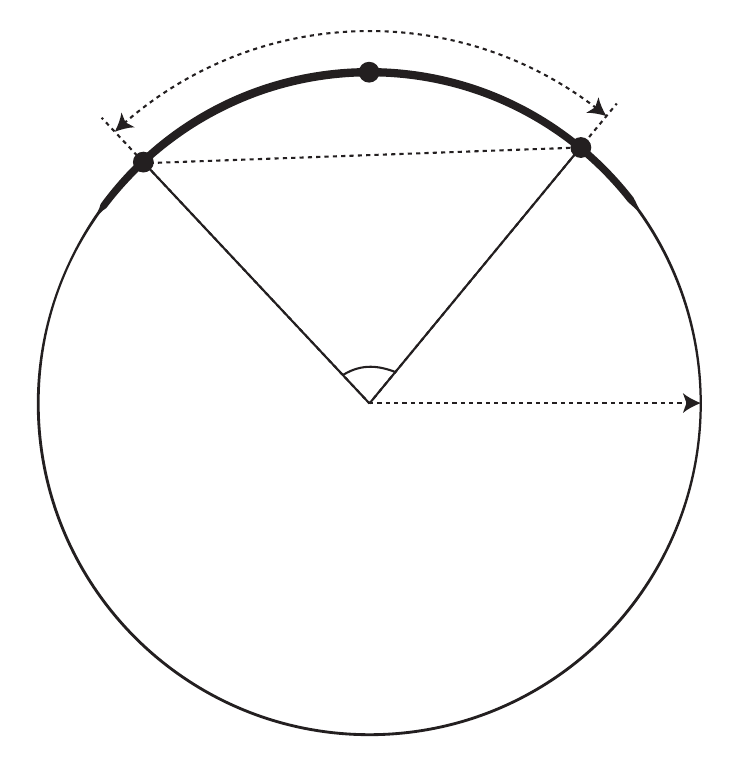}
      \put(46,54){$\alpha$}
      \put(45.5,43.5){$0$}
      \put(10,78){$\lambda$}
      \put(79,80){$\nu $}
      \put(46,84){$\mu$}
      \put(34,74){$\mathrm{dist}_H (\lambda, \nu) $}
      \put(36, 99){$\mathrm{dist}_{\mathcal{FR}} (\lambda, \nu)$}
      \put(65, 50){$r = 1$}
      \put(81, 14){$S^\infty_{L^2} \subset L^2(M, d\mu)$}
    \end{overpic}
\caption{The Hellinger distance $\mathrm{dist}_H (\lambda, \nu) $ and the spherical Hellinger distance $\mathrm{dist}_{\mathcal{FR}} (\lambda, \nu)$ between two densities $d\lambda = f^2 d\mu$ and $d\nu = g^2 d\mu$  in $S^\infty_{L^2}$. The thick arc represents the image of 
        $\mathfrak{D}(M)$ under the square root map $\Psi$.}\label{fig:HellingerDistance}
     \end{center}
\end{figure}

Let $0 < \alpha < \pi/2$ denote the angle between two vectors $f$ and $g$ of the unit sphere in $L^2(M, d\mu)$. 
In this case we have 
$$ 
\mathrm{dist}_H(\lambda, \nu) 
= 
2 \sin{\alpha/2} 
\qquad 
\text{and} 
\qquad 
BC(\lambda, \nu) = \cos{\alpha} 
$$ 
while 
$$ 
\mathrm{dist}_{\mathcal{FR}} (\lambda, \nu) 
= 
\arccos{BC(\lambda, \nu)}. 
$$ 
We can therefore refer to the Riemannian distance of the infinite dimensional Fisher-Rao metric 
on $\mathfrak{Dens}(M)$ as the \emph{spherical Hellinger distance}; 
see Figure \ref{fig:HellingerDistance}.


\section{Factorizations via descending metrics}
\label{sec:factorizations_abstract}

The purpose of this section is to give an abstract, geometric description of the relation between ``distance square'' optimal transport problems and descending metrics.
Thereby, we can generalize the polar factorization result of Brenier~\cite{Br1991} in a geometric framework, which we can then apply in other settings.
For example, as a finite dimensional analogue, we show in Section~\ref{sub:qr} that $QR$~factorisation of square matrices can be seen as polar factorisation corresponding to optimal transport of inner products.
The main result, however, is given in Section~\ref{sec:factorizations_information}, where we apply the framework in an information geometric context via the \emph{optimal information transport} problem, which gives a factorization of diffeomorphisms that is optimal with respect to the information metric \eqref{eq:info_metric} and compatible with the Fisher-Rao metric on $\mathfrak{Dens}(M)$ (instead of the Wasserstein $L^2$ distance, as in optimal mass transport).
For more details on the material presented here, we refer to \cite{Modin15,Modin17}.

Abstractly, we formulate geometric optimal transport problems as follows.
Let $\mathfrak{G}$ be a Lie group acting transitively from the right on a manifold~$\mathfrak{B}$.
We then equip $\mathfrak{G}$ with a cost function $c\colon \mathfrak{G}\times \mathfrak{G}\to \R_{+}$.
This gives us a geometric formulation of the Monge problem to find a transport map $\eta$:
\begin{equation}\label{eq:optimal_transport_abstract}
\text{
Given $b,b'\in \mathfrak{B}$, find $\eta\in \{g\in \mathfrak{G}\mid   b\cdot g = b' \}$ that minimize $c(e,\eta)$.
}
\end{equation}

We consider now the special case $c = \mathrm{dist}^{2}$, where $\mathrm{dist}$ is the geodesic distance of a Riemannian metric~$\mathcal{G}$ on $\mathfrak{G}$.
In particular, we are interested in the case when $\mathcal{G}$ is descending with respect to a fiber bundle structure $\pi\colon \mathfrak{G}\to \mathfrak{B}$.
In addition to the vertical distribution $V\mathfrak{G}$, which is independent of the Riemannian structure, one can then define the horizontal distribution as the sub-bundle of $T\mathfrak{G}$ given by
\begin{equation}\label{eq:horizontal_dist}
	H_g\mathfrak{G} = \{ X\in T_g\mathfrak{G} \mid \mathcal{G}_g(X,V_g\mathfrak{G}) = 0 \}.
\end{equation}
Let $\Pi_H\colon T\mathfrak{G}\to H\mathfrak{G}$ be the bundle map for the orthogonal projection.
The optimal transport problem~\eqref{eq:optimal_transport_abstract} then reduces to the following: 
\begin{enumerate}
	\item find the shortest curve on $\mathfrak{B}$ connecting $b$ and $b'$,
	\item lift that curve to a horizontal curve on~$\mathfrak{G}$, 
	\item take the endpoint as the solution.
\end{enumerate}

\begin{figure}
	\includegraphics{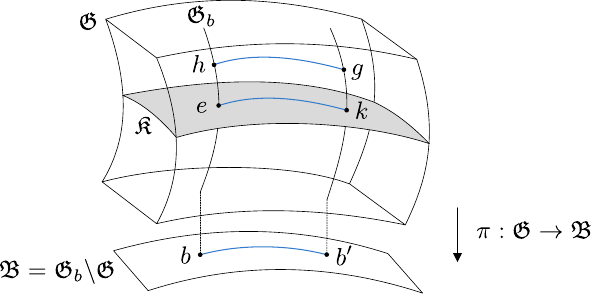}
	\caption{Illustration of the principal bundle structure and the associated Riemannian submersion that give rise to a factorization of elements in the Lie group $\mathfrak{G}$. An arbitrary element $g\in\mathfrak{G}$ is projected to $b' = \pi(g)$. The geodesic in $\mathfrak{B}$ between $b$ and $b'$ is then lifted to a horizontal geodesic in $\mathfrak{G}$ whose endpoint $k$ is an element of the polar cone $\mathfrak{K}$. By construction, $g$ and $k$ belong to the same fiber, which means that $h = gk^{-1}$ is an element of the subgroup $\mathfrak{G}_b$ (the identity fiber).
	Thus, we obtain the factorization $g = hk$.}	\label{fig:polar_cone}
\end{figure}

This simplification utilizes that minimal geodesic curves between fibers are horizontal:

\begin{lemma}[\cite{Modin15}]\label{lem:horiz_shortest_curve}
	Let $(\mathfrak{G},\mathcal{G})$ and $(\mathfrak{B},\mathcal{B})$ be Riemannian manifolds and let $\pi\colon \mathfrak{G}\to \mathfrak{B}$ be a Riemannian submersion. 
	Consider an arbitrary smooth curve $\zeta\colon [0,1]\to \mathfrak{G}$.
	Then there is a unique horizontal curve $\zeta_h\colon [0,1]\to \mathfrak{G}$ such that $\zeta_h(0)=\Pi_H\zeta(0)$ and $\pi\circ\zeta=\pi\circ\zeta_h$.
	The length of the curve~$\zeta_h$ with respect to the Riemannian metric is less than or equal the length of the curve~$\zeta$, with equality if and only if $\zeta$ is horizontal.
\end{lemma}

\begin{proof}
	For each $t\in [0,1]$ there is a unique decomposition $\dot\zeta(t) = v(t) + h(t)$, where $v(t)\in\mathcal{V}_{\zeta(t)}$ and $h(t)\in H_{\zeta(t)}\mathfrak{G}$.
	Thus, we have the curves $t\mapsto v(t) \in V\mathfrak{G}$ and $t\mapsto h(t)\in H\mathfrak{G}$.
	By the projection $\pi$ we also get a curve $\bar\zeta(t) = \pi(\zeta(t)) \in \mathfrak{B}$.
	This curve can be lifted to a horizontal curve as follows.
	Take any time-dependent smooth vector field~$\bar X_t$ on~$\mathfrak{B}$ for which $\bar\zeta$ is an integral curve, i.e., $\dot{\bar\zeta}(t) = \bar X_t(\bar\zeta(t))$.
	Now lift $\bar X_t$ to its corresponding horizontal section $X_t(g) = (T_g\pi)^{-1}\cdot \bar X_t(\pi(g))$.
	(We can do this since $T_g\pi\colon H_g\mathfrak{G}\to T_{\pi(g)}\mathfrak{B}$ is an isomorphism.)
	Next, let $\zeta_h$ be the unique integral curve of $X_t$ with $\zeta_h(0) = \zeta(0)$.
	By construction, $\pi(\zeta_h(t)) = \pi(\zeta(t))$ and $\mathcal{G}(\dot\zeta_h(t),\dot\zeta_h(t)) = \mathcal{G}(h(t),h(t))$.
	Thus, $\mathcal{G}(\dot\zeta_h(t),\dot\zeta_h(t))\leq \mathcal{G}(\dot\zeta(t),\dot\zeta(t))$, with equality if and only if $\dot\zeta(t)\in H\mathfrak{G}$.

	It remains to show that $\zeta_h$ is unique.
	Assume $\zeta_h':[0,1]\to \mathfrak{G}$ is another horizontal curve with $\pi\circ\zeta_h' = \bar\zeta$ and $\zeta_h'(0) = \zeta(0)$.
	By differentiation with respect to $t$ we obtain
	\begin{equation*}
		T_{\zeta_h'(t)}\pi\cdot \dot\zeta_h'(t) = \dot{\bar\zeta}(t) = \bar X_t(\bar\zeta(t)) = 
		\bar X_t(\pi(\zeta_h'(t))) = T_{\zeta_h'(t)}\pi\cdot X_t(\zeta_h'(t)).
	\end{equation*}
	Since $\zeta_h'$ is horizontal, it is an integral curve of $X_t$.
	Since $\zeta_h'$ and $\zeta_h$ have the same initial conditions, uniqueness of integral curves yield $\zeta_h' = \zeta_h$.
	This concludes the proof.
\end{proof}


The key observation is that \Cref{lem:horiz_shortest_curve} reduces the abstract geometric Monge problem~\eqref{eq:optimal_transport_abstract} to a problem that can be formulated entirely on~$\mathfrak{B}$, namely, to find a minimal geodesic curve between the two given elements $b,b'\in \mathfrak{B}$.
In general, this problem is not easier to solve than the original one, but if the geometry of the Riemannian manifold $(\mathfrak{B},\mathcal{B})$ is such that any two elements in $\mathfrak{B}$ are connected by a minimal geodesic, then problem~\eqref{eq:optimal_transport_abstract} simplifies significantly, as we shall see in \Cref{pro:abstract_optimal_transport} below.

Let us now discuss the concept of \emph{polar factorisations}\index{polar factorization} and how it is related to  optimal transport problems.
For a fixed $b\in \mathfrak{B}$, consider its isotropy sub-group 
\[
	\mathfrak{G}_{b} = \{ g\in \mathfrak{G}\mid b \cdot g = b \}.
\]
As a natural abstraction of the work of Brenier~\cite{Br1991}, consider the \emph{polar cone}
\begin{equation}\label{eq:polar_cone}
	\mathfrak{K} = \{ k\in \mathfrak{G}\mid  \mathrm{dist}(e,k)\leq\mathrm{dist}(h,k),\;\forall h\in G_b \}.
\end{equation}
Expressed in words, the polar cone~\eqref{eq:polar_cone} consists of elements in $G$ whose closest point on the identity fiber is~$e$.

\begin{proposition}\label{pro:abstract_optimal_transport}
	Let $\mathcal{G}$ be a right-invariant metric on $\mathfrak{G}$ that descends to a metric $\mathcal{B}$ on $\mathfrak{B}$ with respect to the fiber bundle $\pi(g) = b\cdot g$ for some fixed $b\in \mathfrak{B}$.
	Then the following statements are equivalent:
	\begin{enumerate}
		\item For any $b'\in \mathfrak{B}$, there exists a unique minimal geodesic between $b$ and $b'$.
		\item For any $b'',b'\in \mathfrak{B}$, there exists a unique minimal geodesic from $b''$ to $b'$.
		\item There exists a unique solution to the optimal transport problem~\eqref{eq:optimal_transport_abstract} with $c=\mathrm{dist}^{2}$, and that solution is connected to~$e$ by a unique minimal geodesic.
		\item Every $g\in G$ has a unique factorisation $g = h k$, with $h\in \mathfrak{G}_b$ and $k\in \mathfrak{K}$, and every $k\in \mathfrak{K}$ is connected to~$e$ by a unique minimal geodesic.
	\end{enumerate}
\end{proposition}

\begin{proof}
	$1\Rightarrow 2$. 
	Since $\mathcal{G}$ is right-invariant, we get that if $\zeta\colon[0,1]\to \mathfrak{G}$ is a minimal geodesic, then so is $\zeta\cdot g$ for any $g\in \mathfrak{G}$.
	Since the action of $\mathfrak{G}$ is transitive, $b'' = b\cdot g$ for some $g\in \mathfrak{G}$.

	$2\Rightarrow 3$. 
	Let $\bar\zeta\colon[0,1]\to \mathfrak{B}$ be the minimal geodesic from $b$ to $b'$.
	By \Cref{lem:horiz_shortest_curve}, there is a unique corresponding horizontal geodesic $\zeta\colon[0,1]\to \mathfrak{G}$ with $\zeta(0) = e$ and $\pi(\zeta(t))=\bar\zeta(t)$.
	There cannot be any curve from $e$ to $\pi^{-1}(\{ b' \})$ shorter than $\zeta$, because then $\bar\zeta$ would not be a minimal geodesic.
	A curve from $e$ to $\pi_b^{-1}(\{ b' \})$ of the same length as $\zeta$ must be horizontal (by \Cref{lem:horiz_shortest_curve}), and therefore equal to~$\zeta$ (also by \Cref{lem:horiz_shortest_curve}).
	Thus, if $q \in \pi^{-1}(\{ b' \})\backslash\{\zeta(1)\}$ then $\mathrm{dist}(e,q)>\mathrm{dist}(e,\zeta(1))$, so $\zeta(1)$ is the unique solution to problem~\eqref{eq:optimal_transport_abstract}.
	Also, $\zeta$ is a unique minimal geodesic between $e$ and $\zeta(1)$.

	$3\Rightarrow 4$.
	Let $k$ be the unique solution to~\eqref{eq:optimal_transport_abstract} with $b' = \pi(g)$.
	Then $k$ and $g$ belong to the same fiber, so $g=h k$ for some unique element $h\in \mathfrak{G}_b$.
	There cannot be another such factorization $g=h' k'$, because then $k$ would not be a unique solution to~\eqref{eq:optimal_transport_abstract}.
	Now take any $k\in \mathfrak{K}$.
	Then $k$ is the unique solution to~\eqref{eq:optimal_transport_abstract} with $b'= \pi(k)$, and that solution is connected to~$e$ by a unique minimal geodesic.
	Thus, any $k\in \mathfrak{K}$ is connected to~$e$ by a unique minimal geodesic.

	$4\Rightarrow 1$.
	For any $b'\in \mathfrak{B}$ we can find $g\in \mathfrak{G}$ such that $b' = b\cdot g = \pi(g)$, which follows since the action is transitive.
	Let $g=hk$ be the unique factorization, and let $\zeta\colon[0,1]\to \mathfrak{G}$ be the unique minimal geodesic from $e$ to $k$.
	Assume now that $\zeta$ is not horizontal.
	Then, by \Cref{lem:horiz_shortest_curve}, we can find a horizontal curve $\zeta_h\colon[0,1]\to \mathfrak{G}$ with $\zeta_h(0)=k$ and $\zeta_h(1)\in \mathfrak{G}_b$ that is strictly shorter than~$\zeta$.
	Since $\zeta$ is a unique minimal geodesic between $e$ and $k$, it cannot hold that $\zeta_h(1) = e$.
	Thus, $k\notin K$ since there is a point $\zeta_h(1)$ on the identity fiber closer to $k$ than $e$, so we reach a contradiction.
	Therefore, $\zeta$ must be horizontal, so it descends to a corresponding geodesic $\bar\zeta$ between $b$ and $b'$. The geodesic
	$\bar\zeta$ must be unique minimal, otherwise $\zeta$ cannot be unique minimal.
	
	This concludes the proof.
\end{proof}

\begin{remark}
	Notice that $\mathfrak{G}$ acts on the right and the Riemannian metric $\mathcal{G}$ is right-invariant.
	This situation is geometrically different from the setup of the Wasserstein-Otto metric in optimal mass transport, where the action is on the left and the Riemannian metric is semi right-invariant (it is right-invariant only with respect to the isotropy subgroup). Thus, on some fundamental level,
	this is what distinguishes optimal mass transport from information geometry. 
\end{remark}

\subsection{Optimal transport of inner products and $QR$~factorisation} 
\label{sub:qr}

In this section we show how the $QR$~factorization of square matrices is related to optimal transport of inner products on~$\R^{n}$ in accordance with the abstract transport framework just presented.
One can think of this example as a finite-dimensional analogue of optimal information transport problem, which we describe in~\Cref{sec:factorizations_information} below.
It also gives some geometric insight into the $QR$ and Cholesky factorizations.
For further details, we refer the reader to Modin~\cite{Modin17}.

Let $\mathfrak{G}=\mathrm{GL}^+(n,\R) = \{A\in \mathrm{GL}(n,\R)\mid \operatorname{det} A > 0 \}$ and let $\mathfrak{B}=\mathrm{SPD}_n$ be the manifold of inner products on $\R^{n}$. The manifold
$\mathrm{SPD}_n$ is identified with the space of symmetric positive definite $n\times n$~matrices;
if $M$ is a symmetric positive definite matrix, then the corresponding inner product is $\pair{\mathbf{x},\mathbf{y}}_{M} = \mathbf{x}^{\top}M\mathbf{y}$. Note that 
$\mathrm{SPD}_n$ is a convex open subset of the vector space $\mathrm{S}_n$ of all symmetric $n\times n$~matrices.

An element $A\in \mathrm{GL}^+(n,\R)$ acts on $M\in\mathrm{SPD}_n$ from the right by $M\cdot A  = A^{\top}M A$.
This action is transitive and 
the lifted action is given by $U\cdot A = A^{\top}U A$, where $U\in T_M\mathrm{SPD}_n=\mathrm{S}_n$.

Let $I$ denote the identity matrix, which is an element in both $\mathrm{GL}^+(n,\R)$ and $\mathrm{SPD}_n$. 
Consider the projection $\pi\colon\mathrm{GL}^+(n,\R)\to\mathrm{SPD}_n$ given by $\pi(A) = I\cdot A = A^{\top}A$.
The corresponding isotropy group is $\mathfrak{G}_I = \mathrm{SO}(n)$, because
\begin{equation*}
\pi(Q A) = A^{\top}Q^{\top}Q A = A^{\top}A = \pi(A)
\end{equation*}
with equality if and only if $Q\in\mathrm{O}(n)\cap \mathrm{GL}^+(n,\R) = \mathrm{SO}(n)$ 
and therefore we have a principal bundle
\begin{equation}\label{eq:principal_bundle_QR}
	\mathrm{SO}(n) \xhookrightarrow{\quad} \mathrm{GL}^+(n,\R) \xrightarrow{\;\pi_I\;} \mathrm{SPD}_n .
\end{equation}

There is a natural metric $h$ on $\mathrm{SPD}_n$ given by
\begin{equation}\label{eq:sym_metric}
	h_{M}(U,V) = \mathrm{tr}(M^{-1}UM^{-1} V), \quad U,V\in T_M\mathrm{SPD}_n .
\end{equation}
This metric is invariant with respect to the action, since
\begin{equation*}
	\begin{split}
		h_{M\cdot A}( U\cdot A, V\cdot A)
		&= \mathrm{tr}((A^{\top}MA)^{-1}A^{\top}UA (A^{\top}MA)^{-1} A^{\top}VA) \\
		&= \mathrm{tr}(A^{-1}M^{-1}A^{-\top}A^{\top}UA A^{-1}M^{-1}A^{-\top} A^{\top}VA) \\
		&= \mathrm{tr}(A^{-1}M^{-1}UM^{-1}VA) \\
		& \text{(using cyclic property: $\mathrm{tr}(ABC) = \mathrm{tr}(BCA)$)} \\
		&= \mathrm{tr}(M^{-1}UM^{-1}VA A^{-1}) \\
		&= \mathrm{tr}(M^{-1}UM^{-1}V)  = h_{M}(U,V).
	\end{split}
\end{equation*}
In fact, this is precisely the Fisher-Rao metric restricted to \emph{Gaussian density functions}\index{Gaussian density} on $\R^n$ (see \cite[Lemma~3.3]{Modin17}).
Indeed, $M \in \mathrm{SPD}_n$ is associated with the Gaussian density
\[
	\rho_M(x) = \sqrt{\frac{\operatorname{det}W}{(2\pi)^n}}\operatorname{exp}(-\frac{1}{2}x^\top M x).
\]

We proceed by constructing a compatible metric on $\mathrm{GL}^+(n,\R)$.
Consider the projection operator $\ell\colon\mathfrak{gl}(n,\R)\to\mathfrak{gl}(n,\R)$ given by 
\begin{equation*}
	\ell(U)_{ij} = \begin{cases}
		U_{ij} &\text{if } i \geq j, \\
		0 &\text{otherwise.}
	\end{cases}	
\end{equation*}
In other words, $\ell(U)$ is zero on the strictly upper triangular entries and is equal to~$U$ elsewhere.
Let $g$ be the right-invariant metric on $\mathrm{GL}^+(n,\R)$ defined by
\begin{equation}\label{eq:QR_metric}
	g_{I}(u,v) = \mathrm{tr}\big(\ell(u)^{\top}\ell(v)\big) + \mathrm{tr}\big((u+u^{\top})(v+v^{\top})\big).
\end{equation}
Using right translations we have $g_{A}(U,V) = g_{I}(UA^{-1},VA^{-1})$.
The orthogonal complement of $T_I\mathrm{SO}(n) = \mathfrak{so}(n)$ with respect to $g_{A}$ consists of the upper triangular matrices.
Indeed, since matrices in $\mathfrak{so}(n)$ are skew symmetric,  the second term in~\eqref{eq:QR_metric} vanishes if either $u$ or $v$ belong to $\mathfrak{so}(n)$.
In mathematical terms
\begin{equation*}
	\mathfrak{so}(n)^{\top} = \mathfrak{upp}(n) \coloneqq \{ u\in \mathfrak{gl}(n)\mid \ell(u) = 0 \}.	
\end{equation*}

\begin{proposition}
	The right-invariant metric $g$ on $\mathrm{GL}^+(n,\R)$ is descending with respect to the principal bundle structure~\eqref{eq:principal_bundle_QR}.
	The corresponding metric on $\mathrm{SPD}_n$ is given by $h$ in the formula \eqref{eq:sym_metric}.
\end{proposition}

\begin{proof}
	By \Cref{prop:descend} we need to show that 
	\begin{equation*}
		g_{I}( \mathrm{ad}_\xi(u),v) + g_{I}(u,\mathrm{ad}_\xi(v)) = 0, \quad \forall\, u,v\in \mathfrak{upp}(n), \xi\in\mathfrak{so}(n).
	\end{equation*}
	We have
	\begin{equation*}
		g_{I}(\mathrm{ad}_\xi(u),v) = \mathrm{tr}\big( ([\xi,u]+[\xi,u]^{\top})(v+v^{\top})\big) = 
		\mathrm{tr}\big(([\xi,u+u^{\top}])(v+v^{\top}) \big).
	\end{equation*}
	By the cyclic property of the trace
	\begin{equation*}
		\begin{split}
			\mathrm{tr}\big(([\xi,u+u^{\top}])(v+v^{\top}) \big) &= -\mathrm{tr}\big((u+u^{\top})([\xi,v+v^{\top}]) \big) \\
			&= -\mathrm{tr}\big((u+u^{\top})([\xi,v] + [\xi,v]^{\top}]) \big) \\
			&= -g_{I}(u,\mathrm{ad}_\xi(v)).
		\end{split}
	\end{equation*}
	Therefore, the metric is descending.
	If $u\in\mathfrak{upp}(n)$, then $T_I\pi\cdot u = u + u^{\top}$, so
	\begin{equation*}
		g_{I}(u,v) = \frac{1}{4}\mathrm{tr}\big((u+u^{\top})(v+v^{\top}) \big) = h_{I}(T_I\pi\cdot u,T_I\pi\cdot v).
	\end{equation*}
	Since $h$ is right-invariant, $g$ descends to $h$ (up to scaling which can be checked manually).
	This proves the result.
\end{proof}

The horizontal distribution is given by $H_A\mathfrak{G} = \mathfrak{upp}(n)A$.
Something special happens in this case, namely, $\mathfrak{upp}(n)$ is a Lie algebra (closed under the matrix commutator).
Consequently, the horizontal distribution is integrable.
Its integral manifold through the identity is the Lie group $\mathrm{Upp}(n)$ of upper triangular $n\times n$~matrices with strictly positive diagonal entries.
Observe that $\mathrm{Upp}(n)$ forms the polar cone~\eqref{eq:polar_cone}.

Let $A\in\mathrm{GL}^+(n,\R)$.
If there exists a unique minimal geodesic $\bar\zeta\colon[0,1]\to\mathrm{SPD}_n$ from $I$ to $\pi(A)$, then, by \Cref{pro:abstract_optimal_transport}, we obtain a factorization $A=QR$, with $Q\in\mathrm{SO}(n)$ and $R\in\mathrm{Upp}(n)$.
Since the metric~\eqref{eq:QR_metric} is smooth, it follows from standard Riemannian considerations that there exists a neighborhood $\mathcal{O}\subset \mathrm{SPD}_n$ of $I$ such that any element in $\mathcal{O}$ is connected to $I$ by a unique minimal geodesic.
In fact, this can be extended to a global result (see \cite[Lemma~3.8]{Modin17}).
Therefore, $A$ has a unique $QR$~factorization. \index{$QR$~factorization}

\begin{proposition}[$QR$, or Iwasawa, factorization]
	Let $A\in \mathrm{GL}^+(n,\R)$. Then there exists a unique $Q\in \mathrm{SO}(n)$ and an unique $R\in \mathrm{Upp}(n)$ such that $A = QR$.
\end{proposition}

In summary, the upper triangular matrix~$R$ corresponding to $A$ solves the problem of optimally (with respect to the cost function given by the distance function $\mathrm{dist}^{2}$ corresponding to the metric $\mathcal{G}$) transporting the Euclidean inner product on $\R^{n}$ to the inner product defined by $M=A^{\top}A$.
In addition, the factor~$R$ is the transpose of the Cholesky factorization of~$M$.
Indeed, if $L=R^{\top}$ then 
\begin{equation*}
	M=\pi(A) = \pi(R) = R^{\top}R = LL^{\top}.	
\end{equation*}

\begin{remark}
	The above setting can be extended to $\mathrm{GL}(n,\C)$ by replacing $\mathrm{SO}(n)$ with $\mathrm{U}(n)$ and every transpose with the Hermitian conjugate.
\end{remark}

\begin{remark}
	One motivation for the constructions in \cite{Otto01} was to study gradient flows with respect to the Wasserstein metric in optimal mass transport. 
	Analogous gradient flows for the optimal transport problem in this section are related to the Toda flow, the ``$QR$~algorithm''~\cite{GoLo1989}, Brockett's flow for continuous diagonalization of matrices~\cite{Br1988}, and the incompressible porous medium equation \cite{KhMo2023}.
\end{remark}

\section{Information theoretic factorization of diffeomorphisms}
\label{sec:factorizations_information}

We next turn to an example of particular interest, namely $\mathfrak{G}=\mathfrak{D}(M)$ and $\mathfrak{B}=\mathfrak{Dens}(M)$.
Recall from \Cref{thm:metric_is_descending} that the information metric on $\mathfrak{D}(M)$ descends to the Fisher metric on $\mathfrak{Dens}(M)$.
Also, we have that
\begin{itemize}
	\item the action of $\eta\in\mathfrak{D}(M)$ on $\nu\in\mathfrak{Dens}(M)$ is $\nu\cdot \eta = \eta^{*}\nu$, so the projection is $\pi(\eta) = \eta^{*}\mu$;
	\item $\mathfrak{D}(M)$ and $\mathfrak{Dens}(M)$ are Fréchet manifolds;
	\item the projection $\pi\colon\mathfrak{D}(M)\to\mathfrak{Dens}(M)$ is smooth (in the Fréchet topology);
	\item any $\nu,\mu\in\mathfrak{Dens}(M)$ are connected by a unique Fisher-Rao geodesic (as follows from the square root map discussed above).
\end{itemize}
Thus, all prerequisites in \Cref{pro:abstract_optimal_transport} are fulfilled.

The information metric on $\mathfrak{D}(M)$ allows us to obtain a non-degenerate optimal transport formulation in accordance with the abstract framework in the previous section. 
In particular, we obtain a factorization result for diffeomorphisms.

Let $\lambda,\nu\in\mathfrak{Dens}(M)$.
Consider the following transport problem: 
\begin{equation}\label{eq:optimal_info_trans}
	\text{Find $\varphi\in\{ \eta\in\mathfrak{D}(M)\mid\eta^{*}\lambda=\nu\}$ minimizing $\mathrm{dist}^{2}(e,\varphi)$.}
\end{equation}
Here, $\mathrm{dist}$ is the Riemannian distance corresponding to the information metric~\eqref{eq:info_metric}.
Since it descends to Fisher's information metric, we will refer to the problem in \eqref{eq:optimal_info_trans} as the \emph{optimal information transport} \index{optimal information transport} which consists in finding the optimal diffeomorphism that pulls one prescribed probability density $\lambda$ to another $\nu$.


As already discussed above, Friedrich~\cite{Friedrich91} showed that the Fisher metric has constant curvature and its geodesics are thus easy to analyze.
Indeed, recall the infinite dimensional sphere of radius $r = \sqrt{\mu(M)}$
\begin{equation*}
	S^{\infty}(M) = \Big\{
		f\in C^\infty(M)\mid \pair{f,f}_{L^{2}}=\mu(M)
	\Big\},
\end{equation*}
which is a Fréchet submanifold of $C^{\infty}(M)$.
The $L^{2}$ inner product restricted to $S^{\infty}(M)$ provides a weak Riemannian metric 
whose geodesic spray is nonetheless smooth. In fact, 
the geodesics are precisely the great circles and hence $S^{\infty}(M)$ is geodesically complete with  diameter given by $\pi \sqrt{\mu(M)}$.

Let $\mathcal{O}(M) = \{f\in S^{\infty}(M)\mid f>0 \}$ denote the set of positive functions of $L^2$-norm~$\sqrt{\mu(M)}$.
$\mathcal{O}(M)$ is an open subset of $S^{\infty}(M)$ and therefore a Fréchet manifold itself.
The following result is a more precise reformulation of Theorem~\ref{thm:isometry} above, see also \cite{KhesinLenellsMisiolekPreston13}.

\begin{theorem}\label{thm:sphere_diffeo}
	The map
	\begin{equation*}
		\Phi\colon\mathfrak{Dens}(M)\ni \nu \longmapsto \sqrt{\frac{\nu}{\mu}}
	\end{equation*}
	is an isometric diffeomorphism $\mathfrak{Dens}(M)\to \mathcal{O}(M)$.
	The diameter of $\mathcal{O}(M)$ (and hence $\mathfrak{Dens}(M)$) is $\frac{\pi}{2}\sqrt{\mu(M)}$.\index{square root map}
\end{theorem}

If $f,g\in\mathcal{O}(M)$, then there is a unique minimal geodesic $\sigma\colon[0,1]\to S^{\infty}(M)$ from $f$ to $g$ which is contained in $\mathcal{O}(M)$, i.e., $\mathcal{O}(M)$ is a convex subset of $S^{\infty}(M)$.
Indeed, the minimal geodesic is explicitly given by
\begin{equation}\label{eq:explicit_geodesic}
	\sigma\colon[0,1]\ni t \longmapsto \frac{\sin\!\big((1-t)\theta\big)}{\sin \theta}\, f + \frac{\sin(t\theta)}{\sin \theta}\, g 
\end{equation}
where 
$\theta = \arccos\Big( \frac{\pair{f,g}_{L^{2}}}{\mu(M)} \Big)$. 
The polar cone of $\mathfrak{D}(M)$ with respect to the information metric~\eqref{eq:info_metric} is
\begin{equation}\label{eq:polar_cone_diff}
	K(M) = \big\{
		\varphi\in\mathfrak{D}(M)\mid  \mathrm{dist}_{\Delta}(e,\varphi) \leq \mathrm{dist}_{\Delta}(\eta,\varphi),\; 
		\forall\, \eta\in\mathfrak{D}_{\mu}(M)
	\big\}.
\end{equation}
There is a unique minimal geodesic between $\mu$ and any $\nu\in\mathfrak{Dens}(M)$, so from \Cref{pro:abstract_optimal_transport} every $\psi\in K(M)$ is the endpoint of a minimal horizontal geodesic $\zeta\colon[0,1]\to\mathfrak{D}(M)$ with $\zeta(0) = e$.
Since $\zeta$ is horizontal, it is of the form $\zeta(t)=\mathrm{Exp}_{e}(t\nabla(w_0))$ for a unique $w_0\in {C}^{\infty}(M)$, where $\mathrm{Exp}\colon T\mathfrak{D}(M)\to\mathfrak{D}(M)$ denotes the Riemannian exponential\index{Riemannian exponential} corresponding to the information metric.

Let $\varphi\in\mathfrak{D}(M)$.
Due to the explicit form~\eqref{eq:explicit_geodesic} of minimal geodesics in $\mathcal{O}(M)$, thus $\mathfrak{Dens}(M)$, we can compute the function $w_0\in{C}^{\infty}(M)$ such that $\mathrm{Exp}_e(\nabla(w_0))$ is the unique element in $K(M)$ belonging to the same fiber as $\varphi$.
Indeed,
\begin{equation*}
	\begin{split}
	T_{e}\pi_\mu\cdot\nabla(w_0) &= \frac{{d}}{{d} t}\Big|_{t=0} \pi_\mu(\mathrm{Exp}_e(t\nabla(w_0))) \\
	&= \frac{{d}}{{d} t}\Big|_{t=0}\pi_\mu(\zeta(t)) = \frac{{d}}{{d} t}\Big|_{t=0}\sigma(t)^{2}\mu 
	= 2\dot\sigma(0)\sigma(0)\mu,		
	\end{split}
\end{equation*}
where $\sigma(t)$ is the curve \eqref{eq:explicit_geodesic} with $f=1$ and $g=\sqrt{\mathrm{Jac}(\varphi)}$ (the Jacobian is defined by $\mathrm{Jac}(\psi)\mu = \psi^{*}\mu$).
Since $\sigma(0) = 1$ and $T_{e}\pi_\mu\cdot\nabla(w_0) = L_{\nabla(w_0)}\mu = \Delta w_0\mu$, we get
\begin{equation}\label{eq:w0}
	\Delta w_0 = \frac{2\theta\sqrt{\mathrm{Jac}(\varphi)}-2\theta\cos(\theta)}{\sin\theta}, \quad \theta = \arccos\Big(\frac{\int_M \sqrt{\mathrm{Jac}(\varphi)} \mu}{\mu(M)}\Big).
\end{equation}

How do we compute the horizontal geodesic $\zeta(t) = \mathrm{Exp}_e(t \nabla(w_0))$?

One way is to solve the geodesic equation  with $u(0) = \nabla(w_0)$, and reconstruct~$\zeta(t)$ by integrating the non-autonomous equation $\dot\zeta(t) = u(t)\circ\zeta(t)$ with $\zeta(0)=e$.
From \cite[Proposition~4.5]{KhesinLenellsMisiolekPreston13} and the isomorphism $T_e\pi_\mu\colon \nabla({C}^{\infty}(M))\to T_\mu\mathfrak{Dens}(M)$, the solution exists for $t\in [0,1]$ (for details on the maximal existence time, see \cite[\S\!~4.2]{KhesinLenellsMisiolekPreston13}).

Another way is to directly lift the geodesic curve~\eqref{eq:explicit_geodesic} by the technique in the proof of \Cref{lem:horiz_shortest_curve}.
Indeed, the geodesic $\bar\zeta(t)$ in $\mathfrak{Dens}(M)$ corresponding to $\zeta(t)$ is given by $\bar\zeta(t) = \Phi^{-1}(\sigma(t)) = \sigma(t)^{2}\mu$.
Since $\pi_\mu(\zeta(t)) = \bar\zeta(t)$ and 
\begin{equation*}
	TR_{\zeta(t)^{-1}}(\zeta(t),\dot\zeta(t)) = (e,\dot\zeta(t)\circ\zeta(t)^{-1}) = (e,\nabla(w_t)), \quad w_t\in{H}^{s+1}(M)	
\end{equation*}
we conclude
\begin{equation*}
	\begin{split}
	\dot{\bar\zeta}(t) &= T_e (\bar R_{\zeta(t)}\circ \pi_\mu) \cdot \nabla(w_t) \\
		&= \zeta(t)^{*}\big(L_{\nabla(w_t)}\mu \big) \\
		&= \operatorname{div}(\nabla(w_t))\circ\zeta(t)\mathrm{Jac}(\zeta(t))\mu.
	\end{split}
\end{equation*}
From $\dot{\bar\zeta}(t) = 2\dot\sigma(t)\sigma(t)\mu$ and $\mathrm{Jac}(\zeta(t)) = \sigma(t)^{2}$, we get
\begin{equation*}
	2\dot\sigma(t)\sigma(t) = \big(\Delta w_t\circ\zeta(t)\big) \sigma(t)^{2} . 
\end{equation*}
The horizontal geodesic $\zeta(t)$ is now constructed by solving the following non-autonomous ordinary differential equation on $\mathfrak{D}(M)$
\begin{equation}\label{eq:lifting}
	\begin{split}
		\dot\zeta(t) &= \nabla(w_t)\circ\zeta(t), \quad \zeta(0)=e \\
		\Delta w_t &= \frac{2\dot\sigma(t)}{\sigma(t)}\circ\zeta(t)^{-1} \\
		\sigma(t) &= \frac{\sin\!\big((1-t)\theta\big)}{\sin \theta} + \frac{\sin(t\theta)}{\sin \theta}\sqrt{\mathrm{Jac}(\varphi)},
		\quad \theta = \arccos\Big(\frac{\int_M \sqrt{\mathrm{Jac}(\varphi)} \mu}{\mu(M)}\Big).
	\end{split}
\end{equation}
Explicitly, the vector field $X_{t}$ on $\mathfrak{D}(M)$ is
\begin{equation*}
	X_t\colon\zeta \mapsto \widetilde{\nabla}\circ \tilde{\Delta}^{-1}\Big( \zeta, \frac{2\dot\sigma(t)}{\sigma(t)} \Big) = \nabla\Big(\Delta^{-1}\Big( \frac{2\dot\sigma(t)}{\sigma(t)}\circ\zeta^{-1} \Big)\Big)\circ \zeta .
\end{equation*}
That is, equation~\eqref{eq:lifting} is the non-autonomous ordinary differential equation
\begin{equation*}
	\dot\zeta(t) = X_t(\zeta(t)),\quad \zeta(0) = e .
\end{equation*}
Smoothness of $X_{t}$ is obtained by standard techniques for conjugation of diffeomorphisms (see, e.g., \cite[Lemma~3.4]{Modin15}).

In summary, we have proved the following ``information factorization'' \index{information factorization} result for diffeomorphisms.


\begin{theorem}\label{thm:factorisation_of_diff}
	Every $\varphi\in\mathfrak{D}(M)$ admits a unique factorisation $\varphi = \eta\circ\psi$, with $\eta\in\mathfrak{D}_{\mu}(M)$ and $\psi \in K(M)$.
	We have $\psi = \mathrm{Exp}_e(\nabla(w_0))$ with $w_0$ given by equation~\eqref{eq:w0}.
	There is a unique minimal geodesic $\zeta(t)$ with $\zeta(0)=e$ and $\zeta(1)=\psi$;
	it can be computed by solving equation~\eqref{eq:lifting}.
	The geodesic $\zeta(t)$ is horizontal.
\end{theorem}

The result in Theorem~\ref{thm:factorisation_of_diff} is a Fisher-Rao analogue of Brenier's~\cite{Br1991} result on factorizations of maps that give unique solutions to the Wasserstein $L^2$ optimal mass transport problem.
Indeed, the element $\psi$ in Theorem~\ref{thm:factorisation_of_diff} is the information optimal diffeomorphism which maps $\mu\in\mathfrak{Dens}(M)$ to $\nu = \varphi^*\mu$.




\section{Entropy, Fisher information, and gradient flows}\label{Entropy}



Perhaps the most central concept in information theory is \emph{entropy}\index{entropy}, introduced by Shannon~\cite{Shannon} as a measure of information.
The basic idea is as follows.
Consider a set with $k$ entries and suppose that you encode a piece of information, for example a message, by a string of $n$ elements from the set.
The total number of possibilities is, of course, $k^n$.
Thus, one could say that the information content is expressed by $k^n$ however Shannon defined it for such a string as $\log(k^n)$.
This makes sense, because it means that a string of $2n$ elements contains twice as much information.
But this definition of information assumes that each possible message is equally likely and therefore depends on how the messages are encoded.
Indeed, consider a pool of $m$ possible messages.
These can be encoded by a string of $n$ entries as long as $k^n \geq m$, yet the information as defined above grows linearly with $n$.
Shannon's real insight was to assign information instead to the probability distribution on a set $\mathcal X$ of possible messages.
If the values are distributed according to the distribution function $p\colon \mathcal X\to [0,1]$ then the information content of a sample $x \in \mathcal{X}$ is defined as $-\log p(x)$. 
Thus, the less likely an element is, the greater its information content.
This definition of information is independent of how the messages are encoded.
Furthermore, it is consistent with the previous one, as can be seen by taking $\mathcal X = \{1,\ldots,k \}^n$ and assigning equal probabilities $p = 1/k^n$ to each element.
For an arbitrary distribution $p(x)$, Shannon considered the expected information carried by a (discrete) probability distribution $p$ and he called it \emph{entropy}\footnote{Contrary to science folklore, it was not John von Neumann who suggested the term. Shannon himself saw the connection to thermodynamics and therefore named it ``entropy''.} 
$$
	S(p) \coloneqq \mathbb{E}[-\log p] =  - \sum_{x\in\mathcal X} p(x) \log p(x).
$$
This entropy functional turned out to be extremely useful in building a theory of information (or ``theory of communication'', as Shannon called it).



In the context of diffeomorphisms, we consider an infinitesimal analog of Shannon's definition, where the set $\mathcal X$ is replaced by a manifold $M$ and the discrete probability distribution $p$ by a density $\nu \in \mathfrak{Dens}(M)$.
From this viewpoint, it is clear that Shannon's entropy is not invariantly defined: it depends on a minimal entropy configuration $\mu$ (which in the discrete case was assumed to be the uniform distribution).
Informally, the analog of the message string is a configuration of a point cloud in $M$ (or, in hydrodynamical terms, a fluid particle configuration). 
The set of transformations that re-encodes messages (i.e., point clouds) without changing their information content is then $\mathfrak{D}_\mu(M)$, whereas a general diffeomorphism changes information.

But what is the geometric link between entropy and the Fisher-Rao metric?
To answer that question, let us first give the geometrically natural, infinitesimal analog of Shannon's entropy and another functional introduced by Fisher~\cite{Fi1922}.

\begin{definition}\label{entropy-Fisher}
The {\it entropy of a probability distribution} $\nu\in\mathfrak{Dens}(M)$ relative to a reference distribution $\mu\in\mathfrak{Dens}(M)$ is given by
\begin{equation*}
    S_\mu(\nu) = -\int_M \log\left( \frac{\nu}{\mu}\right)\nu .
\end{equation*}
The {\it Fisher information functional}\index{Fisher information} $I(\rho)$ of the density $\rho = \nu/\mu$ is defined as 
$$
 I(\rho):= \int_M \frac{\lvert \nabla\rho \rvert^2}{\rho} \, \mu\,.
$$
\end{definition}

A relation between these two quantities is shown in the following proposition (which can be also used as a definition). 

\begin{proposition}\label{prop-Fisher}
The rate of change of the entropy along trajectories of the heat equation
is equal to the Fisher information functional.
 \end{proposition}

\begin{proof}
Consider a curve $\nu(t)$ in $\mathfrak{Dens}(M)$.
The infinitesimal change of entropy along this curve is
\begin{equation*}
    \frac{d}{dt}S_\mu(\nu) = -\int_M \frac{\dot\nu}{\mu} \frac{\mu}{\nu} \nu
    - \int_M \log\left( \frac{\nu}{\mu} \right)\dot\nu
    = -\int_M \left(\log\left( \frac{\nu}{\mu} \right)+1\right)\dot\nu.
\end{equation*}
Since $T_\nu\mathfrak{Dens}(M) = \{\dot\nu \mid \int_M \dot\nu = 0 \}$ we get
\begin{equation*}
    \frac{d}{dt}S_\mu(\nu) = -\int_M \log\left( \frac{\nu}{\mu} \right)\dot\nu = \int_M \log\left( \frac{\mu}{\nu} \right)\dot\nu .
\end{equation*}
Let $\rho = \nu/\mu$ and assume that $\mu$ is the Riemannian volume form.
Along the heat flow $\dot\rho = \Delta \rho$ we then get
\begin{align*}
    \frac{d}{dt}S_\mu(\nu) &= -\int_M \log\left( \frac{\nu}{\mu} \right)\Delta\rho \mu = \int_M \nabla\log\left( \frac{\nu}{\mu} \right)\cdot \nabla \rho \; \mu
    \\
    &= \int_M \frac{\lvert \nabla\rho \rvert^2}{\rho} \, \mu \eqqcolon  \,I(\rho)  \,.
\end{align*}
\end{proof}

\medskip

There are various ways of linking the Fisher information functional and the Fisher-Rao metric.
The following, adapted from \cite[Proposition~6.1]{KhesinLenellsMisiolekPreston13}, exhibits again a link with the heat flow, but now with respect to the Fisher-Rao metric instead of the Wasserstein-Otto metric as in Proposition~\ref{prop-Fisher}.

\begin{corollary}
The gradient flow\index{gradient flow} of $I(\rho)$ with respect to the Fisher-Rao metric recovers the standard heat flow.
\end{corollary}

\begin{proof}
First note that  $I(\rho)  = 4 \int_M  \lvert\nabla \sqrt{\rho} \rvert^2 \, \mu$. 
The gradient of the Fisher information functional with respect to the Fisher-Rao metric is given by

\begin{align*}
   \frac 14 \frac{d}{dt} I(\rho) &= \int_M 2 \frac{\dot\rho}{\sqrt{\rho}} \Delta\sqrt{\rho} \, \mu = FR(\dot\rho, 2\sqrt{\rho}\Delta\sqrt{\rho} ).
\end{align*}
Thus, the gradient flow of $I(\rho)$ with respect to the Fisher-Rao metric is
\begin{equation*}
    \dot\rho = 2\sqrt{\rho}\Delta\sqrt{\rho} - c\rho.
\end{equation*}
Under the square root map $r = \sqrt{\rho}$ this gradient flow again recovers the standard heat flow
\begin{equation*}
    \dot r = \Delta r - \frac{c}{2} r.
\end{equation*}
where $c\in\mathbb{R}$ is a Lagrange multiplier determined so that $\int_M \dot\rho\, \mu = 0$ (see \cite[Proposition~6.1]{KhesinLenellsMisiolekPreston13} for details).
\end{proof}

\begin{corollary}
The gradient flow of $S_\mu(\nu)$ with respect to the Fisher-Rao metric is given by
\[
    \dot\nu = -\log\left(\frac{\nu}{\mu}\right)\nu + c\nu,
\]
where $c$ is a Lagrange multiplier, determined so that $\int_M \dot\nu = 0$.
Under the change of variables $\nu/\mu = \exp(f)$ the flow becomes
\[
    \dot f = 
    -f + c
\]
where $c$ now corresponds to the constraint $\int_M \exp(f) \mu = 1$.
\end{corollary}

\section{Sasaki-Fisher-Rao metric and the Madelung transform}\label{sec:madelung}

\subsection{The Madelung transform as a symplectomorphism}
Being published in 1926, the Schr\"odinger equation almost immediately 
received a hydrodynamical reformulation. Namely, in 1927
Madelung \cite{Madelung27} proposed to view it as as the Euler equation for a compressible-type fluid. 
It turned out that this reformulation has a special meaning in the context of information geometry.

\begin{definition}\label{def:madelung} 
Let $\rho$ and $\theta$ be real-valued functions on $M$ with $\rho >0$. 
The \emph{Madelung transform}\index{Madelung transform}\index{square root map} is the mapping $\Phi\colon(\rho,\theta) \mapsto \psi$ defined by 
\begin{equation}\label{eq:madelung_def} 
	\Phi(\rho,\theta) := \sqrt{\rho \exp({\ii\theta})} . 
\end{equation} 
\end{definition} 
Note that $\Phi$ is a complex extension of the square root map described 
in \Cref{thm:isometry}.

It turns out that the Madelung transform induces a symplectomorphism from 
the cotangent bundle of probability densities to the projective space of non-vanishing complex functions.
Namely, let $PC^\infty(M,\C)$ stand for the complex projective space of smooth complex-valued functions on $M$.
Its elements can be regarded as cosets $[\psi]$ fibering the $L^2$-sphere of smooth functions, 
where $\psi'\in[\psi]$ if and only if $\psi' = \exp({i \alpha})\psi$ for some $\alpha\in\R$.
The space $PC^\infty(M,\C\backslash \{0\})$ of non-vanishing complex functions is a submanifold of $PC^\infty(M,\C)$.
\begin{theorem}\cite{KhesinMisiolekModinARMA}\label{thm:madelung_symplectomorphism}
	For a simply-connected manifold $M$ the Madelung transform \eqref{eq:madelung_def} induces a map 
	\begin{equation}\label{eq:madelung_symplectic}
		\Phi\colon T^*\mathfrak{Dens}(M)\to PC^\infty(M,\C\backslash \{0\}) 
	\end{equation} 
	which is a symplectomorphism (in the Fr\'echet topology of smooth functions) with respect to 
	the canonical symplectic structure of $T^*\mathfrak{Dens}(M)$ and the complex projective structure of $PC^\infty(M,\C)$.\index{complex projective space}
\end{theorem}
The Madelung transform has already been shown to be a symplectic submersion 
from $ T^*\mathfrak{Dens}(M)$ to the unit sphere of non-vanishing wave functions by \cite{Renesse12}. 
By considering projectivization $PC^\infty(M,\C\backslash \{0\})$ one obtains a stronger, symplectomorphism property stated in the above theorem.

\begin{figure}	
	\includegraphics{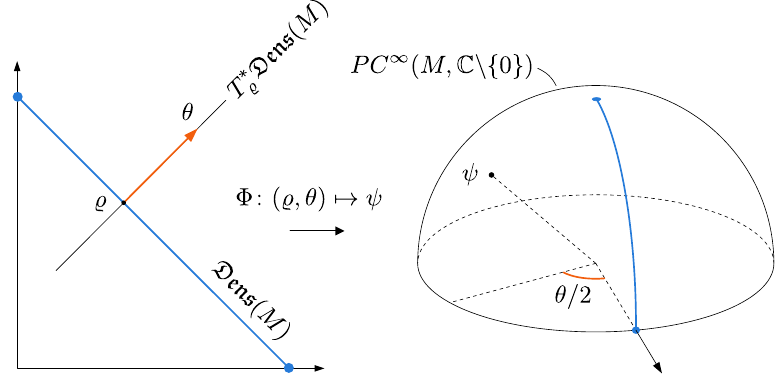}
	\caption{ 
		Illustration of the Madelung transform $\Phi$ given by \eqref{eq:madelung_def}.
		It is a Kähler map from $T^*\mathfrak{Dens}(M)$ to $PC^\infty(M, \mathbb C\setminus 0)\subset PC^\infty(M, \mathbb C)$.
	}\label{fig:Madelung}
\end{figure}

\medskip

\subsection{Relation of the Fubini-Study and Sasaki-Fisher-Rao metrics}\label{sub:kahler_properties_of_madelung}
Now equip the spaces $T^{*}\mathfrak{Dens}(M)$ and $P C^{\infty}(M,\C)$  with suitable Riemannian structures. 
For  the tangent bundle $TT^\ast\mathfrak{Dens}(M)$ of $T^\ast\mathfrak{Dens}(M)$ its  elements are 4-tuples 
$(\varrho,\theta,\dot\varrho,\dot\theta)$ 
with $\varrho\in\mathfrak{Dens}(M)$, $[\theta] \in C^{\infty}(M)/\R$, $\dot\varrho \in \Omega^{n}_0(M)$ 
and $\dot\theta \in C^{\infty}(M)$ subject to the constraint 
\begin{equation}
\int_M \dot\theta \varrho = 0. 
\end{equation} 
\begin{definition} 
The 
\emph{Sasaki} (or 
\emph{Sasaki-Fisher-Rao{\rm )} metric} on $T^\ast\mathfrak{Dens}(M)$
is the lift of the Fisher-Rao metric \eqref{eq:FRmetric} from $\mathfrak{Dens}(M)$: 
\begin{equation}\label{eq:sasaki_FR_metric}
	\mathcal{SFR}^*_{(\varrho,[\theta])}\left((\dot\varrho,\dot\theta),(\dot\varrho,\dot\theta)\right) = 
	\frac{1}{4}\int_{M} \left( \left(\frac{\dot\varrho}{\varrho}\right)^2 + \dot\theta^2 \right) \varrho .
\end{equation}
The canonical metric on $PC^\infty(M,\C)$  
\begin{equation} \label{eq:fubini_study}
\mathcal{FS}^*_\psi(\dot\psi,\dot\psi) 
=  
\frac{\pair{\dot\psi,\dot\psi}}{\pair{\psi,\psi}} - \frac{\pair{\psi,\dot\psi}\pair{\dot\psi,\psi}}{\pair{\psi,\psi}^{2}} 
\end{equation} 
is the (infinite dimensional) \textit{Fubini-Study metric}. \index{Fubini-Study metric}\index{Sasaki-Fisher-Rao metric}
\end{definition} 
%
%
\begin{theorem}\cite{KhesinMisiolekModinARMA}\label{thm:madelung_isometry} 
For a simply-connected underlying manifold $M$ the Madelung transform~\eqref{eq:madelung_symplectic} is an isometry between $T^\ast\mathfrak{Dens}(M)$ equipped with 
the Sasaki-Fisher-Rao metric~\eqref{eq:sasaki_FR_metric} and $P C^{\infty}(M,\C\backslash \{0\})$ equipped with 
the Fubini-Study metric~\eqref{eq:fubini_study}. 
\end{theorem}
Since the Fubini-Study metric together with the complex structure of $PC^\infty(M,\C)$ defines a Kähler structure, 
it follows that $T^*\mathfrak{Dens}(M)$ also admits a natural Kähler structure compatible with its canonical symplectic structure.

\begin{remark}
If the manifold $M$ is not simply-connected, the space of nonvanishing wave functions consists of infinitely many components. Indeed, a wave function changing along a closed non-contractible path in $M$ by  
$\exp(2\pi k\ii )$ for $k\in \mathbb Z$ are univalued on $M$. Different $k$ correspond to wave functions lying in  different connected components of $C^\infty(M, \C\setminus \{0\})$.
On the other hand, the space $T^*\mathfrak{Dens}(M)$ is connected and contractible. The above theorems establish that every connected component of the space $P C^{\infty}(M,\C\backslash \{0\})$ is symplectomorphic and isomeric to $T^*\mathfrak{Dens}(M)$, just like in the case of a simply-connected $M$.

Futhermore, for wave functions with zeros in $M$ one can observe a similar phenomenon: wave functions
can gain an extra factor $\exp(2\pi \ii )$ on a small path around zeros of $\psi$ (a nontrivial ``monodromy"  around zero). Moreover, this is a typical situation whenever zero is a non-critical value of $\psi$: the zero level in the latter case has codimension 2 in $M$, and there is a non-contactible path around it. Thus not all pairs $(\rho, \theta)$ 
can be obtained as the Madelung images of wave functions, but only those satisfying certain ``quantization conditions", as was pointed out by Wallstrom \cite{Wallstrom}. 
It is a very interesting phenomenon, which stirred numerous attempts to explain it and broad discussions in the area, see e.g. \cite{Derakhshani, Fritsche}.

It would be very interesting to approach this ``quantization conditions" for the Madelung transform from the point of view of its symplectic, isometric, and momentum map properties, cf. \cite{Fusca, KhesinMisiolekModinARMA}. This would allow one to extend the K\"ahler framework described above to the case of wave functions with zeros.
\end{remark}

\subsection{Geodesics of the Sasaki-Fisher-Rao metric}
\begin{definition}
The {\it 2-component Hunter-Saxton (2HS) equation}\index{Hunter-Saxton equation} is a system of two equations 
\begin{equation}\label{eq:two_HS_eq}
\left\{
\begin{array}{l}
		\dot u_{xx} = -2 u_x u_{xx} - u u_{xxx} + \sigma\sigma_x, \\
		\dot\sigma = - (\sigma u)_x 
\end{array} \right. 
\end{equation} 
where $u(t,x)$ and $\sigma(t,x)$ are time-dependent periodic functions on the line. 
It can be viewed as a high-frequency limit of the two-component Camassa-Holm equation, cf.\ \cite{Wunsch}. 
\end{definition} 
It turns out that this system is closely related to the K\"ahler geometry of the Madelung transformation 
and the Sasaki-Fisher-Rao metric~\eqref{eq:sasaki_FR_metric}. 
Namely, consider the semi-direct product $\mathcal{G} = \mathfrak{D}_0({\mathbb T})\ltimes C^{\infty}({\mathbb T},{\mathbb T})$, 
where $\mathfrak{D}_0({\mathbb T})$ is the group of circle diffeomorphisms fixing a prescribed point and 
$C^{\infty}({\mathbb T},{\mathbb T})$ stands for ${\mathbb T}$-valued maps of a circle.
Define a right-invariant Riemannian metric on $\mathcal G$ given at the identity by 
\begin{equation*}
\mathsf{g}_{(e,0)}\big( (u,\sigma), (v,\tau) \big) = \frac{1}{4} \int_{{\mathbb T}} \left( u_x v_x +  \sigma\tau\right) d x. 
\end{equation*} 
If $t \to (\varphi(t), \alpha(t))$ is a geodesic in $\mathcal{G}$ then 
$u = \dot\varphi\circ\varphi^{-1}$ and $\sigma = \dot\alpha\circ\varphi^{-1}$ 
satisfy equations \eqref{eq:two_HS_eq}, cf.~\cite{Kohlmann11}.
 Lenells \cite{Lenells13} showed that the map 
\begin{equation}\label{eq:lenells_map} 
(\varphi,\alpha) \mapsto \sqrt{\varphi_x \, \exp{(\ii\alpha)}} 
\end{equation} 
is an isometry from $\mathcal G$ to an open subset of 
$S^{\infty} = \{ \psi \in C^{\infty}({\mathbb T},\C) \mid \|{\psi}\|_{L^2} = 1 \}. $
Moreover, solutions to \eqref{eq:two_HS_eq} satisfying $\int_{{\mathbb T}} \sigma d x = 0$ correspond to 
geodesics on the complex projective space $PC^{\infty}({\mathbb T},\C)$ equipped with the Fubini-Study metric.
It turns out that this isometry is a particular case of \Cref{thm:madelung_isometry}. 
%
\begin{proposition}\cite{KhesinMisiolekModinBAMS, Lenells13}\label{prop:2HS_as_Sasaki}
The 2-component Hunter--Saxton equation~\eqref{eq:two_HS_eq} with initial data satisfying 
$\int_{{\mathbb T}} \sigma \,d x = 0$ is equivalent to the geodesic equation of the Sasaki-Fisher-Rao metric 
\eqref{eq:sasaki_FR_metric} on $T^\ast\mathfrak{Dens}({\mathbb T})$. 
\end{proposition}
Indeed,  the mapping \eqref{eq:lenells_map} can be expressed as 
$(\varphi,\alpha) \mapsto \Phi(\pi(\varphi),\alpha),$
where $\Phi$ is the Madelung transform and $\pi$ is the projection $\varphi\mapsto \varphi^*\mu$ specialized to the case $M={\mathbb T}$. 

\begin{remark}
Observe that if $\sigma=0$ at $t=0$ then $\sigma(t)=0$ for all $t$ and 
the 2-component Hunter-Saxton equation \eqref{eq:two_HS_eq} reduces to the standard Hunter-Saxton equation \eqref{eq:HS}. 
Geometrically, this is a consequence of the fact that horizontal geodesics on $T^*\mathfrak{Dens}(M)$ with respect to 
the Sasaki-Fisher-Rao metric descend to geodesics on $\mathfrak{Dens}(M)$ with respect to the Fisher-Rao metric.
\end{remark} 

\subsection{Quantum mechanics as a compressible fluid}
So far we almost exclusively discussed the geodesic equations, whose energy has only kinetic component. 
Now extend this study to Newton-type equations by allowing an additional potential energy. Namely, let $\psi$ be a wavefunction and consider the family of Schr\"odinger\index{Schr\"odinger equation} (or Gross-Pitaevsky) equations 
(with Planck's constant $\hbar=1$ and mass $m=1/2$) of the form 
\begin{equation}\label{eq:schrodinger} 
 	\mathrm{i}\dot\psi = - \Delta\psi +  V\psi + f(|{\psi}|^2)\psi, 
\end{equation} 
where $V\colon M\to \R$ and $f\colon \R\to \R$. 
If $f\equiv 0$ we obtain the linear Schr\"odinger equation with potential~$V$. 
If $V\equiv 0$ then we obtain the family of non-linear Schr\"odinger equations (NLS); common examples are  $f(a) = \kappa a$ and $f(a) = \frac 12(a-1)^2$.

The Schr\"odinger equation \eqref{eq:schrodinger} is a Hamiltonian equation with respect to 
the symplectic structure on $L^2(M,\C)$, induced by the complex structure $J:\psi \mapsto \mathrm{i}\psi$. 
The Hamiltonian associated with \eqref{eq:schrodinger} is 
\begin{equation*} 
H(\psi) 
= 
\frac{1}{2}\|{\nabla\psi}\|_{L^2(M,\C)}^{2} + \frac{1}{2}\int_M \left( V |{\psi}|^2 + F(|{\psi}|^{2}) \right)\mu,
\end{equation*} 
where $F\colon \R\to \R$ is a primitive of $f$. 


Observe that the $L^2$ norm of a wave function satisfying the Schr\"odinger equation~\eqref{eq:schrodinger} 
is conserved in time, and we will restrict it to the unit sphere. 
Furthermore, the equation is also equivariant with respect to a constant change of phase 
$\psi(x)\mapsto e^{i\alpha}\psi(x)$ and so it descends to the projective space $PC^\infty(M,\C)$.
Geometrically, the Schr\"odinger equation is thus an equation on the complex projective space, cf. \cite{Kibble}.
%
%
%
%
\begin{proposition}[cf.\ \cite{Madelung27,Renesse12, KhesinMisiolekModinARMA}] \label{prop-Madelung}
The Madelung transform $\Phi$ \index{Madelung transform}
maps the family of Schr\"odinger equations \eqref{eq:schrodinger} to the system  
\begin{equation}\label{eq:barotropic2} 
\left\{ 
  \begin{aligned} 
	&\dot v + \nabla_v v + \nabla\Big(V + f(\rho) - \frac{2\Delta\sqrt{\rho}}{\sqrt{\rho}} \Big) = 0 \\ 
	&\dot\rho +\operatorname{div}(\rho v) = 0 \,.
\end{aligned} \right. 
\end{equation} 
for  a gradient field $v=\nabla\theta$ and a density $\rho$, which is a family of Newton's equations
on $\mathfrak{Dens}(M)$ equipped with the Wasserstein-Otto metric \eqref{eq:Otto}\index{Wasserstein-Otto metric} 
and with potential functions 
\begin{equation}
\bar U(\varrho) = \frac 12 I(\varrho) + \int_M V\varrho + \int_M F(\rho)\mu, 
\end{equation}
where $I$ is the Fisher information functional \index{Fisher information}
\begin{equation}\label{eq:Fisher_info_func}
I(\varrho) = \int_M \frac{\|{\nabla \rho}\|^{2}}{\rho}\mu, 
\qquad \text{with} \quad \rho = \frac{\varrho}{\mu},
\end{equation}	
cf. Definition \ref{entropy-Fisher}.
\end{proposition}

%
%
%

Thus the Hamiltonian system \eqref{eq:barotropic2} on $T^*\mathfrak{Dens}(M)$ for potential solutions   is mapped to the Schr\"odinger equation (\ref{eq:schrodinger}) by a symplectomorphism. Conversely, classical hydrodynamic PDEs can be expressed as NLS-type equations. 
In particular, potential solutions of the compressible Euler equations \eqref{eq:barotropic2} of a barotropic fluid 
can be formulated as an NLS equation with Hamiltonian
\begin{equation}\label{eq:compressible_Euler_NLS_Hamiltonian}
H(\psi) 
= 
\frac{1}{2}\|{\nabla\psi}\|_{L^2}^{2} 
- 
\frac{1}{2}\|{\nabla |{\psi}|}\|_{L^2}^{2} 
+ 
\int_M e(|{\psi}|^{2})|{\psi}|^{2}\mu. 
\end{equation} 
The choice $e=0$ gives a Schrödinger formulation for potential solutions of Burgers' equation, whose solutions describe geodesics of the Wasserstein-Otto metric \eqref{eq:Otto} on $\mathfrak{Dens}(M)$. 
Thus the geometric framework connects optimal transport for cost functions with potentials,
the compressible Euler equations and the NLS-type equations described above. In particular, this opens the way to transfer rigorous constructions of  Hamiltonian structures between the NLS and fluid equations, cf. \cite{Staffilani}. 

Note also that we also encounter the Fisher information functional in the context of entropy in Proposition \ref{prop-Fisher}.

\medskip


\begin{figure}
	\centering
	\includegraphics{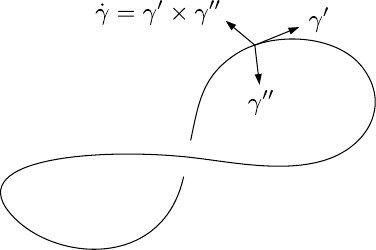}
	\caption{Illustration of the vortex filament flow. The curve $\mathbb T \ni x \mapsto \gamma(t,x)\in \mathbb{R}^3$ evolves in time such that the points on the curve move orthogonally to the osculating plane. The Madelung transform maps this system to a non-linear Schrödinger equation on $\mathbb R$. 
	}\label{fig:binormal}
\end{figure}

\begin{example}
Possibly the most well-known example of the correspondence above is  
the celebrated vortex filament (or binormal) equation\index{vortex filament equation}
$$
\dot \gamma=\gamma'\times \gamma''\,.
$$
It is an evolution equation for a (closed) curve $\gamma\subset \R^3$, 
where  $\gamma=\gamma(t,x)$ and $\gamma':=\partial \gamma/\partial x$ and where $x$ is the arc-length parameter 
(see \Cref{fig:binormal}).
It describes a localized induction approximation of the 3D Euler equation, where vorticity of the initial velocity field is supported on a curve $\gamma$. 

This  equation is known to be Hamiltonian with respect to the Marsden-Weinstein symplectic structure 
on the space of curves in $\R^3$ with Hamiltonian given by the length functional, see e.g.\ \cite{MarsdenWeinstein83, ArnoldKhesin98}. 
On the other hand, it becomes the equation of the 1D barotropic-type fluid \eqref{eq:compressible_Euler_NLS_Hamiltonian} 
with $\rho=k^2$ and $v=\tau$, where $k$ and $\tau$ denote curvature and torsion of the curve $\gamma$, respectively. 
\smallskip

In 1972 Hasimoto~\cite{Ha1972} discovered a surprising transformation assigning
a wave function $\psi:\R\to\C$ to a curve $\gamma$ with curvature $k$ and torsion $\tau$, according to the formula 
\begin{equation}\label{def:hasimoto} 
(k(x),\tau(x))\mapsto \psi(x)=k(x)e^{i\int^x\tau(\tilde x)d\tilde x}. 
\end{equation} 
This map takes the vortex filament  equation to the 1D NLS equation
$i\dot \psi+\psi''+\frac 12 |\psi|^2\psi=0\,.$ 
In particular, the filament equation becomes a completely integrable system 
(since the 1D NLS is one) 
whose first integrals are obtained by pulling back those of the NLS equation. 
\end{example}

\begin{proposition} The Hasimoto transformation is the 
 Madelung transform in the 1D case.\index{Madelung transform}
 \end{proposition}
 This can be seen by comparing Definition \ref{def:madelung} 
 and Equation \eqref{def:hasimoto}.
 Alternatively, one can note that for $\psi(x)=\sqrt{\rho(x)\exp(\mathrm{i}\theta(x)/2)}$
 the pair $(\rho, v)$ with $v=\nabla\theta$ satisfies the compressible Euler equation, while
 in the 1D case these variables are expressed via the curvature $\sqrt{\rho}=\sqrt{k^2}=k$
 and the (indefinite) integral of torsion 
 $\theta(x)/2=\int^x v(\tilde x)d\tilde x=\int^x \tau(\tilde x)d\tilde x$.

\begin{remark} 
The filament equation  has a higher dimensional analog for membranes 
(i.e., compact oriented surfaces $\Sigma$ of codimension 2 in $\R^n$) 
as a skew-mean-curvature flow
$$
\partial_t q={ J}({\bf MC}(q))
$$
where $q\in \Sigma$ is any point of the membrane, ${\bf MC}(q)$ is the mean curvature vector to $\Sigma$ at the point $q$ 
and $ J$ is the operator of rotation by $\pi/2$ in the positive direction (i.e. the operator of almost complex structure) in every normal space to $\Sigma$ \cite{Jerrard, Khesin12}. 
This equation is again Hamiltonian with respect to the Marsden-Weinstein structure on membranes of codimension 2 \cite{MarsdenWeinstein83}
and with a Hamiltonian function given by the $(n-2)$-dimensional volume of the membrane, see e.g.\ \cite{Shashi, Khesin12}.
It turns out that there is no natural  analog of the Hasimoto map, which sends a 
skew-mean-curvature flow to an NLS-type equation for higher $n>3$, see  \cite{KhesinYang}.
\end{remark}
%

\begin{remark}
This way the Madelung transform 
allows one to naturally relate many Schr\"odinger-type equations to
equations of fluid dynamics. Furthermore, such this relation  between equations of quantum
mechanics and hydrodynamics might shed some light on the hydrodynamical
quantum analogues studied in \cite{Bush, Couder}: the motion of bouncing droplets in
certain vibrating liquids (like glycerin at 70Hz) manifests many properties of quantum mechanical particles, such as a similarity with the double slit experiment, tunneling effects, and coral-type potentials.   
\end{remark}




%
%
%

\chapter{Amari-Chentsov connections on the space of densities} 
\label{sec:AC} 
One of the questions posed in the monograph \cite{AmariNagaoka00} 
asked for an infinite dimensional theory of Amari-Chentsov connections.\index{connection} 
So far several proposals have been developed in various levels of generality. 
In this section we present an approach that fits well within the framework of diffeomorphism groups 
and their quotients. 

\section{Amari-Chentsov connections and their geodesics} 
\label{sec:AChard} 

Let $M$ be a compact Riemannian manifold without boundary. 
On the product $\mathfrak{D}(M)\times\mathfrak{D}(M)$ consider the following family of 
real-valued functions 
\begin{align} \label{eq:alfa} 
D^{(\alpha)}(\xi, \eta) 
&= 
\frac{1}{1-\alpha^2} 
\bigg( 
1 - \int_M (\mathrm{Jac}_\mu\xi)^{\frac{1-\alpha}{2}} (\mathrm{Jac}_\mu\eta)^{\frac{1+\alpha}{2}} d\mu 
\bigg) \,,
\\ \label{eq:alfalfa} 
D^{(-1)}(\xi, \eta) = D^{(1)}(\eta, \xi) 
&= 
\frac{1}{4} \int_M \big( \log{\mathrm{Jac}_\mu\xi} - \log{\mathrm{Jac}_\mu\eta} \big) \mathrm{Jac}_\mu \xi \, d\mu \,,
\end{align} 
where $-1 < \alpha < 1$. 
These functions are clearly well defined on $\mathfrak{D}_\mu(M) \backslash \mathfrak{D}(M)$ 
and satisfy $D^{(\alpha)}(\xi, \eta) \geq 0$ with equality if and only if $\xi$ and $\eta$ project onto 
the same density on $M$. 
They can be naturally viewed as diffeomorphism group analogues of the contrast functions ($\alpha$-\emph{divergences})\index{divergence} 
considered by Amari and Chentsov in the classical setting of finite dimensional statistical models. 

Although for the sake of clarity we will focus on the one-dimensional case, 
all the constructions can be readily generalized to diffeomorphism groups of 
higher-dimensional manifolds. 

If $M$ is the unit circle $\mathbb{T} = \mathbb{R}/\mathbb{Z}$ then 
$\mathfrak{D}_\mu(\mathbb{T})$ is simply the set of rigid rotations $\mathrm{Rot}(\mathbb{T}) \simeq \mathbb{T}$. 
In this case it will be convenient to identify the quotient space of densities 
$\mathfrak{Dens}(\mathbb{T}) = \mathrm{Rot}(\mathbb{T})\backslash\mathfrak{D}(\mathbb{T})$ 
with the subgroup of all those circle diffeomorphisms which fix a prescribed point, e.g. 
$\mathfrak{Dens}(\mathbb{T}) \simeq \big\{ \xi \in \mathfrak{D}(\mathbb{T})\mid \xi(0) = 0 \big\}$. 
Its tangent space at the identity map can then be identified with the space of smooth periodic functions 
that vanish at $x=0$. 
Furthermore, for any such function $u(x)$ the inverse operator of 
$A = -\partial_x^2$ 
can be written explicitly in the form 
\begin{align} \label{eq:A-inv} 
A^{-1} u(x) 
= 
- \int_0^x \int_0^y u(z) \, dz dy + x\int_0^1\int_0^y u(z) \, dzdy. 
\end{align} 

We are now in a position to prove the following result.
\begin{theorem}{\rm (\cite{LenellsMisiolek13})} \label{thm:AC-PJ} \emph{(Reduced $\alpha$-geodesic equations)} 
\begin{enumerate} 
\item[1.] 
Each contrast function $D^{(\alpha)}$ induces\footnote{Recall the formulae \eqref{eq:metric_from_divergence} from above.} on $\mathfrak{Dens}(\mathbb{T})$ 
the homogeneous $\dot{H}^1$ Sobolev metric and an affine connection $\nabla^{(\alpha)}$ 
whose Christoffel symbols are given by 
\begin{equation} \label{eq:AC-Christ} 
\Gamma^{(\alpha)}_\xi (W,V) 
= 
- \frac{1+\alpha}{2} \Big\{ 
A^{-1} \partial_x \Big( (V\circ\xi^{-1})_x (W\circ\xi^{-1})_x \Big) \Big\}\circ\xi 
\end{equation} 
where $-1\leq \alpha \leq 1$. 
\item[2.] 
For any $\alpha$ the connections $\nabla^{(\alpha)}$ and $\nabla^{(-\alpha)}$ are dual\footnote{Recall the formula \eqref{eq:dual-con} from above.} 
with respect to the right-invariant Sobolev $\dot{H}^1$ metric given at the identity by 
$$ 
\langle u, v \rangle_{\dot{H}^1} 
= 
\frac{1}{4} \int_{\mathbb{T}} u_x v_x dx. 
$$ 
$\nabla^{(0)}$ is the corresponding self-dual Levi-Civita connection. 
\item[3.] 
The geodesic equations of $\nabla^{(\alpha)}$ on $\mathfrak{Dens}(\mathbb{T})$ 
correspond to the generalized Proudman-Johnson equations 
\begin{equation} \label{eq:PJ} 
\partial_t\partial_x^2 u + (2-\alpha) \partial_x u \partial_{xx}u + u \partial_{x}^3 u = 0. 
\end{equation} 
In particular, the case $\alpha = 0$ yields the completely integrable Hunter-Saxton equation 
\begin{equation*} 
u_{txx} + 2u_x u_{xx} + uu_{xxx} = 0 \,,
\end{equation*} 
while the case $\alpha = -1$ yields the completely integrable $\mu$-Burgers equation 
$$ 
u_{txx} + 3u_x u_{xx} + uu_{xxx} = 0. 
$$ 
(The equation corresponding to $\alpha = 1$ is also integrable in that its solutions can be 
written down explicitly, see Theorem~\ref{thm:AC-exp} below.)
\end{enumerate} 
\end{theorem} 

Connections  $\nabla^{(\alpha)}$ are called {\it Amari-Chentsov $\alpha$-connections}.
\index{Amari-Chentsov connections}
\begin{proof} 
As in finite dimensions the functions $D^{(\alpha)}$ induce metrics and connections, 
see Section~\ref{subsubsec:cgwrm}. 
Assume first that $\alpha \neq \pm 1$. 
Given any vectors $V, W$ tangent at $\xi \in \mathfrak{D}(\mathbb{T})$ 
let $\xi_{s,t}$ be a two-parameter family of diffeomorphisms in $\mathfrak{D}(\mathbb{T})$ 
such that 
$\xi|_{s=t=0}=\xi$ with $\frac{\partial}{\partial s}\xi|_{s=t=0} = V$ 
and 
$\frac{\partial}{\partial t}\xi|_{s=t=0} = W$. 
Then from \eqref{eq:alfa} we have 
\begin{align} \label{eq:alfa-met} 
\langle V, W \rangle_\alpha 
&= 
-\frac{\partial}{\partial s}\big|_{s=0} \frac{\partial}{\partial t}\big|_{t=0} D^{(\alpha)}( \xi_{s,0}, \xi_{0,t}) 
\\ \nonumber 
&= 
\frac{1}{1-\alpha^2} \frac{\partial}{\partial s}\big|_{s=0} \frac{\partial}{\partial t}\big|_{t=0} 
\int_{\mathbb{T}} (\partial_x\xi_{s,0})^{\frac{1-\alpha}{2}} (\partial_x\xi_{0,t})^{\frac{1+\alpha}{2}} dx 
\\ \nonumber 
&= 
\frac{1}{4} \int_{\mathbb{T}} \partial_x V \partial_x W \, (\partial_x\xi)^{-\frac{1+\alpha}{2}} (\partial_x\xi)^{\frac{-1+\alpha}{2}} dx 
\\ \nonumber 
&= 
\frac{1}{4} \int_{\mathbb{T}} \frac{V_x W_x}{\xi_x} \, dx 
= 
\langle V, W \rangle_{\dot{H}^1}. 
\end{align} 
Suppose that $W$ is a vector field on $\mathfrak{D}(\mathbb{T})$ defined in some neighbourhood of $\xi$. 
Let $\xi_{s,t,r}$ be a three-parameter family of diffeomorphisms such that 
$\xi|_{s=t=r=0} = \xi$ 
with 
$\frac{\partial}{\partial s} \xi|_{s=t=r=0} = V$, 
$\frac{\partial}{\partial r}\xi|_{s=t=r=0} = Z$ 
and 
$\frac{\partial}{\partial t} \xi_s|_{t=r=0} = W_{\xi_{s,0,0}}$ 
for sufficiently small $s$. 
Now, using \eqref{eq:alfa} and \eqref{eq:alfa-met} we compute 
\begin{align*} 
\langle \nabla^{(\alpha)}_V W, Z \rangle_\alpha 
&= 
\frac{1}{4} \int_{\mathbb{T}} \frac{(\nabla^{(\alpha)}_V W)_x Z_x }{\xi_x} \, dx 
\\ 
&= 
- \frac{\partial}{\partial s}\big|_{s=0} \frac{\partial}{\partial t}\big|_{t=0} \frac{\partial}{\partial r}\big|_{r=0} 
D^{(\alpha)} (\xi_{s,t,0}, \xi_{0,0,r}) 
\\ 
&= 
\frac{1}{1-\alpha^2} \frac{\partial}{\partial s}\big|_{s=0} \frac{\partial}{\partial t}\big|_{t=0} \frac{\partial}{\partial r}\big|_{r=0} 
\int_{\mathbb{T}} (\partial_x\xi_{s,t,0})^{\frac{1-\alpha}{2}} (\partial_x \xi_{0,0,r})^{\frac{1+\alpha}{2}}  dx 
\\ 
&= 
\frac{1}{4} \int_{\mathbb{T}} \bigg\{ 
\Big( (dW{\cdot}V)\circ\xi^{-1}\Big)_x 
- 
\frac{1+\alpha}{2} \big(V\circ\xi^{-1}\big)_x \big(W\circ\xi^{-1}\big)_x 
\bigg\} \big(Z\circ\xi^{-1}\big)_x dx \,.
\end{align*} 
Now integrating by parts and using the fact that $Z$ is arbitrary, we obtain 
$$ 
\big( \nabla^{(\alpha)}_V W \big)_\xi 
= 
\big( dW{\cdot}V \big)(\xi) - \Gamma^{\alpha}_\xi(W, V) 
$$ 
where the Christoffel map is given by the formula \eqref{eq:AC-Christ}. 

The computations in the remaining two cases are analogous 
and for $\alpha = -1$ and $\alpha = 1$ yield 
\begin{align} \label{eq:Ummagumma-1} 
\Gamma^{(-1)}_\xi(W,V) 
&= 0 \,,
\\ \label{eq:Ummagumma1} 
\Gamma^{(1)}_\xi (W,V) 
&= 
-A^{-1} \partial_x \Big( (V\circ\xi^{-1})_x (W\circ\xi^{-1})_x \Big) \circ\xi \,,
\end{align} 
which establishes the first part of the theorem. 

To establish the second part we need to verify that for any vector fields $V, W$ and $Z$ 
on $\mathrm{Rot}(\mathbb{T})\backslash\mathfrak{D}(\mathbb{T})$ we have 
\begin{equation} \label{eq:alfa-dual} 
V {\cdot} \langle W,Z \rangle_{\dot{H}^1} 
= 
\langle \nabla^{(\alpha)}_V W, Z \rangle_{\dot{H}^1} 
+ 
\langle W, \nabla^{(\alpha)}_V Z \rangle_{\dot{H}^1}. 
\end{equation} 
This is done by a direct calculation as above. Alternatively, it can be deduced from general properties of 
contrast functions of the type \eqref{eq:alfa} and \eqref{eq:alfalfa} 
as discussed e.g. in Chapter 3 of \cite{AmariNagaoka00}. 
The fact that $\nabla^{(0)}$ is a Levi-Civita connection of the $\dot{H}^1$-metric 
follows at once from \eqref{eq:alfa-dual}. 

The equation for geodesics of $\nabla^{(\alpha)}$ on $\mathfrak{Dens}(\mathbb{T})$ 
has the form 
\begin{equation} \label{eq:gg} 
\frac{d^2 \gamma}{dt^2} 
= 
\Gamma^{(\alpha)}_\gamma\Big( \frac{d\gamma}{dt}, \frac{d\gamma}{dt} \Big). 
\end{equation} 
Let $d\gamma/dt = u\circ\gamma$ where $u$ is a time-dependent vector field on $\mathbb{T}$ 
(i.e., a periodic function vanishing at $x=0$). 
Differentiating in the time variable and substituting into \eqref{eq:gg} we obtain 
the corresponding nonlinear PDE 
$$ 
u_t + uu_x 
= 
-\frac{1+\alpha}{2} A^{-1}\partial_x(u_x^2) \,,
$$ 
which we can rewrite as 
$$ 
-u_{txx} - 3u_x u_{xx} - uu_{xxx} = -(1+\alpha)u_x u_{xx} \,,
$$ 
which is precisely \eqref{eq:PJ}. 
\end{proof} 
\begin{remark} 
The Hunter-Saxton equation \eqref{eq:HS} (cf. Theorem~\ref{thm:AC-PJ}, Part~3) 
can be alternatively derived by observing that it is the Euler-Arnold equation of $\nabla^{(0)}$ 
on tangent space to $\mathfrak{Dens}(\mathbb{T})$ at the identity map 
and as such it is obtained from the geodesic equation of the right-invariant $\dot{H}^1$-metric  
by a standard procedure, see \cite{KhesinMisiolek03}. 
\end{remark} 
\begin{remark}[$\alpha$-curvature] 
Using the Christoffel symbols in \eqref{eq:AC-Christ} it is possible to calculate the curvature 
of the $\alpha$-connections. 
It turns out to be proportional to the curvature of the $\dot{H}^1$ metric, 
i.e.,  
\begin{equation} \label{eq:alfa-curv} 
\mathcal{R}^{(\alpha)}(V, W)Z 
= 
(1 - \alpha^2) \Big( \langle V{\cdot}\langle W, Z \rangle_{\dot{H}^1} 
+ 
W{\cdot} \langle V, Z \rangle_{\dot{H}^1} \Big) 
\end{equation} 
for any vector fields $V, W$ and $Z$ on $\mathfrak{Dens}(\mathbb{T})$. 
This formula can be computed as in finite dimensions, 
see \cite{ChentsovMorozova91} where a different choice of parameters is made. 
\end{remark} 

As already mentioned, it turns out that the geodesic equation corresponding to $\alpha =1$ 
can be integrated as well. This is done indirectly by constructing affine coordinates for $\nabla^{(1)}$. 
Observe that from \eqref{eq:alfa-curv} we already know that the connections 
$\nabla^{(-1)}$ and $\nabla^{(1)}$ are flat. 
In the former case this is also evident from \eqref{eq:Ummagumma-1}. 
\begin{theorem}{\rm (\cite{LenellsMisiolek13})}  \label{thm:AC-exp} 
The geodesic equations of $\nabla^{(1)}$ corresponding to the Euler-Arnold equation 
\begin{equation} \label{eq:PJ1} 
u_{txx} + u_x u_{xx} + uu_{xxx} = 0 
\end{equation} 
is integrable with solutions given explicitly by 
\begin{equation} \label{eq:PJ1-sol} 
u = \frac{d\xi}{dt} \circ \xi^{-1} 
\qquad 
\text{where} 
\quad 
\xi(t,x) 
= 
\frac{ \int_0^x e^{a(y)t + b(y)} dy }{ \int_0^1 e^{a(x)t + b(x)} dx } 
\end{equation} 
and $a$ and $b$ are smooth mean-zero functions on $\mathbb{T}$. 
\end{theorem} 
\begin{proof} 
We will construct a chart on 
$\mathfrak{Dens}(\mathbb{T}) = \mathrm{Rot}(\mathbb{T})\backslash\mathfrak{D}(\mathbb{T})$ in which 
the Christoffel symbols of $\nabla^{(1)}$ vanish. Consider the map 
\begin{equation} \label{eq:c-map} 
\xi \mapsto \varphi(\xi) 
= 
\log{\xi_x} - \int_{\mathbb{T}} \log{\xi_x} dx 
\end{equation} 
from the quotient space to the space of smooth periodic mean-zero functions. 
To determine how the Christoffel symbols transform under the change of variables 
$\xi \mapsto \tilde{\xi} = \varphi(\xi)$ we first compute 
$$ 
d_\xi \varphi (W) 
= 
\frac{W_x}{\xi_x} 
- 
\int_{\mathbb{T}} \frac{W_x}{\xi_x} dx 
$$ 
and 
$$ 
d^2_\xi \varphi (W,V) 
= 
- \frac{V_x W_x}{\xi_x^2} 
+ 
\int_{\mathbb{T}} \frac{V_x W_x}{\xi_x^2} dx 
$$ 
for $V, W \in T_\xi (\mathfrak{Dens}(\mathbb{T})$. 
Using \eqref{eq:A-inv} and \eqref{eq:Ummagumma1} with extra work we now find that 
$$ 
\widetilde{\Gamma}^{(1)}_{\varphi(\xi)} \big( d_\xi \varphi(W), d_\xi \varphi(V) \big) 
= 
d_\xi^2 \varphi (W,V) 
+ 
d_\xi\varphi \big( \Gamma^{(1)}_\xi(W,V) \big) 
= 
0 
$$ 
where $v = V\circ\xi^{-1}$ and $w = W \circ \xi^{-1}$. 

We can now construct explicit solutions of \eqref{eq:PJ1} as follows. 
Since $\widetilde{\Gamma}^{(1)} = 0$ all geodesics of $\nabla^{(1)}$ in the affine coordinates 
are straight lines 
$$ 
t \mapsto \tilde{\xi}(t,x) = a(x) t + b(x) 
\qquad {\rm for}\quad 
x \in \mathbb{T} 
$$ 
where $a$ and $b$ are smooth periodic functions of mean zero. 
To find a general solution $u(t,x)$ it now suffices 
to invert the map $\varphi$ in \eqref{eq:c-map} to obtain the flow 
$t \mapsto \xi(t) = \varphi^{-1}\tilde{\xi}(t)$ 
and then 
right-translate the velocity vector of the curve $\xi(t)$ to the tangent space at the identity 
in $\mathfrak{Dens}(\mathbb{T})$. 
This yields the explicit formulas in \eqref{eq:PJ1-sol}. 
\end{proof} 

The proof of Theorem \ref{thm:AC-exp} shows that the equation \eqref{eq:PJ1} is integrable. 
In fact, the explicit change of coordinates linearizes the flow in the same spirit as 
the formalism of the inverse scattering transform. 
For further details we refer to the paper \cite{LenellsMisiolek13}. 

\section{Hessian structure of the Fisher-Rao metric}

The Fisher-Rao metric has a Hessian structure relative to the flat Amari connection $\nabla^{(-1)}$ on $\mathfrak{Dens}(M)$.
Namely, it is recovered from the $\nabla^{(-1)}$-Hessian of the relative entropy functional $S_\mu$.

\begin{proposition}
    The $\nabla^{(-1)}$-Hessian of minus the relative entropy $-S_\mu(\varrho)$ recovers the Fisher--Rao metric. 
    \index{Fisher-Rao metric}
\end{proposition}

\begin{proof}
We compute the Hessian by differentiation along the curve $\varrho = \varrho(t)$
\begin{align*}
    -\frac{d^2}{dt^2}S_\mu(\varrho) &= \frac{d}{dt} \int_M \log\left( \frac{\varrho}{\mu} \right)\dot\varrho =
    \int_M \frac{\mu}{\varrho}\frac{\dot\varrho}{\mu}\dot\varrho + \int_M \log\left( \frac{\varrho}{\mu}\right)\ddot\varrho  \\
    &= \mathcal{FR}_\varrho(\dot\varrho,\dot\varrho) + \int_M \log\left( \frac{\varrho}{\mu}\right)\ddot\varrho  = \mathcal{FR}_\varrho(\dot\varrho,\dot\varrho) .
\end{align*}
The result now follows since $\ddot\varrho = 0$ along $\nabla^{(-1)}$-geodesics.
\end{proof}


\begin{remark}
    In principle, and in the same way, any connection $\nabla^{(\alpha)}$ gives rise to a Riemannian tensor 
    \[
        g_\varrho(\dot\varrho,\dot\varrho) = \nabla^{(\alpha)}\nabla^{(\alpha)} S_\mu[\dot\varrho,\dot\varrho] = \mathcal{FR}_\varrho(\dot\varrho,\dot\varrho) + \int_M \log\left( \frac{\varrho}{\mu} \right) \Gamma_\varrho^{(\alpha)}(\dot\varrho,\dot\varrho) .
    \]
    However, it is neither invariant nor positive definite in general. 
    Presumably, it is positive definite for $\varrho$ in a neighborhood of $\mu$.
\end{remark}


\section{Geometry behind Amari-Chentsov connections} 
\label{sec:ACeasy} 

The geometric description of the Amari-Chentsov $\alpha$-connections $\nabla^{(\alpha)}$ was beautifully clarified in \cite{BauerLp}   \index{Amari-Chentsov connections} 
answering a question posed earlier by Paolo Gibilisco in \cite{Gibilisco}. 
Namely, let us consider the vector space of volume forms  on $M$ equipped
with the Finsler  $ L^p$-norm, $p\in (1, \infty)$. By fixing a reference density $\mu$ on $M$ one can regard it as the space $L^p(M, d\mu)$. 

We will be interested only in the ``positive quadrant" of that space,  consisting of positive volume forms of arbitrary total volume. 
Because of the  $ L^p$-norm it is convenient to think of that part as the space of {\it $1/p$-volume forms}  $\mathfrak{Vol}^{1/p}(M)$ (where the action of 
a diffeomorphism $\phi$ on $\nu\in \mathfrak{Vol}^{1/p}(M)$ produces a factor 
$\sqrt[p]{Jac(\phi)}$, rather than just  $\mathrm{Jac}(\phi)$ for a standard volume form).

Note that  this $ L^p$-norm on $\mathfrak{Vol}^{1/p}(M)$, although corresponding to a Riemannian metric only for $p=2$, actually defines a flat structure for all $p\in (1, \infty)$. Thus the geodesics in the $ L^p$-norm are straight lines. Hence in order to connect two $1/p$-volume forms $\nu_0$ and $\nu_1$, i.e. two points in $\mathfrak{Vol}^{1/p}(M)\subset L^p(M, d\mu)$, in the shortest way, one takes a segment 
$seg:=\{(1-t)\nu_0+t\nu_1~|~t\in [0,1]\}$ joining them. 
\smallskip

Next, consider the unit sphere (more precisely, its positive part) $S^\infty_{L^p}\subset \mathfrak{Vol}^{1/p}(M)$, defined by the condition $\int_M v^p=1$. Take two    $1/p$-volume forms $\nu_0$ and $\nu_1$ in $S^\infty_{L^p}$. While the segment {\it seg} lies in $\mathfrak{Vol}^{1/p}(M)$, but not on the sphere $S^\infty_{L^p}\subset \mathfrak{Vol}^{1/p}(M)$, we project it to the sphere along the radial directions. In other words, consider the arc ${arc}=S^\infty_{L^p}\cap P(\nu_0, \nu_1)$, where $P(\nu_0, \nu_1)\subset \mathfrak{Vol}^{1/p}(M)$ is the plane through the origin and the segment {\it seg} in $L^p(M, d\mu)$. Such ``projected arcs" will be regarded as the new geodesics joining points on the unit sphere $S^\infty_{L^p}$ in the space $\mathfrak{Vol}^{1/p}(M)$.

\medskip

Now we return to $\alpha$-connections on the  space $\mathfrak{Dens}(M)$ of normalized densities on $M$. First we enlarge this space to $\mathfrak{Vol}(M)\supset \mathfrak{Dens}(M)$ to include positive volume forms of all total volumes, and where  $\mathfrak{Dens}(M)$ is defined by the affine condition
$\int_M d\lambda=1$. 

Consider the $p$th root map $\Phi_p(\lambda):={\lambda}^{1/p}$,
sending any volume form $\lambda\in \mathfrak{Vol}(M)$ to its $p$th root understood as
$1/p$-volume form ${\lambda}^{1/p}\in \mathfrak{Vol}^{1/p}(M)$. This map $\Phi_p$ pulls back the $L^p$ metric from $\mathfrak{Vol}^{1/p}(M)$
to a  Finsler  $L^p$-Fisher-Rao metrics $F_p$ for 
$p \in (1, \infty)$ on the space $\mathfrak{Vol}(M)$. The latter  generalizes the classical Fisher-Rao metric $F_2$ for $p=2$ when restricted to $\mathfrak{Dens}(M)$. Consider the Chern connection (an analog of the Levi-Civita connection for a Finsler metric) for the $L^p$-Fisher-Rao metric
on  $\mathfrak{Vol}(M)$. 

\begin{theorem}{\rm (\cite{BauerLp})}
The Amari-Chentsov $\alpha$-connection $\nabla^{(\alpha)}$\index{connection} on  the space $\mathfrak{Vol}(M)$ of  positive volume forms is the Chern connection for the $L^p$-Fisher-Rao metric $F_p$ for $p=2/(1-\alpha)$ defined by the map $\Phi_p$. In particular, the Amari-Chentsov $\alpha$-geodesics on  $\mathfrak{Vol}(M)$ are pull-backs $\Phi_p^{-1}(seg)$ of straight segments 
$${\it seg}\subset\mathfrak{Vol}^{1/p}(M)$$. 

The Amari-Chentsov $\alpha$-geodesics on the space $\mathfrak{Dens}(M)$ of normalized densities on $M$ are pull-backs $\Phi_p^{-1}(arc)$ of spherical arcs ${\it arc}\subset S^\infty_{L^p}\subset \mathfrak{Vol}^{1/p}(M)$. 

In particular, the geodesic equations on $\mathfrak{Dens}(M)$ and $\mathfrak{Vol}(M)$ are integrable\index{integrable equation} for all $p\in (1,\infty)$.
\end{theorem}

This theorem is a generalization of Theorem~\ref{thm:integ} on integrability of geodesics on $\mathfrak{Dens}(M)$ for the Fisher-Rao metric, where $p=2$.
The paper \cite{BauerLp} also contains computations or curvatures and comparative numerical modeling of geodesics for various $p$'s. 

Note also that the duality $\alpha \leftrightarrow -\alpha$ discussed in Theorem~\ref{thm:AC-PJ} inducing the duality of connections $\nabla^{(\alpha)}$ and $\nabla^{(-\alpha)}$ for $-1<\alpha<1$ corresponds to the classical duality of the $L^p$ and $L^q$-spaces for  $p=2/(1-\alpha)$ and  $q=2/(1+\alpha)$, where $1/p+1/q=1$.



\appendix
%
%
%

\chapter{Banach completions of manifolds of maps}\label{Banach} 
%
Even though for our purposes it was convenient to work with $\mathcal{C}^\infty$ maps, 
most of the constructions presented in this chapter could be 
(and, in the general literature on the subject, typically are) 
carried out in the framework of Banach spaces, such as Sobolev spaces.
We describe this setup briefly and refer the reader to e.g., 
\cite{EbinMarsden70} or \cite{Omori97} for further details. 

\section{Sobolev spaces of arbitrary order} 
As before let $M$ be a closed Riemannian manifold and let $E$ be a Hermitian vector bundle 
with fiber over $M$ and connection $\nabla$. 
To define Sobolev sections of $E$ of arbitrary order $s \in \mathbb{R}$ 
it will be convenient to use the Fourier transform. 
By compactness we pick a trivializing cover by charts 
$\varphi_i : U_i \subset M \to \{ x \in \mathbb{R}^n\mid |x| \leq 1 \}$ 
where $i = 1, \dots, N$ so that 
$E\vert_{U_i} \simeq U_i \times \mathbb{C}^m$ 
with a smooth extension to a neighbourhood of each $U_i$.                                                                                                                                                                                                                                                                                    
We can further arrange things so that 
$M = \bigcup_{i=1}^N B_i ( 1/\sqrt{2} )$ 
where 
$B_i(1/\sqrt{2}) = \big\{ x \in U_i\mid |\varphi(x)| < 1/\sqrt{2} \big\}$ 
and setting 
$\psi_i = \varphi_i/(1 - |\varphi_i|^2)^{1/2}$ 
produce a coordinate cover with the property that 
$U_i \simeq \mathbb{R}^n$ 
and 
the restriction to $U_i$ of any smooth section of $E$ 
can be viewed as a function $u: \mathbb{R}^n \to \mathbb{C}^m$ 
that is bounded together with its weighted derivatives 
$x \mapsto |x^\alpha D^\alpha u(x)|$ 
for any multi-index $\alpha$. 
Let $\{ \varrho_i \}_{i=1, \dots, N}$ be a smooth partition of unity subordinate to the $U_i$'s. 
Then a section of $E$ can be written as $u = \sum_i \varrho_i u$ whose terms 
are smooth functions of compact support in the unit ball in $\mathbb{R}^n$. 
This effectively reduces the study of sections of $E$ to that of $\mathbb{C}^m$ valued Schwartz class functions 
(cf. Example~\ref{ex:Schwartz} of Section~\ref{subsec:DCFS}) 
and makes it possible to utilize the \emph{Fourier transform} on $\mathbb{R}^n$ 
$$ 
\hat{u}(\xi) = (2\pi)^{-n/2} \int_{\mathbb{R}^n} e^{-ix \cdot \xi} u(x) \, dx 
$$ 
and its inversion on $\mathcal{S}(\mathbb{R}^n)$. 
For any $s \in \mathbb{R}$ we let 
\begin{equation} \label{eq:s-Sob} 
\| u \|^2_{H^s} 
= 
\int_{\mathbb{R}^n} \big( 1 + |\xi|^2 \big)^s |\hat{u}(\xi)|^2 d\xi 
\end{equation} 
and define the \emph{Sobolev space} of $H^s$ sections of $E$ 
to be the completion of $\mathcal{S}(\mathbb{R}^n)$ in the norm \eqref{eq:s-Sob}. 
Clearly, this construction generalizes the spaces of Sobolev sections introduced in Example \ref{ex:F-sections} 
for any integer $s=k$ with norms given by \eqref{eq:Sob-norm}.

The importance of Sobolev spaces lies in the fact that they account for the ``size" of functions 
in terms of both the ``mass" and ``height" of their derivatives in a way that relates the two. 
This is borne out by the following result. 
\begin{theorem}[Sobolev Lemma] 
If $s>n/2 +k$ then any $u \in H^s(E)$ can be modified (on a set of $\mu$-measure zero) 
to a function of class $\mathcal{C}^k$. Furthermore, we have 
$$ 
\| u \|_{\mathcal{C}^k} \leq C_{n,s} \| u \|_{H^s} 
$$ 
for some constant $C_{n,s}$ depending only on $s$ and $n$. 
\end{theorem} 
\begin{proof} 
Observe that the function $\xi \mapsto (1 + |\xi|^2)^{-s}$ is integrable on $\mathbb{R}^n$ whenever $s>n/2$. 
For any $u \in \mathcal{S}(\mathbb{R}^n)$ using the Fourier inversion formula 
and applying the Cauchy-Schwartz inequality we find 
$$ 
|u(x)|^2 
\leq 
(2\pi)^{-n} \int_{\mathbb{R}^n} \big( 1 + |\xi|^2 \big)^{-s} d\xi 
\int_{\mathbb{R}^n} \big( 1+|\xi|^2 \big)^s |\hat{u}(\xi)|^2 d\xi 
\leq 
C_{n,s} \| u \|_{H^s}^2 
$$ 
for any $x \in \mathbb{R}^n$ which settles the case $k=0$. 
Repeating this argument for each derivative $D^\alpha u$ with $|\alpha| < s - n/2$ 
yields the result. 
\end{proof} 
%

\section{Sobolev manifolds of maps}
The set $H^s(M,M)$ is now defined as consisting of maps\index{Banach manifold} 
$f$ of $M$ into itself such that for any $x \in M$ there are local charts $(U,\varphi)$ at $x$ 
and $(V, \psi)$ at $f(x)$ for which the composition 
$\psi\circ f \circ \varphi^{-1}$ belongs to $H^s(\phi(U), \mathbb{R}^n)$. 
If $s>n/2$ then using the Sobolev lemma one shows that 
this definition is independent of the choice of charts on $M$. 
The tangent space at $f \in H^s(M,M)$ is the set of all $H^s$ cross-sections of 
the pull-back bundle $T_fH^s(M,M) = H^s(f^{-1}TM)$ 
and is used as the model space. 

A differentiable atlas for $H^s(M,M)$ can be constructed using the Riemannian exponential map on $M$. 
For example, to find a chart at the identity $f=e$ we may consider the map 
$\mathrm{Exp}: TM \to M \times M$ 
given by 
$$ 
v \mapsto \mathrm{Exp}(v) = \big( \pi(v), \exp_{\pi(v)}v_{\pi(v)} \big) 
$$ 
where $\pi : TM \to M$ is the tangent bundle projection. 
$\mathrm{Exp}$ is clearly a diffeomorphism from a neighbourhood $U$ 
of the zero section onto a neighbourhood of the diagonal in $M \times M$. 
Using this map one defines a bijection from the set 
$$ 
\mathcal{U}_e = \left\{ v \in H^s(TM)\mid ~ v(M) \subset U \right\} 
$$ 
onto a neighbourhood of the identity map in $H^s(M,M)$ by 
$$ 
\Psi : T_e H^s(M,M) \supset \mathcal{U}_e   \to H^s(M,M) 
\qquad 
v \mapsto \Psi(v) = \mathrm{Exp} \circ v. 
$$ 
One checks now that the pair $(\mathcal{U}_e, \Psi)$ defines a chart in $H^s(M,M)$ around $e$. 
Compactness, properties of Riemannian exponential maps 
and standard facts about compositions of Sobolev functions and diffeomorphisms 
ensure that the charts are well-defined and independent of the metric on $M$ 
and their transition functions are smooth on the overlaps. 

For any $s>n/2 +1$ the group of $H^s$ diffeomorphisms of $M$ is now defined as 
$$ 
\mathfrak{D}^s(M) = \mathfrak{D}^1(M) \cap H^s(M,M), 
$$ 
where $\mathfrak{D}^1(M)$ is the set of $\mathcal{C}^1$ diffeomorphisms of $M$. 
Since $\mathfrak{D}^1(M)$ forms an open set in $\mathcal{C}^1(M,M)$, it follows that 
$\mathfrak{D}^s(M)$ is also open as a subset of the Hilbert manifold $H^s(M,M)$ (Sobolev lemma) 
and hence a smooth manifold. 
Furthermore, it is a topological group under composition of diffeomorphisms. 
In fact, right multiplications $R_\eta(\xi) = \xi\circ\eta$ are smooth in the $H^s$ topology, 
but left multiplications $L_\eta(\xi) = \eta\circ\xi$ and inversions $\mathrm{i}(\xi) = \xi^{-1}$ 
are continuous, but not Lipschitz continuous. 
\begin{example} 
Let 
$f(x) = x + ah(x)$ and $g(x) = x + bh(x)$, 
where $-1 \leq x \leq 1$, $0 < b=2a \ll 1$ and $h(x)$ is smooth function of compact support in $(-1,1)$ 
such that $h(0)=0$ and $\| h \|_{H^2} \leq 1$. 
Clearly, both $f$ and $g$ are increasing functions. 

Furthermore, on the one hand we have 
$$ 
\| f - g \|_{H^2}  =  \|a h \|_{H^2} \leq a. 
$$ 
On the other hand 
\begin{align*} 
\| f^{-1} - g^{-1} \|_{H^2} 
&\geq 
\big| (f^{-1})'(0) - (g^{-1})'(0) \big| 
\\ 
&= 
\Big| \frac{1}{1+ah'(0)} - \frac{1}{1+bh'(0)} \Big| 
= 
a^{-1} 
\frac{|1-a|}{|1-2a|}. 
\end{align*} 
Letting now $a \searrow 0$ we find that the functions converge in $H^2$, 
while the norm of their inverses blows up. 
\end{example} 

The subgroup of volume-preserving $H^s$ diffeomorphisms 
$$ 
\mathfrak{D}^s_\mu(M) 
= 
\big\{ \eta \in \mathfrak{D}(M)\mid \eta^\ast\mu = \mu \big\} 
$$ 
is a closed $\mathcal{C}^\infty$ submanifold of $\mathfrak{D}^s(M)$. 
This follows directly from the implicit function theorem for Banach manifolds 
and the Hodge decomposition. 
%


\backmatter
\bibliographystyle{amsalpha}



%
%
%
%
%

\printindex
%

\end{document}